\documentclass{amsart}
\usepackage{amssymb}
\usepackage{braket}
\usepackage{mathrsfs}
\usepackage{ifthen}
\usepackage{here}
\usepackage{todonotes}
\usepackage{tikz}
\usetikzlibrary{patterns}
\usepackage{comment} 
\usepackage{mleftright}
\usepackage{mathrsfs}
\usepackage{cleveref}
 \usepackage[all]{xy}
 \usepackage{amscd}
 \usepackage{amsrefs}
 \usepackage{color}
 \usepackage{enumitem}
\newlist{steps}{enumerate}{1}
\setlist[steps, 1]{label = Step \arabic*:}
\usepackage[abs]{overpic}
\usepackage[latin1]{inputenc}
\usepackage{tikz-cd}
\usepackage{marginnote}

\makeatletter
\DeclareRobustCommand\widecheck[1]{{\mathpalette\@widecheck{#1}}}
\def\@widecheck#1#2{%
   \setbox\z@\hbox{\m@th$#1#2$}%
   \setbox\tw@\hbox{\m@th$#1%
      {%
         \vrule\@width\z@\@height\ht\z@
         \vrule\@height\z@\@width\wd\z@}$}%
   \dp\tw@-\ht\z@
   \@tempdima\ht\z@ \advance\@tempdima2\ht\tw@ \divide\@tempdima\thr@@
   \setbox\tw@\hbox{%
      \raise\@tempdima\hbox{\scalebox{1}[-1]{\lower\@tempdima\box\tw@}}}%
   {\ooalign{\box\tw@ \cr \box\z@}}}
\makeatother

\theoremstyle{plain}
\newtheorem{thm}{Theorem}[section]
\crefname{thm}{Theorem}{Theorems}
\Crefname{thm}{Theorem}{Theorems}
\newtheorem{prop}[thm]{Proposition}
\crefname{prop}{Proposition}{Propositions}
\Crefname{prop}{Proposition}{Propositions}
\newtheorem{lem}[thm]{Lemma}
\crefname{lem}{Lemma}{Lemmas}
\Crefname{lem}{Lemma}{Lemmas}
\newtheorem{cor}[thm]{Corollary}
\crefname{cor}{Corollary}{Corollaries}
\Crefname{cor}{Corollary}{Corollaries}
\newtheorem {claim}[thm]{Claim}
\crefname{claim}{Claim}{Claims}
\Crefname{claim}{Claim}{Claims}

\crefname{property}{Property}{Properties}
\Crefname{property}{Property}{Properties}
\newtheorem{problem}[thm]{Problem}
\crefname{problem}{Problem}{Problems}
\Crefname{problem}{Problem}{Problems}

\crefname{conjecture}{Conjecture}{Conjecture}
\Crefname{conjecture}{Conjecture}{Conjecture}

\theoremstyle{definition}
\newtheorem{defn}[thm]{Definition}
\crefname{defn}{Definition}{Definitions}
\Crefname{defn}{Definition}{Definitions}

\crefname{notation}{Notation}{Notations}
\Crefname{notation}{Notation}{Notations}

\crefname{convention}{Convention}{Conventions}
\Crefname{convention}{Convention}{Conventions}

\crefname{cond}{Condition}{Conditions}
\Crefname{cond}{Condition}{Conditions}

\crefname{assum}{Assumption}{Assumptions}
\Crefname{assum}{Assumption}{Assumptions}

\crefname{conj}{Conjecture}{Conjectures}
\Crefname{conj}{Conjecture}{Conjectures}

\crefname{claim1}{Claim}{Claims}
\Crefname{claim1}{Claim}{Claims}
\newtheorem{ques}[thm]{Question}
\Crefname{ques}{Question}{Question}
\crefname{ques}{Question}{Question}

\theoremstyle{remark}
\newtheorem{rem}[thm]{Remark}
\crefname{rem}{Remark}{Remarks}
\Crefname{rem}{Remark}{Remarks}
\newtheorem{ex}[thm]{Example}
\crefname{ex}{Example}{Examples}
\Crefname{ex}{Example}{Examples}

\crefname{section}{Section}{Sections}
\Crefname{section}{Section}{Sections}
\crefname{subsection}{Subsection}{Subsections}
\Crefname{subsection}{Subsection}{Subsections}
\crefname{figure}{Figure}{Figures}
\Crefname{figure}{Figure}{Figures}

\newtheorem*{acknowledgement}{Acknowledgement}

\newcommand{\spinc}{\text{spin}^c}
\newcommand{\Z}{\mathbb{Z}}

\newcommand{\Q}{\mathbb{Q}}
\newcommand{\CP}{\mathbb{CP}}

\newcommand{\pt}{\mathrm{pt}}

\newcommand{\fraks}{\mathfrak{s}}
\newcommand{\frakt}{\mathfrak{t}}

\newcommand{\Diff}{\mathrm{Diff}}
\newcommand{\Homeo}{\mathrm{Homeo}}

\newcommand{\del}{\partial}

\newcommand{\Coker}{\mathop{\mathrm{Coker}}\nolimits}

\newcommand{\rank}{\mathop{\mathrm{rank}}\nolimits}

\newcommand{\id}{\mathrm{id}}

\newcommand{\C}{\mathbb{C}}
\newcommand{\s}{\mathfrak{s}}

\newcommand{\R}{\mathbb R}

\newcommand{\F}{\mathbb{F}_2}

\def\ker{\operatorname{Ker}}

\def\dim{\operatorname{dim}}
\def\rank{\operatorname{rank}}

\def\Th{\operatorname{Th}}

\def\id{\operatorname{id}}

\newcommand{\mbar}[1]{{\ooalign{\hfil#1\hfil\crcr\raise.167ex\hbox{--}}}}

\def\wt{\widetilde}

     \RequirePackage{rotating}                   
    \def\HMt{%
       \setbox0=\hbox{$\widehat{\mathit{HM}}$}
       \setbox1=\hbox{$\mathit{HM}$}
       \dimen0=1.1\ht0
       \advance\dimen0 by 1.17\ht1
       \smash{\mskip2mu\raise\dimen0\rlap{%
          \begin{turn}{180}
              {$\widehat{\phantom{\mathit{HM}}}$}
           \end{turn}} \mskip-2mu    
                \mathit{HM}
                    }{\vphantom{\widehat{\mathit{HM}}}}{}}
\newcommand{\blue}[1]{\textcolor{blue}{#1}}

\title{From diffeomorphisms to exotic phenomena in small 4-manifolds}
\author{Hokuto Konno}
\address{Graduate School of Mathematical Sciences, the University of Tokyo, 3-8-1 Komaba, Meguro, Tokyo 153-8914, Japan}
\email{konno@ms.u-tokyo.ac.jp}

\author{Abhishek Mallick}
\address{Department of Mathematics, Rutgers University, Hill Center, Busch Campus, 110 Frelinghuysen Road
Piscataway, NJ 08854, USA}
\email{abhishek.mallick@rutgers.edu}

\author{Masaki Taniguchi} 
\address{Department of Mathematics, Graduate School of Science, Kyoto University, Kitashirakawa Oiwake-cho, Sakyo-ku, Kyoto 606-8502, Japan}
\email{taniguchi.masaki.7m@kyoto-u.ac.jp}

\begin{document}

\maketitle

\begin{abstract}
We provide an approach to study exotic phenomena in relatively small 4-manifolds that captures many different exotic behaviors under one umbrella. These phenomena include exotic smooth structures on 4-manifolds with $b_2=1$, examples of strong corks, and exotic codimension-$1$ embeddings into $\C P^2 \# - \C P^2$ that survive external stabilization. 
We also give a new way to detect a homeomorphism of a 4-manifold that is not topologically isotopic to any diffeomorphism and give lower bounds of relative genera of certain knots. Our primary tools are constraints on diffeomorphisms of 4-manifolds obtained from families Seiberg--Witten theory.

\end{abstract}

\tableofcontents

\section{Introduction}
The study of exotic behavior, that is, the distinction between the smooth and topological category on $4$-manifolds is a central topic in low-dimensional topology. Typically, detecting the occurrence of exotic phenomena on compact orientable small 4-manifolds is a more challenging topic compared to $4$-manifolds with large Euler characteristic. The primary tool for studying such exotic behavior has been various $4$-manifold invariants such as the Donaldson polynomial invariant \cite{Do90}, Seiberg--Witten invariant \cite{witten1994monopoles} and Heegaard Floer mixed invariant \cite{OS_4manifolds}. Using such $4$-manifold invariants,  many authors have produced small exotic $4$-manifolds starting with the work of Akbulut \cite{Ak91_cork, Ak91}. More recently, new tools such as {\it involutive variants of 
Floer homology} and {\it the families gauge theory} have proven to be useful in the pursuit of various exotic behaviors. 
The purpose of this article is to combine and contrast these latter-mentioned tools with an eye towards application to exotic phenomenons in $4$-dimension. 

\begin{figure}[h!]
\center
\includegraphics[scale=0.7]{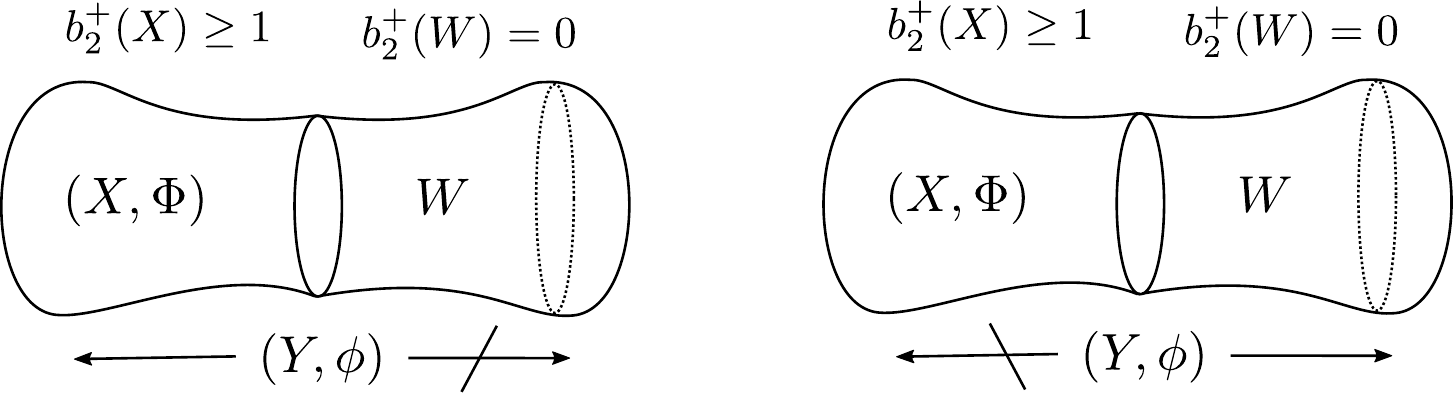}
\caption{A schematic description of the slogan of this article. The dashed part of $W$ indicates that $W$ could have another boundary component. The arrows indicate (non)-extension of $\phi$.}\label{intro_slogan}
\end{figure}
\noindent
There is an overarching slogan in our approach. This can loosely be described as follows. Suppose a $4$-manifold is decomposed as two separate $4$-manifolds, $X$ with $b^{+}(X) \geq 1$ and $W$ with $b^{+}(W)=0$, which are glued along an integer homology sphere $Y$. Now if there is a diffeomorphism $\phi$ on $Y$, which smoothly extends over $X$, (satisfying certain conditions) then $\phi$ \textit{cannot} smoothly extend over $W$ and vice versa, see Figure~\ref{intro_slogan}.

The results in this article can be broadly categorized into $5$ different sections. Firstly, combining families Seiberg--Witten theory \cite{KT22} and the involutive Heegaard Floer homology \cite{DHM20}, we give a recipe to cook up exotic 4-manifold with $b_2=1$ and use it to construct new examples of such manifolds. In particular, our approach does not refer to any $4$-manifold invariant. Our second application is regarding strong corks, which are fundamental objects in the study of exotic smooth structures of closed $4$-manifolds. The first example of such a strong cork was produced by  Lin, Ruberman and Saveliev in \cite{LRS18} using monopole Floer homology. Later, using Heegaard Floer homology, Dai, Hedden and the second author in \cite{DHM20} developed a Floer-theoretic invariant for strong cork detection using the action of the symmetry on Floer homology and the local equivalence formulation from \cite{HMZ}.  In this article, we show that families gauge theory detects strong corks. In contrast to the previous approaches, our method for detecting strong corks does not directly refer to any action of the symmetry on the Floer homology (nor do we use any local equivalence formulation), nevertheless, it is sufficient to recover most of the strong cork examples in the current literature \cite{DHM20}. Thirdly, we obtain results concerning the exotic embedding of $3$-manifolds in $4$-manifolds. More specifically, following up on a recent work by the first and the third author with Mukherjee~\cite{KMT22}, we show the existence of an exotic, codimension-1 embedding into $\mathbb{C}P^2 \# -\mathbb{C}P^2$ that survives external stabilization. This answers a posed question in \cite{KMT22}, and is currently the \textit{smallest} known $4$-manifold with this property. In contrast, in \cite{KMT22} existence of such exotic embedding was shown into $\#_2 (S^2\times S^2)$ and $\#_3 (\mathbb{C}P^2\# -\mathbb{C}P^2)$. 
As the fourth topic, using our formalism, we give new examples of homeomorphisms that are not isotopic to any diffeomorphisms on certain small 4-manifolds. Baraglia~\cite{Ba19} provided a necessary condition for homeomorphisms of such 4-manifolds to be isotopic to diffeomorphisms. Our result shows that his constraint is \textit{not} a sufficient condition.
Lastly, we give lower bounds of relative genera of certain knots in 4-manifolds with $b^+  \leq 2$. This generalizes Bryan's equivariant 10/8 inequality \cite{Br98}. Our strategy to give genus bounds is also the use of diffeomorphisms and families Seiberg--Witten theory.

In most of the applications mentioned above, we prove and repeatedly use some version of the slogan. Let us now expand on our results:



\subsection{Exotic 4-manifolds with $b_2=1$} 
Two $4$-manifolds (with or without boundary) are said to be exotic copies of each other if they are homeomorphic but not diffeomorphic. We focus on compact (orientable) 4-manifolds with $b_2=1$. For such 4-manifolds, exotic structures have been mainly discovered by making use of 4-manifold invariants. The process requires a pair of embeddings of exotic 4-manifolds into another known exotic, closed 4-manifolds which are distinguished by either one of the $4$-manifold invariants, or by the adjunction inequality, for example, see \cite{Ak91, AM97,Ak99a,AY10, AY13, KS13,Y15}. 
 

In this article, we give a criterion for constructing exotic $4$-manifolds (with boundary) such that it has $b_2=1$. We then use this to provide many examples of exotic $4$-manifolds with $b_2=1$ that were previously not known to be exotic. Our method does not rely on the existence of (effective) embeddings and 4-manifold invariants. Instead, we use families Seiberg--Witten theory, involutive Heegaard Floer theory (and filtered instanton theory) and also Akbulut--Ruberman's technique \cite{AR16}.  

Before stating our result, we will need to establish some terminology. We will denote the instanton Fr\o yshov invariant \cite{Fr02} of an oriented homology 3-sphere $Y$ by $h(Y)$. The analog of the Fr\o yshov invariant in Heegaard Floer homology is the $d$-invariant \cite{ozsvath2003absolutely}. Both of these homology cobordism invariants are integer-valued (for integer homology spheres) and conjecturally equal. Given a knot $K$, we will be interested in the quantity $d(S^{3}_{+1}(K))$. Typically in the literature, we define
\[
V_0(K):=-\frac{1}{2} d(S^{3}_{+1}(K)).
\]
The quantity $V_0(K)$ is in fact a knot concordance invariant defined by Rasmussen (see \cite{Rasmussen}).  
 Another input for us will be symmetric knots, such as the strongly invertible knots. A knot $K$ is said to be \textit{strongly invertible} if there exists an orientation-preserving involution $\tau$ of $S^3$ which preserves the knot set-wise but reverses the orientation of the knot. Given such a strongly invertible knot $(K,\tau)$, in \cite{DMS} Dai, Stoffregen and the second author constructed a numerical invariant $\underline{V}^{\tau}_0(K)$ of the equivariant knot concordance group. We refer readers to Subsection~\ref{HF_intro} for the explicit definition of the invariant $\underline{V}^{\tau}_0(K)$. 
Lastly, for a knot $K$ in $S^3$ and a positive integer $n \geq 2$, we define  a compact, smooth, oriented, negative-definite 4-manifold $W_n(K)$ as the 4-dimensional cobordism in Figure~\ref{thm_1_fig} minus a small open neighborhood of an arc connecting a point in $S^3_{1/n}(K)$ and a point in $S^3_{1/(n-1)} (K)$. Note that we have 
\[
\partial W_n(K) =  -S^3_{1/n}(K) \# S^3_{1/(n-1)}(K) 
\]
and $b_2(W_n)=1$. It is checked in \cite{NST19} that $W_n(K)$ is simply connected. 
\begin{figure}[h!]
\center
\includegraphics[scale=1.0]{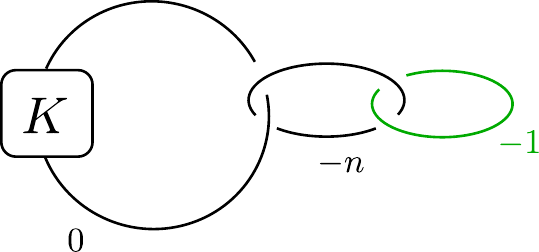}
\caption{The $(-1)$-framed $2$-handle cobordism used to define the manifold $W_n$, depicted in green.}\label{thm_1_fig}
\end{figure}
We are now in place to state the theorem:

\begin{thm}\label{main:exotic}
Let $n$ be a positive integer.
Let $K$ be a strongly invertible knot in $S^3$ satisfying the following conditions 
\[
\underline{V}^{\tau}_{0}(K) > V_0(K) > 0 \text{ and } h(S^3_1(K))<0.
\]
Then, the 4-manifold $W_{2n}(K)$ contains a pair of codimension-0 smooth compact manifolds $M_n(K)$ and $M_n'(K)$ that satisfy that:
\begin{enumerate} [label=(\alph*)]
\item $M_n(K)$ and $M_n'(K)$ are exotic, i.e. homeomorphic to each other but not diffeomorphic.
\item $M_n(K)$ and $M_n'(K)$
are homotopy equivalent to $W_{2n}(K)$. In particular, $M_n(K)$ and $M_n'(K)$ are simply connected and $b_2=1$.
\item $\del M_n(K)$ and $\del M_n'(K)$ are diffeomorphic to each other, and they are smoothly homology cobordant to $\del W_{2n}(K)$.
\end{enumerate}
Moreover, if $n \neq n'$, then $\partial M_n(K)$ is not homeomorphic to $\partial M_{n'}(K)$.
\end{thm}
\noindent
As advertised before, note that the assumptions of \cref{main:exotic} do {\it not} need the non-triviality of 4-manifold invariants such as Seiberg--Witten invariant, nor do we use it in the proof. Instead, we need to assume some conditions only on the knot. Moreover, since, in the proof of \cref{main:exotic}, we use Akbulut--Ruberman's technique \cite{AR16} which involves non-existence results of diffeomorphisms on certain 4-manifolds, we prove certain non-existence results of diffeomorphisms on 4-manifolds with $b_2=1$.
See \cref{non extension} for the details. 
\begin{rem}
Following \cite{HM}, we have the following relation between the two invariants appearing in the hypothesis of Theorem~\ref{main:exotic}
 \[
 \underline{V}^{\tau}_{0}(K) \geq V_0(K).
 \]
Hence, the condition on the hypothesis requires that the above inequality is strict. Moreover, if we assume that the conjecture that the instanton Fr\o yshov invariant coincides with and the Heegaard Floer $d$-invariant (up to multiplication by $-1/2$), is true, then the condition on $h$ and $V_0$ in the hypothesis are equivalent:
\[
V_0(K) > 0 \Leftrightarrow h(S^{3}_1(K)) < 0.
\]

\end{rem}
Although at the first glance, the hypothesis may sound restrictive, nevertheless we have infinitely many knots satisfying the assumptions. Indeed, for example we have:

\begin{ex}\label{example exotic}
Let $\tau$ be the strong involution on $T_{2,2n+1}$, the $(2,2n+1)$-torus knots described in Figure~\ref{trefoil_connect_intro}. 
    For any odd positive integer $n$ \footnote{This parameter should not be confused with the parameter $n$ from $W_n$.}, one can take a sequence of strongly invertible knots 
    \[
    (K_n, \tau_n)  = (T_{2,2n+1} \# T_{2,2n+1}, \tau \# \tau) , \ n \in \Z_{>0}
    \]
    as examples satisfying all assumptions in \cref{main:exotic}. Moreover, 
    we can also distinguish $M_2(K_n)$ and $M_2(K_m)$ for different choices of $n$ and $m$ using filtered instanton theory.  See \cref{distinguish more precisely}. 
\end{ex}
\begin{figure}[h!]
\center
\includegraphics[scale=0.5]{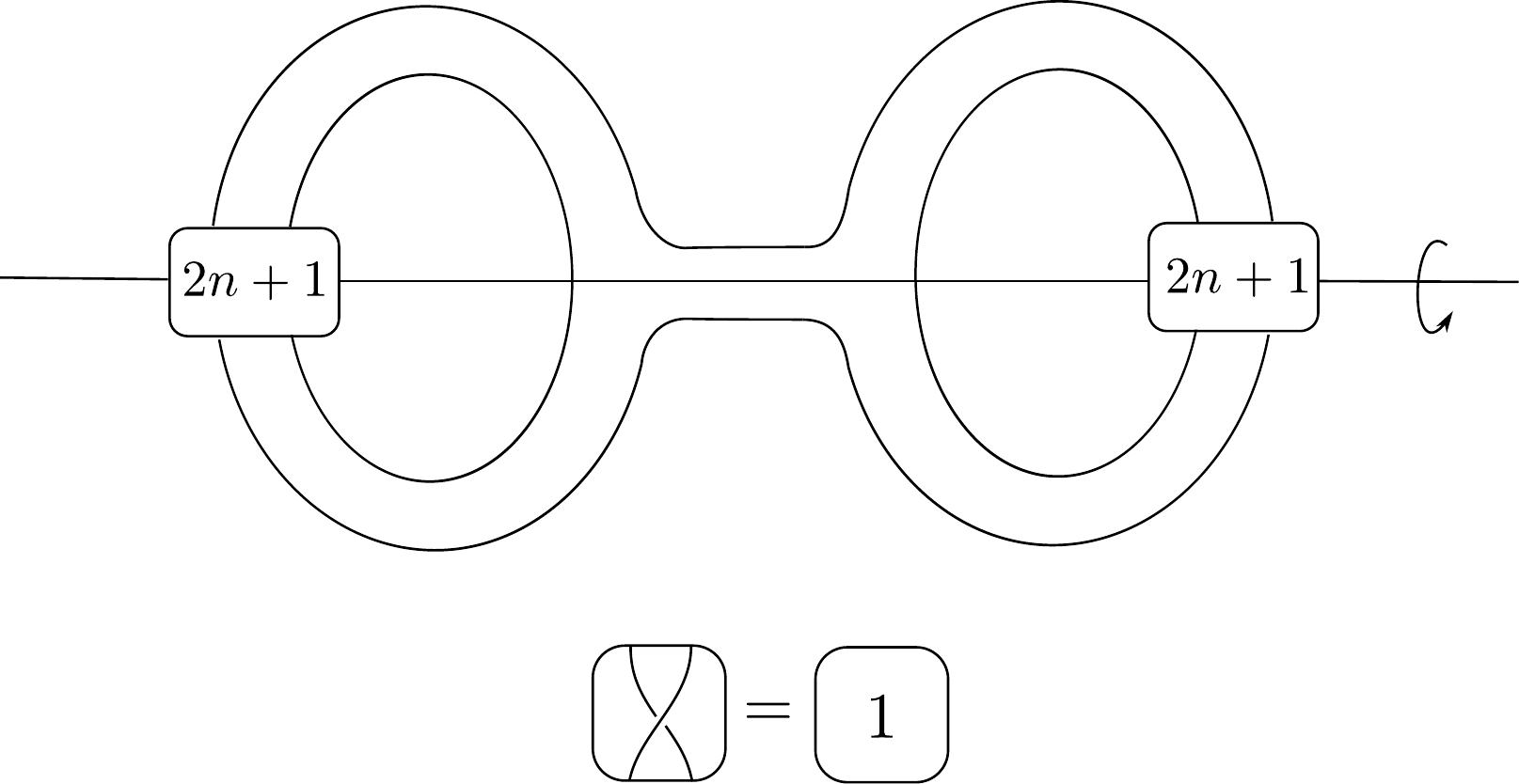}
\caption{The knot $K_n= T_{2,2n+1} \# T_{2,2n+1}$, with the  strong involution $\tau \# \tau$, induced from the strong involution $\tau$ on $T_{2,2n+1}$.}\label{trefoil_connect_intro}
\end{figure}
\noindent



While the hypothesis of Theorem~\ref{main:exotic} makes no reference to families gauge theory, the proof crucially uses it while also borrowing techniques from involutive Heegaard Floer homology  \cite{DHM20} and filtered instanton Floer homology \cite{NST19}. We refer readers to Section~\ref{proof_main_exotic} for the proof.

\subsection{Families Seiberg--Witten theory detects strong corks} 

Corks are one of the fundamental objects in the study of the exotic $4$-manifolds. Indeed, by the work of  Matveyev \cite{MAt} and Curtis-Freedman-Hsiang-Stong  \cite{CFHS}, it is well-known that any two different exotic smooth structures on a closed simply-connected $4$-manifold are related by a \textit{cork-twist}. Recently, a more generalized version of corks was introduced by Lin, Ruberman and Saveliev \cite{LRS18}, called the strong cork. We recall the definition of the (strong) cork below.
\begin{defn}
Let $(Y, \tau, W)$ be a tuple, where $\tau$ is an orientation-preserving smooth involution on an integer homology sphere $Y$ and $Y$ is a boundary of a contractible, smooth, compact manifold $W$. The triple $(Y, \tau, W)$ is said to be a \textit{cork} if the involution $\tau$ on $Y$ does \text{not} extend over $W$ as a diffeomorphism. Moreover, consider $(Y,\tau)$ as before, if $\tau$ does not extend over \textit{any} $\mathbb{Z}_2$-homology ball, $W$, as a diffeomorphism, then $(Y,\tau)$ is called a \textit{strong cork}.
\end{defn}
 \noindent
 Sometimes, we will be sloppy and call a pair $(Y,\tau)$ a strong cork, if $\tau$ does not extend over any $\mathbb{Z}_2$-homology ball, without requiring that $Y$ bounds a contractible manifold. As in the case of exotic manifold detection, most existing literature on cork-detection uses the following technique. Firstly, one finds an effective embedding of the cork inside a closed $4$-manifold and then one shows that one of the $4$-manifold invariant changes under cork-twist \cite{Ak91_cork}. On the other hand, such methods are incapable of detecting strong corks. 
An alternative to this approach was recently studied in \cite{LRS18} by Lin, Ruberman and Saveliev, who showed that the Akbulut cork is strong. Later in \cite{DHM20} the second author jointly with Dai and Hedden gave a plethora of new examples of strong corks. This was achieved by constructing an invariant coming out of the involutive Heegaard Floer homology, capable of detecting strong corks. In particular, the construction of the invariant used the action of the cork-twist symmetry on the Heegaard Floer homology of the boundary and the local equivalence formulation by Hendricks--Manolescu--Zemke \cite{HMZ}. 

In this article, we will use families Seiberg--Witten gauge theory to produce examples of strong corks. Our method of detecting strong corks is independent of that used in \cite{LRS18} and \cite{DHM20}. In particular, we do \textit{not} use any action of symmetry on the Floer homology of the boundary nor do we use the local equivalence formulation.

Before diving into the statement of the obstruction, we need to recall a few definitions. Given an oriented 4-manifold $X$, we denote by $H^+(X)$, a choice of a maximal-dimensional positive-definite subspace of $H^2(X;\R)$.
Given an orientation-preserving diffeomorphism $\Phi : X \to X$, one may define the notion that $\Phi$ preserves or reverses an orientation of $H^+(X)$, independent of the choices of  $H^+(X)$ (see \cref{subsection A constraint from families Seiberg--Witten theory} for more details).
We now make the following two definitions:
\begin{defn}
Let $X^4$ be an oriented $4$-manifold, with or without boundary,
and $\Phi$ be an orientation-preserving diffeomorphism on $X$. We say that $\Phi$ is $H^{+}$-\textit{preserving} or $H^{+}$-\textit{reversing} according to whether $\Phi$ preserves or reverses orientation of $H^{+}(X)$.
\end{defn}

\begin{defn}\label{spinc_cobordism}
Let $(X,\Phi)$ be as before. We say that $\Phi$ is $\spinc$-\textit{preserving for $\s$}, if $\s$ is a $\spinc$-structure on $X$ that is preserved by $\Phi$ i.e. $\Phi(\s)=\s$. On the other hand, we say that $\Phi$ is $\spinc$-\textit{reversing} (or conjugating) for $\s$, if $\s$ is a $\spinc$-structure on $X$ that is conjugated by $\Phi$, i.e. $\Phi(\s)= \bar{\s}$.
\end{defn}
\noindent
We are now in place to state our obstruction:
\begin{thm}\label{strong_cork}
Let $(Y , \phi)$ be an oriented homology 3-sphere with an orientation-preserving (not necessarily order 2) diffeomorphism $\phi$. Suppose that $Y$ bounds a spin$^c$ 4-manifold $(X, \mathfrak{s})$ with $b^+(X)=1$ and $b_1(X)=0$. 
Now if $\Phi$ is any orientation-preserving, smooth extension of $\phi$ so that $\Phi$ is $H^{+}$-reversing, and $\spinc$-preserving for some $\spinc$-structure $\s$ together with
\[
\frac{c_1(\fraks)^2-\sigma(X)}{8} > 0,
\]
then $(Y, \phi)$ is a strong cork. 
\end{thm}
\noindent
We state another obstruction, where the hypothesis is somewhat `conjugated' to the above:
\begin{thm}\label{strong_cork_conjugation}
Let $(Y , \phi)$ be an oriented homology 3-sphere with an orientation-preserving (not necessarily order 2) diffeomorphism $\phi$. Suppose that $Y$ bounds a spin$^c$ 4-manifold $(X, \mathfrak{s})$ with $b^+(X)=1$ and $b_1(X)=0$. 
Now if $\Phi$ is any orientation-preserving, smooth extension of $\phi$ so that $\Phi$ is $H^{+}$-preserving, and $\spinc$-reserving for some $\spinc$-structure $\s$ together with 
\[
\frac{c_1(\fraks)^2-\sigma(X)}{8} > 0,
\]
then $(Y, \phi)$ is a strong cork.
\end{thm}

\noindent
\noindent
Readers familiar with \cite{DHM20} will see later that the Theorem~\ref{strong_cork_conjugation} plays the part of $\iota \circ \tau$-local equivalence class while Theorem~\ref{strong_cork} substitute for $\tau$-local equivalence class. Indeed, \cref{strong_cork_conjugation} shall be proven by using charge conjugation symmetry in Seiberg--Witten theory, which corresponds to $\iota$-map in Heegaard Floer theory.  However, we remark that there is no obvious relation between the families Seiberg--Witten theory and involutive Heegaard Floer homology.  

While the reader may find the hypotheses for both Theorem~\ref{strong_cork} and Theorem~\ref{strong_cork_conjugation} somewhat stringent, we are still able to recover almost all of the examples of existing strong corks (from \cite{DHM20}) in the literature using them. We list them below:
\begin{thm}\cite{DHM20} \label{cork_detection} Following $3$-manifolds equipped with the specified involution are all strong cork:

\begin{enumerate}[label=(\alph*)]
  \setlength\itemsep{1em}

\item For  $k$ positive and odd, and the symmetries displayed in Figure~\ref{intro_corks}, any $1/k$-surgery on the slice knots $\overline{9}_{41}, \overline{9}_{46}, 10_{35}, \overline{10}_{75}, 10_{155}, 11_{n49}$.

\item The manifolds $(M_n,\tau)$ as displayed in Figure~\ref{intro_corks_2}, with $n>0$.

\item The manifolds $(W_n,\tau)$ as displayed in Figure~\ref{intro_corks_2}, with $n>0$ and odd.

\item  For $k$ positive and odd, and $(K_{-n,n+1},\tau, \sigma)$ as in Figure~\ref{intro_corks_3} 
\[
\begin{cases}
(S^3_{1/k}(\overline{K}_{-n, n+1}), \tau) \; \text{and} \; (S^3_{1/k}(\overline{K}_{-n, n+1}), \sigma) &\text{if } n \text{ is odd}\\
(S^3_{1/k}(K_{-n, n+1}),\tau) \; \text{and} \; (S^3_{1/k}(K_{-n, n+1}),\sigma) &\text{if } n \text{ is even}.
\end{cases}
\]

\end{enumerate}
\end{thm}

\begin{figure}[h!]
\center
\includegraphics[scale=0.2]{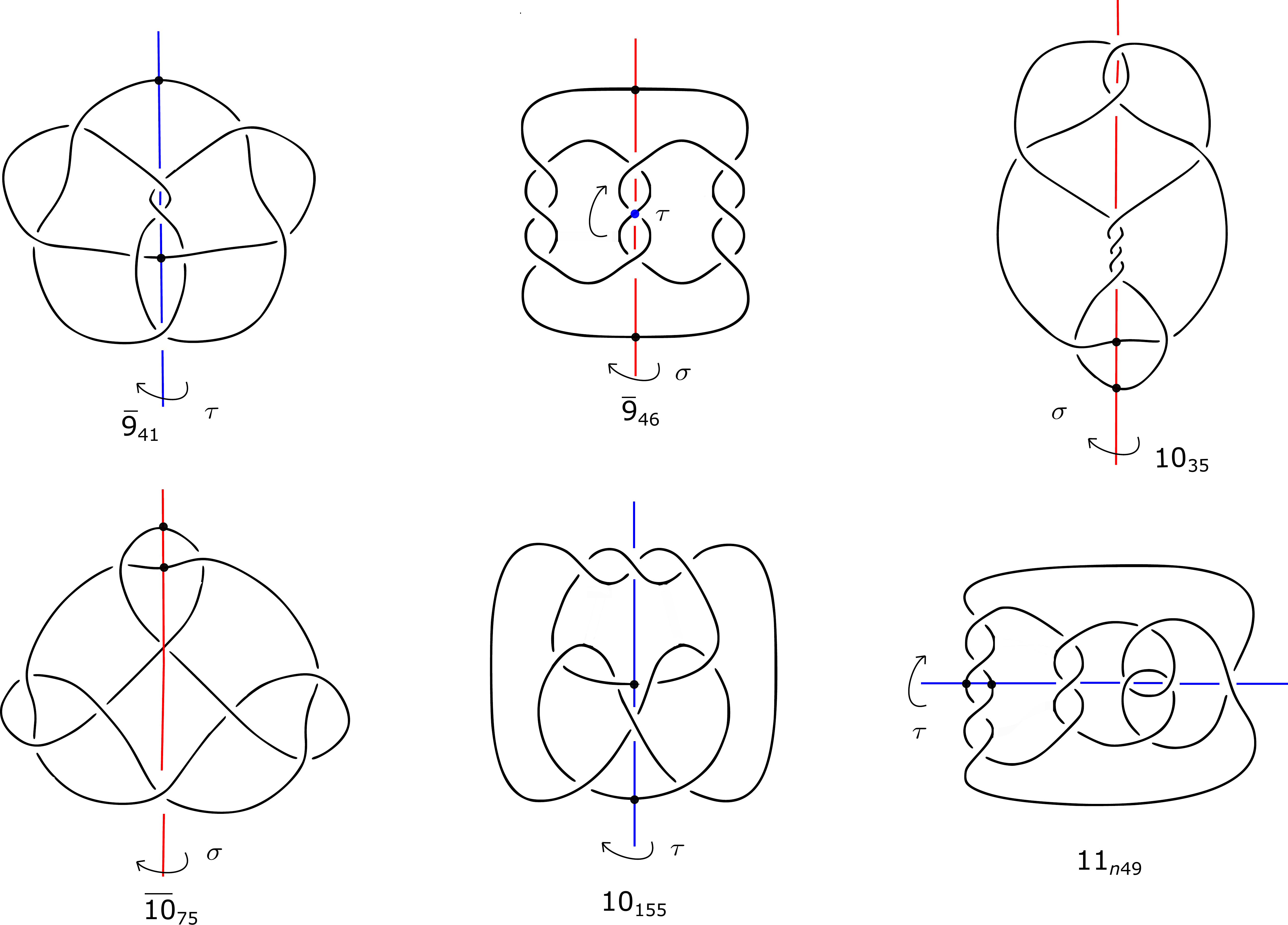}
\caption{The slice knots used in Theorem~\ref{cork_detection}. Figure credit: \cite[Figure 2]{DHM20}.}\label{intro_corks}
\end{figure}

\begin{figure}[h!]
\center
\includegraphics[scale=0.3]{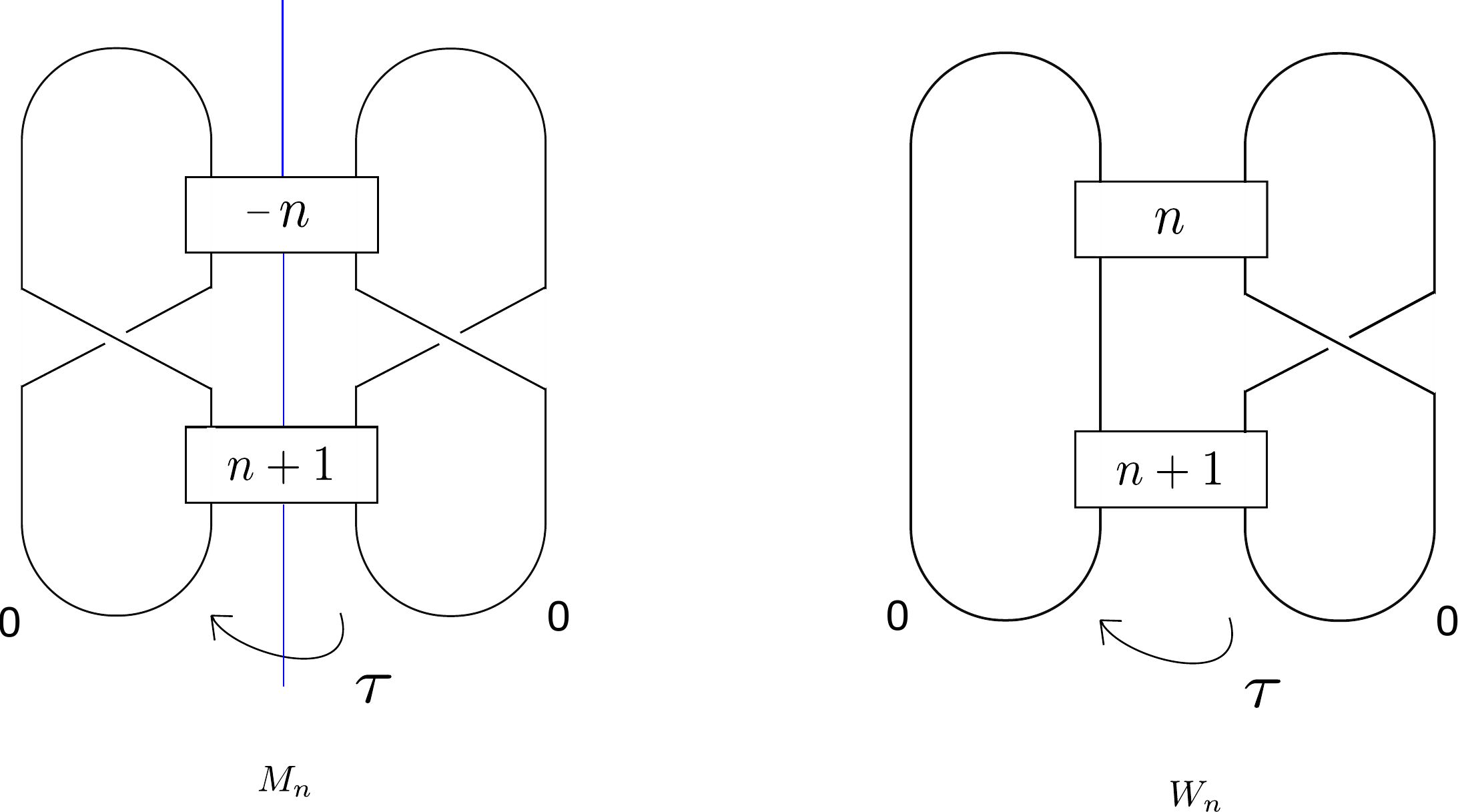}
\caption{The manifolds $W_n$ and $M_n$ used in Theorem~\ref{cork_detection}.}\label{intro_corks_2}
\end{figure}

\begin{figure}[h!]
\center
\includegraphics[scale=0.5]{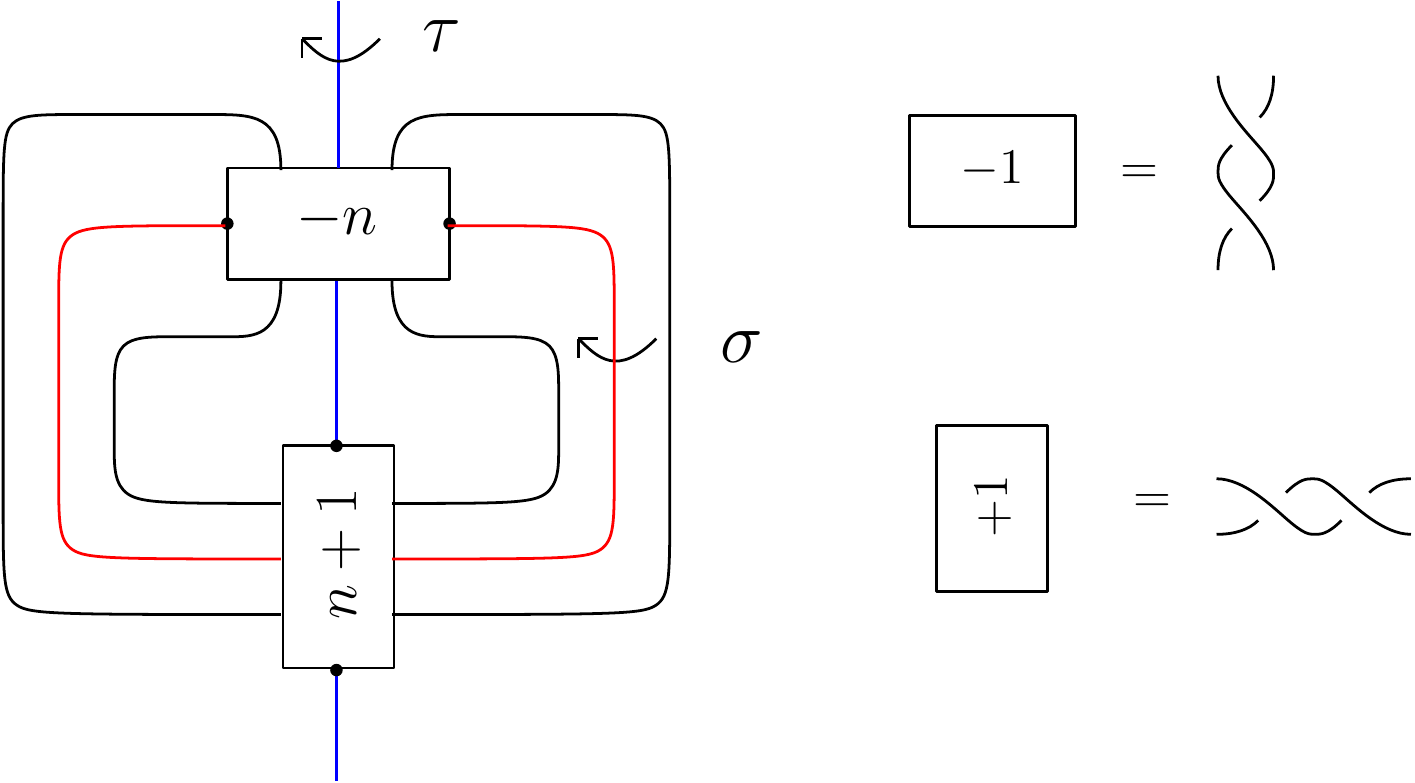}
\caption{The slice knots $K_{-n,n+1}$ used in Theorem~\ref{cork_detection}.}\label{intro_corks_3}
\end{figure}

\noindent
This recovers \cite[Theorem D]{LRS18} and \cite[Theorem 1.10, 1.11, 1.12, 1.13]{DHM20}. We again stress that we do not use the action of the symmetry on the Floer homology (although we still draw from the equivariant cobordisms constructed in \cite{DHM20} for our proof). The advantage of our method as opposed to that in \cite{DHM20} is that it is often hard to compute the action of the symmetry on the Heegaard Floer homology. Indeed, the action is known only for a handful of knots (and hence for surgery on them) \cite{DMS, mallick2022knot}. In contrast, our hypothesis concerns the intersection form of the bound and being able to extend diffeomorphisms in a certain way, which can sometimes be easily checked from the Kirby diagrams, and hence is potentially more user-friendly. However, we do not have an explicit example illuminating this. See Remark~\ref{comparison} for a description of an ad hoc situation.

\subsection{Exotic embeddings of 3-manifolds}

Recently, exotic embeddings of 3-manifolds into 4-manifolds have been studied via various tools \cite{BD19, Wa20,  IKMT22, KMT22}.
Let us first recall the definitions of exotic embeddings used in this paper: 
\newpage
\begin{defn}
\label{def: exo submfd}
Let $Y$ be a 3-manifold.
We say that two smooth embeddings $i_1,i_2 :Y\to X$ into a smooth 4-manifold $X$ are (strongly) {\it exotic} if \begin{itemize}
    \item[(i)] there is a topological ambient isotopy $H_t: X\times [0,1] \to X$ such that $H_1\circ i_1= i_2$,
    \item[(ii)] there is no such smooth isotopy,
    \item[(iii)] the complements of $Y_1$ and $Y_2$ are diffeomorphic, i.e there exists a diffeomorphism $f:X\to X$ such that $f\circ i_1 = i_2$. 
\end{itemize} 
\end{defn}

\begin{rem}
One can consider a weaker version of the notion of exotic embedding, obtained by just dropping (iii).
We can easily provide examples of exotic embeddings in this weaker sense.
For example, see \cite{KMT22}. 
\end{rem}

In \cite{KMT22}, exotic embeddings of homology 3-spheres into $\#_2 (S^2\times S^2)$ and $\#_3 (\mathbb{C}P^2\# -\mathbb{C}P^2)$ have been constructed. 
\noindent
In the same paper, the following question was posted: 

\begin{ques}[\cite{KMT22}]
Does there exist a closed 4-manifold smaller than one in \cite[Theorem 1.4/Theorem 1.7]{KMT22} that supports exotic codimension-1 embeddings?
More concretely, do $S^4$, $S^2\times S^2$, or $\C P^2\#(-\C P^2)$ admit exotic codimension-1 embeddings? 
\end{ques}
\noindent

In this paper, we provide a pair of exotic embeddings into $\mathbb{C}P^2\# -\mathbb{C}P^2$: 

\begin{thm}\label{exotic embeddings}
There is a pair of exotic (strong) embeddings of $\Sigma (2,3,7) $ into  $ \C P^2 \# -\C P^2$. Moreover, these are still exotic after taking the connected sum of any connected smooth 4-manifold, attached outside the images of the embeddings. 
\end{thm}

The argument of the proof of \cref{exotic embeddings} enables us to construct more exotic embeddings. The proof of \cref{exotic embeddings} uses families Seiberg--Witten theory. We also observe \cref{exotic embeddings} can be derived also from only using involutive Heegaard Floer theory. See \cref{inv exotic emb}. 
\begin{rem}
In \cite{KMT22}, it was pointed out that the existence of exotic embeddings of 3-manifolds with trivial mapping class groups into $S^4$ disproves the four-dimensional Smale conjecture. When 3-manifolds have trivial mapping class groups, then exotic embeddings become exotic submanifolds. 
Thus, it is interesting to ask whether one can replace the 3-manifold $\Sigma(2,3,7)$ in \cref{exotic embeddings} with 3-manifolds with trivial mapping class group. 
\end{rem}

\subsection{Homeomorphisms not isotopic to any diffeomorphisms}

Given a smooth 4-manifold $X$, it is natural to ask the connectivity of the inclusion map $i : \Diff(X) \hookrightarrow \Homeo(X)$ from the diffeomorphism group to the homeomorphism group, equipped with the $C^\infty$-topology and the $C^0$-topology respectively.
We shall focus on the problem of whether the induced map
\[
i_\ast : \pi_0(\Diff(X)) \to \pi_0(\Homeo(X))
\]
is surjective.
For $f \in \Homeo(X)$, the topological mapping class $[f]$ lies in the image of $i_\ast$ if and only if $f$ is topologically isotopic to some diffeomorphism.
Given $[f] \in \pi_0(\Homeo(X))$, we say that $[f]$ is {\it realized by  a diffeomorphism} when $[f]$ lies in the image of $i_\ast$.

The map $i_\ast$ is often surjective by Wall's theorem~\cite{Wa64},
but there exist 4-manifolds with non-surjective $i_\ast$.
The smallest 4-manifolds for which $i_\ast$ are known to non-surjective are homotopy $\C P^2\#10(-\C P^2)$.
Friedman--Morgan~\cite{FM88} proved that $i_\ast$ is not surjective for $\C P^2\#n(-\C P^2)$ if $n\geq10$, while $i_\ast$ is surjective for $n<10$ \cite{Wa64}.
More generally, if $X$ is a homotopy $\C P^2\#n(-\C P^2)$ for $n\geq10$, it follows from a result by Baraglia~\cite{Ba19} that $[f] \in \pi_0(\Homeo(X))$ is realized by a diffeormophism only if $f$ preserves orientation of $H^+(X)$.
It is natural to ask whether the converse is true.
We answer this question in the negative:

\begin{thm}
\label{thm: non-realizable homeo}
Let $X$ be a smooth 4-manifold that is homotopy equivalent to $\C P^2\#n(-\C P^2)$ for some $n \geq 10$.
Then there is a homeomorphism $f$ of $X$ such that:
\begin{itemize}
\item[(i)] $f$ preserves orientation of $H^+(X)$, and
\item[(ii)] $[f] \in \pi_0(\Homeo(X))$ is not realized by a diffeomorphism, i.e., $f$ is not topologically isotopic to any diffeomorohism.
\end{itemize}
\end{thm}

\begin{rem}
The smallest spin 4-manifolds for which $i_\ast$ are known to be non-surjective are homotopy $K3$.
This is proven by Donaldson~\cite{Do90} for the $K3$ surface, and for a general homotopy $K3$ this follows from a later result by Morgan--Szab{\'o}~\cite{MS97}.
Again the orientation of $H^+$ is crucial in the realization problem.
More precisely, 
given a homeomorphism $f$ of a homotopy $K3$, the topological mapping class $[f]$ is realized by a diffeomorphism only if $f$ preserves orientation of $H^+$.
For the $K3$ surface, it follows from results by Matumoto~\cite{Matu86} and Borcea~\cite{Bor86} that the converse is also true: if $f$ preserves orientation of $H^+(K3)$, then $[f]$ is realized by a diffeomorphism.
\end{rem}

\subsection{Relative genus bounds from diffeomorphisms}\label{intro_genus}
Studying various notions of sliceness for knots in $S^3$ is a fundamental problem from the viewpoint of the 4-dimensional aspect of knot theory. Let $X$ be a closed, oriented, connected, smooth 4-manifold and $K$ be a knot in $S^3$. We focus on the {\it $X$-slice genera for knots}, which we recall here.  
\begin{defn} Let $x$ be an element in $H_2(X;\Z) \cong H_2( X \setminus \operatorname{int} D^4, \partial; \Z)$. 
The {\it smooth (resp. topological) $X$-genus} $g_{X,x}(K)$ ($g^{\text{top}}_{X,x}(K)$) of a knot $K$ in $S^3$ for $X$ is the minimal genus of smoothly (resp. locally flatly) and properly embedded  surfaces in a punctured $X$ which are bounded by $K$ and representing $x$. 
\end{defn}

For a definite 4-manifold $X$, there are a lot of effective lower bounds of $g_X(K)$ derived from Heegaard Floer theory, Khovanov homology theory, and gauge theoretic Floer theory. For example, the Milnor conjecture claiming 
\[
g_{S^4,0}(T_{p,q}) = u(T_{p,q})= \frac{1}{2}(p-1)(p-1)
\]
has been proven in many different theories (for example, see \cite{KM93, OS03, Ra10, KM13, BH22, DISST22}), 
where $T_{p,q}$ is the positive torus knot of type $(p,q)$ for coprime integers $p$ and $q$ and $u(K)$ is the unknotting number of a knot $K$.
For indefinite 4-manifolds, there are fewer ways to give effective lower bounds of $g_{X,x}(K)$. When we assume a $4$-manifold $X$ has a symplectic structure or non-vanishing gauge theoretic invariant (under the additional assumption of $b^+ \geq 2$), several adjunction type inequalities are known \cite{MR06, MML20, IMT21, Ba22, IKT22}. 
 For indefinite 4-manifolds with vanishing gauge theoretic invariant, using 10/8 type inequalities, a few genus bounds are also known \cite{Br98, MML20, Ka22,  KMT21, Mo22}. In this paper, we use families Fr\o yshov inequality to produce lower bounds of smooth genera $g_{X}(K)$ for both definite and indefinite manifolds.

In \cite{dai20222}, using involutive Heegaard Floer theory, it was proven that the $(2,1)$-cable of the figure-eight knot is not smoothly slice. Recently, in \cite{ACMPS23}, a genus bound from Bryan's 10/8 inequality \cite{Br98} for embedded surfaces in $X= \#_2 \CP^2$ with the cohomology class $(2,6)$:  
\begin{align}\label{Bryan}
g_{\#_2 \C P^2 , (2,6)} (U ) \geq 10,
\end{align}
was used to reprove the non-sliceness of the $(2,1)$-cable of figure-eight. Using the families Seiberg--Witten theory, we will now generalize Bryan's inequality. 

In order to describe our result, it is convenient to use the following homomorphism from the knot concordance group $\mathcal{C}$:
\[
{\bf \Sigma } : \mathcal{C} \to \Theta^{\tau} _{\Z_2}, 
\]
where $\Theta^{\tau} _{\Z_2}$ is the \textit{ $\Z_2$-homology bordism group of diffeomorphisms}. This group $\Theta^{\tau} _{\Z_2}$ is defined similarly as in \cite[Definition 2.2]{DHM20}.
The map ${\bf \Sigma}$ is given as 
\[
{\bf \Sigma} ([K]) := [(\Sigma(K), \tau)], 
\]
where $\tau$ is the covering involution on the double branched covering space $\Sigma(K)$ along $K$. Note that the map ${\bf \Sigma}$ defines a homomorphism.

We generalize \eqref{Bryan} to the following: 
\begin{thm}\label{genus application}
Let $K$ be a knot whose concordance class lies in $\operatorname{Ker} {\bf \Sigma}$ with $\sigma(K)=0$, where $\sigma(K)$ denotes the knot signature. 
Then we have
\[
g_{\#_2 \C P^2 , (2,6)} (K) \geq 10. 
\]
\noindent
\end{thm}
\noindent
Note that the proof of Bryan's equivariant 10/8 inequality  \cite{Br98} and of its generalization by Montague \cite{Mo22} uses honest $\Z_{2}$-equivariant involutions (i.e. diffeomorphisms which are order $2$) on the cobordisms and equivariant $K$-theory with $G=\mathrm{Pin}(2) \times_{\Z_2} \Z_4$. In contrast, our argument works for any diffeomorphisms on $\mathbb{Z}_2$-homology cobordisms (without any restrictions on the order) and also uses $\mathrm{Pin}(2)$-equivariant cohomology theory. 
\begin{rem}
    Examples of elements in $\operatorname{Ker} {\bf \Sigma}$ are given as the following: 
   For a strongly negatively amphichiral knot $K$,  
\begin{align}\label{example wt}
   [K \# -K^r] \in \operatorname{Ker} {\bf \Sigma}
   \end{align}
    where $-K$ denotes the concordance inverse of $K$ and $K^r$ is $K$ with the reversed orientation. 
    See for example, \cite{AMMMS20} for a discussion. Note that for a strongly negatively amphichiral knot $K$, $K\# -K^r$ is a 2-torsion in the knot concordance group. This implies $\sigma(K\# -K^r)=0$. Thus all assumptions of \cref{genus application} are satisfied for such knots. As a concrete example, one can take $K=8_{17}$. It is confirmed in \cite{PC99,AMMMS20} that $8_{17} \# -8_{17}^r$ is a non-trivial element in $\operatorname{Ker} {\bf \Sigma}$.
   See \cite{AMMMS20} for more concrete examples. 
\end{rem}
As a corollary of \cref{genus application}, one can conclude $(4_1)_{(2,1)}$ does not lie in $\operatorname{Ker} {\bf \Sigma}$. 
In particular, one can see:
\begin{cor}\label{figure cork}
The double branched cover of $(2,1)$-cable of $4_1$ with the covering involution is a strong cork. 
\end{cor}

\begin{proof}
In \cite{ACMPS23}, the authors gave an upper bound $g_{\#_2 \C P^2 , (2,6)} ((4_1)_{(2,1)}) \leq 9$.
So \cref{genus application} implies that $(4_1)_{(2,1)}$ does not lie in $\operatorname{Ker} {\bf \Sigma}$, which implies the assertion from the definition of $\operatorname{Ker} {\bf \Sigma}$. 
\end{proof}

\begin{rem}
   Note that Corollary~\ref{figure cork} was already proven in \cite[Remark 2.2]{dai20222} using involutive Heegaard Floer theory. Following \cite{ACMPS23}, this gives an alternative proof of it using families Seiberg--Witten theory. 
\end{rem}

The bound in Theorem~\ref{genus application} follow from a more general genus bound established in \cref{general genus bounds}, using extendability of diffeomorphisms over $b^{+} \neq 0$ bounds.

\begin{rem}\label{main:torus}
    We also exhibit genus bounds for indefinite 4-manifolds with $b^+ \leq 2$. In Section~\ref{branched section}, 
    using our constraint, we exhibit such bounds for certain torus knots in $S^2 \times S^2$, $\#_2 S^2 \times S^2$. For example, using our constraint, we show:
    \[
    g_{S^2\times S^2, 0 } (T_{3,13}) \geq 9. 
    \]
    This inequality also can be proven by using Manolescu's relative 10/8 inequality \cite{Ma14}, the ``real" 10/8 inequality \cite{KMT21}, and Montague's equivariant 10/8-inequality \cite{Mo22}. All of these methods used equivariant K-theory, but ours uses equivariant cohomology theory. 
\end{rem}

The most general genus bound in our paper is stated in \cref{general genus bound} and is derived from bounds for \cref{general genus bounds} for $\alpha$, $\beta$, $\gamma$  and $\delta$ for the double branched covering spaces of knots. In particular, \cref{general genus bounds} gives us some interesting relations between Manolescu's $\alpha$, $\beta$, $\gamma$ invariants \cite{Ma16} and embedded surfaces in 4-manifolds with boundary. Thus, it might be useful for computations of Manolescu's $\alpha$, $\beta$, $\gamma$ invariants. 

\subsection{Questions and connections: $b^+=1$ bounds and local equivalence class.}
Throughout this article, we use $b^+ \geq 1$ bounds for a $3$-manifold to obtain obstructions to extend a diffeomorphism. Our results indicate that some of these obstructions are morally part of a bigger story that is yet to unfold. We pose them below as questions. 
Seiberg--Witten analogs of such questions will also be discussed in upcoming work by Sasahira and the first author \cite{KS23}. Below we assume that the reader is familiar with the invariants defined from \cite{DHM20}, see Subsection~\ref{HF_intro} for a review.

We begin by observing that if $W$ is a cobordism between two integer homology spheres $Y_1$ and $Y_2$, where $b^+(W) = 1$,  then the Heegaard Floer cobordism map, associated to $F_W$ vanish on the infinity flavor. However, as the readers will see in this article, the existence of certain diffeomorphism on $W$ in some cases hints at a relation between the local equivalence class of the boundaries $Y_i$. Such behaviors are reminiscent of the existence of local maps induced by the cobordism maps $F_W$ and the monotonicity relations \cite[Theorem 1.5]{DHM20}. Since the map $F_W$, in this case, is not local, we can think of the following as positing the existence of a \textit{fake local map}. 

\begin{ques}\label{1_c}
\emph{Suppose that $W$ is a spin $4$-manifold bounded by a $3$-manifold $Y$, where $b^{+}(W)=1$, $b_1(W)=0$, $\sigma(W) < 0$ and $Y$ is an  integer homology sphere. Further, assume that $X$ is equipped with an orientation-preserving diffeomorphism $f$, which restricts to an orientation-preserving diffeomorphism $\phi$ on $Y$. Now if $f$ preserve a spin structure on $W$ and reverse the orientation of $H^{+}(W)$, then does that imply the following about the local equivalence class?
\[
h_{\tau}(Y, \phi) \neq 0, \; h_{\iota \circ \tau}(Y, \iota \circ \phi) = 0.
\]}
\end{ques}
\noindent
In a similar manner, we may also pose a `conjugated' version of the above question, owing to different conditions on the extension of $f$.
\begin{ques}\label{2_c}
\emph{Suppose that $X$ is a spin $4$-manifold bounded by a $3$-manifold $Y$, where $b^{+}(W)=1$, $b_1(W)=0$, $\sigma(W)< 0$ and $Y$ is an integer homology sphere. Further, assume that $X$ is equipped with an orientation-preserving diffeomorphism $f$, which restricts to an orientation-preserving diffeomorphism $\phi$ on $Y$. Now we further assume that $f$ preserve a spin structure on $W$ and the orientation of $H^{+}(W)$, then does that imply the following about the local equivalence class
\[
h_{\tau}(Y, \phi) = 0, \; h_{\iota \circ \tau}(Y, \iota \circ \phi) \neq 0.
\]
}
\end{ques}
\noindent

The questions also point towards a connection between the involutive local equivalence class of Brieskorn homology spheres $\Sigma(p,q,r)$ and the existence of a $b^+=1$ bound. Indeed, for such Brieskorn spheres the action of $\iota$ is known to be the same as the action of complex conjugation diffeomorphism (thinking of $\Sigma(p,q,r)$ as a link of singularity) from the work of \cite{dai2019involutive} and \cite{alfieri2020connected}. Hence in the situation where the action of complex conjugation on $\Sigma(p,q,r)$ extends over a bound satisfying the conditions of Question~\ref{1_c}, we would get that the involutive local class of $\Sigma(p,q,r)$ is non-trivial, i.e.
\[
h_{\iota}(\Sigma(p,q,r)):=[(CF(\Sigma(p,q,r)), \iota)] \neq 0
\]

\begin{rem}
For example, as evidence, we can consider the families $\Sigma(2,3,6n+1)$ and $\Sigma(2,2n+1,4n+3)$, for any odd integer $n$. The discussions in Section~\ref{strong_cork_section} will imply that both of these families (with switched orientation) satisfy the conditions of Question~\ref{1_c}, hence, an affirmative answer to the question will show that all members of these families are locally non-trivial. Now, of course, this is already known from the work of Dai and Manolescu \cite{dai2019involutive}. However, we ask, in Question~\ref{1_c}, whether there is an underlying reasoning behind this phenomenon coming from the existence of extension of a diffeomorphism over a $b^{+}=1$ bound with appropriate $\spinc$ and $H^{+}$ condition.
\end{rem}
While glancing through Sections~\ref{proof_main_exotic} and Section~\ref{strong_cork_section} the readers will realize that arguments presented in those sections are secretly in favor of an affirmative reply to the questions above.

\begin{acknowledgement}
This project began in the program entitled ``Floer homotopy theory" held at  MSRI/SL-Math on Aug-Dec in 2022. Thus this work was supported by the National Science Foundation under Grant No. DMS-1928930. The authors thank the organizers of the program for inviting them. We would like to thank Kristen Hendricks and Tye Lidman for their helpful conversations. Also, we thank Danny Ruberman for explaining his work with Akbulut \cite{AR16}, Irving Dai for many helpful discussions regarding corks, Kouichi Yasui for helpful comments in an earlier draft, David Baraglia for pointing out an omission in an earlier draft, and Ian Zemke for helpful correspondence that positively influenced the proof of Theorem~\ref{exotic embeddings}.   Lastly, portions of this project were completed during a visit by the third author to Rutgers funded by a Sloan fellowship held by Kristen Hendricks; the authors are grateful to the Sloan Foundation for its support. 
In addition, Hokuto Konno was partially supported by JSPS KAKENHI Grant Numbers 17H06461, 19K23412, and 21K13785,
Masaki Taniguchi was partially supported by JSPS KAKENHI Grant Number 20K22319, 22K13921 and RIKEN iTHEMS Program.

\end{acknowledgement}

\section{Main ingredients}

\subsection{A constraint from families Seiberg--Witten theory} 

\label{subsection A constraint from families Seiberg--Witten theory}
We first summarize a few results on diffeomorphisms obtained from \cite{KT22}. Firstly, let us recall  the definition of $H^+(X)$.
Let $X$ be an oriented compact 4-manifold, with or without boundary.
Then the Grassmanian $\mathrm{Gr}^+(X)$ that consists of maximal dimensional positive-definite (with respect to the intersection form) subspaces of $H^2(X;\R)$ is known to be contractible.
(See, such as, \cite[Subsection~3.1]{LL01}.)
A choice of such a subspace is denoted by $H^+(X)$.

Because of the contractibility (more weakly, the 1-connectivity) of $\mathrm{Gr}^+(X)$, we can define the notion that an automorphism $\varphi : H^2(X;\R) \to H^2(X;\R)$ of the intersection form {\it preserves} or {\it reverses} orientation of $H^+(X)$.
Indeed, let us pick a particular choice of $H^+(X)$ and pick an orientation $\mathcal{O}$ of $H^+(X)$.
Then $\varphi(H^+(X))$ also lies in $\mathrm{Gr}^+(X)$.
Choosing an isotopy between $H^+(X)$ and $\varphi(H^+(X))$ in $\mathrm{Gr}^+(X)$, we may compare $\mathcal{O}$ with $\varphi(\mathcal{O})$, and because of that $\pi_1(\mathrm{Gr}^+(X))=0$, the choice of isotopy does not affect this comparison.
We say that $\varphi$ {\it preserves} orientation of $H^+(X)$ if $\mathcal{O}$ coincides with $\varphi(\mathcal{O})$ after the above isotopy deformation between $H^+(X)$ and $\varphi(H^+(X))$, and otherwise say that $\varphi$ {\it reverses} orientation of $H^+(X)$.
As we have seen, these notions are determined only by $\varphi$, independent of choices of $H^+(X)$ in $\mathrm{Gr}^+(X)$ and $\mathcal{O}$.

Given an orientation-preserving diffeomorphism $f : X \to X$, we say that $f$ preserves (resp. reverses) orientation of $H^+(X)$ if the induced map $f^\ast : H^2(X;\R) \to H^2(X;\R)$ preserves (resp. reverses) orientation of $H^+(X)$.

\begin{thm}[{\cite[Theorem~1.1]{KT22}}]
\label{theo: main theo}
Let $Y$ be an oriented rational homology $3$-sphere and $X$ be an oriented compact smooth $4$-manifold bounded by $Y$.
Assume that $b_{1}(X)=0$ and $b^+(X)=1$.
Let $\fraks$ be a spin$^{c}$ structure on $X$ and let $\frakt$ be the spin$^{c}$ structure on $Y$ defined as the restriction of $\fraks$. If there is an orientation-preserving diffeomorphism $f : X \to X$ so that $f^* \fraks = \fraks$, $f$ reverses orientation of $H^+(X)$, and $f|_Y=\id_Y$, then we have 
\begin{align}
\label{eq: main ineq in main thm}
\frac{c_{1}(\fraks)^{2} - \sigma(X)}{8} \leq \delta(Y,\frakt).
\end{align}
\end{thm}

\begin{proof}
Let $E \to S^1$ be the mapping torus of $f$ with fiber $X$.
Then the induced real vector bundle $H^+(E) \to S^1$ with fiber $H^+(X)$ satisfies that $w_1(H^+(E)) \neq 0$ since $f$ reverse orientation of $H^+(X)$.
Also, the structure group of $E$ lifts to the automorphism group of the spin$^c$ 4-manifold $(X,\fraks)$ by picking a lift of $f$ to a spin$^c$ automorphism, which can be taken to be trivial over the trivial family $Y\times S^1 \to S^1$ of the boundary.
Thus the claim of the \lcnamecref{theo: main theo} immediately follows from \cite[Theorem~1.1]{KT22}.
\end{proof}

For spin $4$-manifolds with boundary, we have a refinement of \cref{theo: main theo} using the Manolescu invariants $\alpha, \beta, \gamma$ defined in \cite{Ma16}, instead of $\delta$:

\begin{thm}[{\cite[Theorem~1.2]{KT22}}]
\label{theo: main theo2}
Let $Y$ be an oriented rational homology $3$-sphere and $X$ be an oriented compact smooth $4$-manifold bounded by $Y$.
Assume that $b_{1}(X)=0$.
Let $\fraks$ be a spin structure on $X$ and let $\frakt$ be the spin structure on $Y$ defined as the restriction of $\fraks$.

Then:
\begin{itemize}
\item If $b^+(X)=1$ and there is an orientation-preserving diffeomorphism $f : X \to X$ so that $f^* \fraks = \fraks$, $f$ reverses orientation of $H^+(X)$, and $f|_Y=\id_Y$, then we have
\begin{align}
\label{eq: main ineq in main thm2}
\frac{-\sigma(X)}{8} \leq \gamma(Y,\frakt).
\end{align}
\item If $b^+(X)=2$ and there is an orientation-preserving diffeomorphism $f : X \to X$ so that $f^* \fraks = \fraks$, $f$ reverses orientation of $H^+(X)$, and $f|_Y=\id_Y$, then we have
\begin{align}
\label{eq: main ineq in main thm3}
\frac{-\sigma(X)}{8} \leq \beta(Y,\frakt).
\end{align}
\item If $b^+(X)=3$ and there is an orientation-preserving diffeomorphism $f : X \to X$  so that $f^* \fraks = \fraks$, $f$ reverses orientation of $H^+(X)$, and $f|_Y=\id_Y$, then we have
\begin{align}
\label{eq: main ineq in main thm4}
\frac{-\sigma(X)}{8} \leq \alpha(Y,\frakt).
\end{align}
\end{itemize}
\end{thm}

\begin{proof}
As in the proof of \cref{theo: main theo}, form the mapping torus $E \to S^1$ of $f$ with fiber $X$.
Then the structure of $E$ reduces to the spin automorphism group of $(X,\fraks)$.
Then the claim of the \lcnamecref{theo: main theo2} immediately follows from \cite[Theorem~1.2]{KT22}.
\end{proof}

For a closed 4-manifold, we can obtain a variant of \cref{theo: main theo} taking into account what is called the charge conjugation symmetry on the set of spin$^c$ structures.
The statement is:

\begin{thm}
\label{thm: charge conj}
Let $X$ be an oriented closed smooth $4$-manifold.
Assume that $b_{1}(X)=0$ and $b^+(X)=1$.
Let $\fraks$ be a spin$^{c}$ structure on $X$. 
If there is an orientation-preserving diffeomorphism $f : X \to X$ such that $f^* \fraks = \bar{\fraks}$ and $f$ preserves orientation of $H^+(X)$, then we have 
\begin{align}
\label{eq: main ineq in main thm conj}
\frac{c_{1}(\fraks)^{2} - \sigma(X)}{8} \leq 0.
\end{align}    
\end{thm}

We need a few preliminaries to prove \cref{thm: charge conj}.
First let us recall a standard setup of the Seiberg--Witten equations.
Let $(X,\fraks)$ be a smooth spin$^c$ closed 4-manifold with $b_1(X)=0$.
Fix a Riemannian metric $g$ on $X$.
Let $S^+(\fraks,g)$ and $S^-(\fraks,g)$ denote the positive and negative spinor bundle for $(\fraks,g)$ respectively.
Fix a reference $\mathrm{U}(1)$-connection $A_0$ of the determinant line bundle $L(\fraks,g) \to X$ for $(\fraks,g)$.
Let $\mathcal{A}(\fraks,g)$ denote the space of $\mathrm{U}(1)$-connections of $L(\fraks,g)$.

Set 
\begin{align*}
&\mathcal{C}(X,\fraks,g, A_0) =  (A_0 + \ker{d^\ast_g}) \otimes \Gamma(S^+(\fraks,g)),\\
&\mathcal{D}(X,\fraks,g) = \Omega^+_g(X) \oplus \Gamma(S^-(\fraks,g)).
\end{align*}
Here $\ker{d^\ast_g}$ is the kernel of the adjoint operator $d^\ast_g : \Omega^1(X) \to \Omega^0(X)$ of the exterior derivative with respect to $g$, and 
$\Omega^+_g(X)$ is the space of $g$-self-dual 2-forms.
The Seiberg--Witten equations give rise to a non-linear $\mathrm{U}(1)$-equivariant map
\[
\Psi(X,\fraks,g, A_0) : \mathcal{C}(X,\fraks,g, A_0) \to \mathcal{D}(X,\fraks,g),
\]
which we call the {\it monopole map}.
The domain $\mathcal{C}(X,\fraks,g, A_0)$ is called a {\it global slice} for a framed gauge group.
The Seiberg--Witten moduli space corresponds to the quotient $\Psi(X,\fraks,g, A_0)^{-1}(0)/\mathrm{U}(1)$.

Let us suppose that we have an orientation-preserving diffeomorphism $f : X \to X$ such that $f^\ast \fraks$ is isomorphic to $\bar{\fraks}$.
This condition implies that there is an isomorphism $\tilde{f} : \bar{\fraks} \to \fraks$ of spin$^c$ structures that covers $f : X \to X$.
The isomorphism $\tilde{f}$ gives rise to bijections (denoted by the same symbol $\tilde{f}^\ast$)
\begin{align*}
&{\tilde{f}}^\ast : \mathcal{A}(\fraks,g) \to \mathcal{A}(\bar{\fraks},f^\ast g),\\
&{\tilde{f}}^\ast : \mathcal{C}(X,\fraks,g, A_0) \to \mathcal{C}(X,\bar{\fraks},f^\ast g, f^\ast A_0),\\
&{\tilde{f}}^\ast : \mathcal{D}(X,\fraks,g) \to \mathcal{D}(X,\bar{\fraks},f^\ast g).
\end{align*}

Now let us consider the conjugation symmetry on the spin$^c$ structures, which is a symmetry induced from the complex conjugation of $\mathrm{U}(1)$ in $\mathrm{Spin}^c(4) = (\mathrm{Spin}(4) \times \mathrm{U}(1))/\{\pm1\}$.
Let $\bar{\fraks}$ denote the conjugate spin$^c$ structure of $\fraks$.
The conjugation induces bijective maps (denoted by the same symbol $c$)
\begin{align*}
&c : \mathcal{A}(\bar{\fraks}, f^\ast g) \to \mathcal{A}(\fraks,f^\ast g),\\
&c : \mathcal{C}(X,\bar{\fraks},f^\ast g, f^\ast A_0) \to \mathcal{C}(X,\fraks,f^\ast g, c(f^\ast A_0)),\\
&c : \mathcal{D}(X,\bar{\fraks},f^\ast g) \to \mathcal{D}(X,\fraks,f^\ast g),
\end{align*}
and the monopole maps are compatible with these conjugation maps $c$.

Pick take a path $g_t$ between $g$ and $f^\ast g$ in the space of metrics on $X$, and also a path $A_t$ between $A_0$ and $c(f^\ast A_0)$ in $\bigcup_{t \in [0,1]} \mathcal{A}(\fraks,g_t)$ so that $A_t \in \mathcal{A}(\fraks,g_t)$ for each $t$.
Then we obtain a 1-parameter family of monopole maps
\[
\Psi(X,\fraks,g_t, A_t) : \mathcal{C}(X,\fraks,g_t, A_t) \to \mathcal{D}(X,\fraks,g_t).
\]
Gluing $\mathcal{C}(X,\fraks,g_0, A_0)$ with $\mathcal{C}(X,\fraks,g_1, A_1)$ and $\mathcal{D}(X,\fraks,g_0)$ with $\mathcal{D}(X,\fraks,g_1)$ by $c\circ \tilde{f}^\ast$, we obtain 1-parameter families $\mathcal{C}, \mathcal{D}$ of functional spaces parameterized over $B=S^1 = [0,1]/(0\sim1)$.
Then $\Psi(X,\fraks,g_t, A_t)$ gives rise to a fiberwise map
\[
\xymatrix{
\mathcal{C} \ar[rr]^{\Psi} \ar[dr]& &\mathcal{D}\ar[dl]\\
&B&
}
\]
between these families $\mathcal{C}, \mathcal{D}$ over $B$.

Note that the map $\Psi$ is not $\mathrm{U}(1)$-equivariant anymore.
This is because the complex conjugation does not commute with the multiplication by $i \in \mathrm{U}(1)$ on the spinors.
However, the complex conjugation does commute with $\Z/2 \subset \mathrm{U}(1)$.
Thus $\Psi$ is a $\Z/2$-equivariant map.

Now we consider a finite-dimensional approximation following Furuta~\cite{Fu01} and Bauer--Furuta~\cite{BF04}.
As in the construction of the usual Bauer--Furuta invariant for families (see such as \cite{Szy10,BK19}), we can take a finite-dimensional approximation of $\Psi$, which is a pointed $\Z/2$-equivariant continuous map
\begin{align}
\label{eq: twisted BF inv}
\psi : \mathrm{Th}(V') \to \Th(V),
\end{align}
which we call a {\it twisted finite-dimensional approximation}, satisfying the following properties:
\begin{itemize}
\item[(i)] $V$ and $V'$ are finite-rank vector bundles over $B$ and $\mathrm{Th}(-)$ denotes the Thom space.
\item[(ii)] $V$ and $V'$ are direct sums of two types of vector bundles over $B$, written as
\[
V = V_0 \oplus V_1,\quad
V' = V'_0 \oplus V'_1.
\]
The group $\Z/2$ acts trivially on $V_0$ and $V_0'$, which are obtained as finite-dimensional approximations of families of spaces of differential forms.
On the other hand, $\Z/2$ acts on $V_1$ and $V_1'$ as fiberwise scalar multiplication by $\Z/2=\{1,-1\}$, which are obtained as finite-dimensional approximations of families of spinors.
\item[(iii)] $\rank_{\R}V_1 - \rank_{\R}V'_1 = (c_1(\fraks)^2 - \sigma(X))/4$.
\item[(iv)] The $\Z/2$-invariant part $\psi^{\Z/2} : \mathrm{Th}(V_0) \to \mathrm{Th}(V_0')$ is induced from a fiberwise linear injection $\psi^{\Z/2} : V_0 \to V'_0$.
Note that this condition implies that 
\[
\operatorname{Coker}\psi^{\Z/2} := 
\bigcup_{b \in B}(\operatorname{Coker}\psi^{\Z/2}_b)
\]
forms a vector bundle over $B$, where $\psi_b$ denotes the restriction of $\psi$ over the fibers over $b \in B$.
\item[(v)] $\Coker\psi^{\Z/2}$ is isomorphic to $H^+(X_f) \otimes \R_{f} \to B$.
Here $H^+(X_f) \to B$ is an associated vector bundle with fiber $H^+(X)$ determined by the monodromy action of $f$ on $H^+(X)$, and $\R_{f} \to B$ is the non-trivial real line bundle.
\end{itemize}

Except for the presence of $\R_f$ in the last property (v), the usual families Bauer--Furuta invariant also satisfies the all of the above properties \cite{BK19}, under the restriction of the symmetry from $\mathrm{U}(1)$ to $\Z/2$.
The last property (v) reflects the assumption that $f^\ast \fraks \cong \bar{\fraks}$: the line bundle $\R_f$ appears since the conjugation symmetry $c$ restricted to $\Omega^+_{f^\ast g}(X) \to \Omega^+_{f^\ast g}(X)$ is given by multiplication with $-1$.

\begin{rem}
More generally, suppose that one has a fiber bundle $E \to B$ over a general base space with closed 4-manifold fiber $X$ such that the monodromy action of $E$ preserves the set $\{\fraks, \bar{\fraks}\}$.
Then the fiberwise cokernel of the $\Z/2$-invariant part of a twisted finite-dimensional approximation is given as $H^+(E) \otimes \R_{\rho}$.
Here $H^+(E) \to B$ is a vector bundle with fiber $H^+(X)$ associated with $E$, and $\R_\rho \to B$ is the real line bundle determined by a homomorphism $\rho : \pi_1(B) \to \Z/2$ corresponding to the monodromy action of $E$ on the set $\{\fraks, \bar{\fraks}\}$.
If $\fraks=\bar{\fraks}$, we may choose $\rho$ as we like.
\end{rem}

In general, the following Borsuk--Ulam-type theorem holds for a setup that generalizes the above:

\begin{prop}[{\it cf.} {\cite[Theorem~3.1]{KN20}}]
\label{prop: BU}
Let $B$ be a compact topological space.
Let $V_0, V_0' \to B$ be finite-rank vector bundles acted by $\Z/2$ trivially.
Let $V_1, V_1' \to B$ be finite-rank vector bundles acted by $\Z/2$ as fiberwise scalar multiplication by $\Z/2$.
Set 
\[
V=V_0\oplus V_1,\quad  V'=V_0'\oplus V_1'\quad \text{and}\quad  r=\rank_{\R}V'_0 - \rank_{\R}V_0.
\]
Suppose that there is a pointed $\Z/2$-equivariant continuous map
\[
\psi : \mathrm{Th}(V) \to \mathrm{Th}(V')
\]
such that $\psi^{\Z/2} : \mathrm{Th}(V_0) \to \mathrm{Th}(V_0')$ is induced from a fiberwise linear injection $\psi^{\Z/2} : V_0 \to V'_0$.
Suppose also that the Stiefel--Whitney class
\[
w_{r}(\operatorname{Coker}\psi^{\Z/2}) \in H^r(B;\Z/2)
\]
is non-zero.
Then we have 
\begin{align}
\label{eq: ineq BU}
\rank{V}_1 \leq \rank{V}_1'.
\end{align}
\end{prop}

\begin{proof}
This is essentially given in the proof of \cite[Theorem~3.1]{KN20}, but we give a sketch of the proof for readers' convenience.

In this proof, we work with $\Z/2$-coefficient cohomology.
Let $i : \Th(V_0) \to \Th(V)$ and $i' : \Th(V_0') \to \Th(V')$ denote the maps induced from the inclusions $V_0 \hookrightarrow V$ and $V_0' \hookrightarrow V'$.  
Then we obtain a commutative diagram
\begin{align}
\label{eq: Borsuk Ulam comm diag}
\begin{split}
\xymatrix{
\tilde{H}_{\Z/2}^\ast(\Th(V);\Z/2) \ar[d]^{i^\ast} & \tilde{H}_{\Z/2}^\ast(\Th(V')) \ar[d]^{i'^\ast} \ar[l]_-{\psi^\ast} \\
\tilde{H}_{\Z/2}^\ast(\Th(V_0)) & \tilde{H}_{\Z/2}^\ast(\Th(V_0')). \ar[l]_{(\psi^{\Z/2})^\ast}
}
\end{split}
\end{align}

Let $\tau_{\Z/2}(V') \in \tilde{H}_{\Z/2}^\ast(\Th(V'))$ denote the $\Z/2$-equivariant Thom class of $V'$.
By the $\Z/2$-equivariant Thom isomorphism, there is a cohomology class $\alpha \in H_{\Z/2}^\ast(B)$ such that 
\begin{align}
\label{eq: Borsuk Ulam mapping degree}
\alpha\tau_{\Z/2}(V) = \psi^\ast \tau_{\Z/2}(V').
\end{align}
On the other hand, we have
\begin{align}
\label{eq: Borsuk Ulam restriction}
\begin{split}
&i^\ast \tau_{\Z/2}(V) = e_{\Z/2}(V_1)\tau_{\Z/2}(V_0),\quad 
i'^\ast \tau_{\Z/2}(V') = e_{\Z/2}(V_1')\tau_{\Z/2}(V_0'),\\
&(\psi^{\Z/2})^\ast \tau_{\Z/2}(V'_0) = e_{\Z/2}(\Coker{\psi}^{\Z/2})\tau_{\Z/2}(V_0),
\end{split}
\end{align}
where $e_{\Z/2}$ denotes the $\Z/2$-equivariant Euler class and we used the assumption on $\psi^{\Z/2}$ to derive the last equation.
Combining \eqref{eq: Borsuk Ulam comm diag}, \eqref{eq: Borsuk Ulam mapping degree}, \eqref{eq: Borsuk Ulam restriction} with the $\Z/2$-equivariant Thom isomorphism, we obtain
\begin{align}
\label{eq: Euler divisibility}
\alpha e_{\Z/2}(V_1) = e_{\Z/2}(V_1)e_{\Z/2}(\Coker{\psi}^{\Z/2}).
\end{align}

In general, for a vector bundle $E \to B$ of rank $r$ acted by $\Z/2$ trivially, we have 
\begin{align}
\label{eq: Euler for trivially acted}
e_{\Z/2}(E) = w_{r}(E).
\end{align}
On the other hand, for a vector bundle $F \to B$ of rank $n$ acted by $\Z/2$ as fiberwise scalar multiplication by $\Z/2$, we have 
\begin{align}
\label{eq: Euler for acted by scalar multipl}
e_{\Z/2}(E) = w_0(F)w^{n} + w_1(F)w^{n-1} + \cdots + w_{n}(F),
\end{align}
where $w$ is the generator of $H^1_{\Z/2}(\pt)$.

Thus, if $w_{r}(\operatorname{Coker}\psi^{\Z/2}) \neq 0$, the desired inequality \eqref{eq: ineq BU} follows from \eqref{eq: Euler divisibility}, \eqref{eq: Euler for trivially acted}, \eqref{eq: Euler for acted by scalar multipl}.
\end{proof}

Now we are ready to prove \cref{thm: charge conj}:

\begin{proof}[Proof of \cref{thm: charge conj}]
By the assumption that $f^\ast\fraks \cong \bar{\fraks}$, we may have a twisted finite-dimensional approximation \eqref{eq: twisted BF inv}.
By the assumption that $f$ preserves orientation of $H^+(X)$, the real line bundle $H^+(X_f) \otimes \R_f \to B=S^1$ is non-trivial.
Thus the claim of the \lcnamecref{thm: charge conj} follows from \cref{prop: BU} together with the property (iii) that $\psi$ satisfies.
\end{proof}

\subsection{A constraint from involutive Heegaard Floer theory}\label{HF_intro}
We will now discuss actions of various symmetries on the Heegaard Floer homology and related invariants established in \cite{HM, HMZ, DHM20, DMS, mallick2022knot}.

In \cite{HM} Hendricks and Manolescu studied the $\spinc$-conjugation action $\iota$ on the Heegaard Floer homology  of the pair $(Y,\s_0)$, $HF(Y,\s_0)$, where $\s_0$ is a spin structure on $Y$:
\[
\iota: HF(Y,\s_0) \rightarrow HF(Y,\s_0).
\]
Since then many authors have used this action successfully to study homology cobordism, and knot concordance. Indeed, the formalism of \textit{local equivalence} using the $\iota$-action, by Hendricks, Manolescu and Zemke \cite{HMZ} has propelled involutive Heegaard Floer homology to produce many applications in this realm. Later in \cite{DHM20}, the second author jointly with Dai and Hedden employed a similar formulation for the action of an orientation-preserving diffeomorphism $\tau$ on rational homology spheres, equipped with a spin structure $\s_0$ fixed by $\tau$: 
\[
\tau: HF(Y,\s_0) \rightarrow HF(Y,\s_0).
\]
This lead to an invariant capable of detecting strong corks \cite[Theorem 1]{DHM20}.

In \cite{HM}, Hendricks and Manolescu defined $\underline{d}$, the involutive $d$-(lower) invariant. This is a $2\mathbb{Z}$-valued invariant for $\mathbb{Z}_2$-homology spheres and can be interpreted as a map from the $\mathbb{Z}_2$-homology cobordism group:
\[
\underline{d}: \Theta^3_{\mathbb{Z}_2} \rightarrow 2 \mathbb{Z}.
\]
\noindent
Later in the presence of the action of a symmetry $\tau$ on $(Y, \s_0)$, a straightforward generalization of $\underline{d}$ was given in \cite{DHM20}. We record the definition here for the convenience of the reader:

\begin{defn}\label{defn_d_lower}
Let $(Y,\s_0)$ be the $3$-manifold equipped with a spin structure $\s_0$ and $\tau$ be a diffeomorphism on $Y$, which fixes $\s_0$. The invariant $\underline{d}_{\tau}(Y,\s_0)$ is defined to be 2 more than\footnote{Here we follow the convention that $HF(S^3)= \mathbb{F}[U]_{(-2)}$.} the highest grading of a $U$ non-torsion tower generator $[x] \in HF^{-}(Y, \s_0)$, so that $\tau[x]=[x]$. 
\end{defn}
In this case, as shown in \cite{DHM20}, the invariant $\underline{d}_{\tau}(Y,\s_0)$ can be interpreted as a map from the \textit{equivariant $\mathbb{Z}_2$-homology bordism group of diffeomorphisms} see \cite{DHM20}:
\[
\underline{d}_{\tau}: \Theta^{\tau}_{\mathbb{Z}_2} \rightarrow 2 \mathbb{Z}.
\]
\noindent
We refer to this invariant as the \textit{equivariant-involutive $d$-lower invariant}. There is another invariant of this type $\overline{d}_{\tau}(Y,\s_0,\tau)$ called the \textit{equivariant-involutive $d$-upper invariant}. However, this will not be so relevant to this article. We refer readers to \cite{HM} for a definition.

We now focus on the symmetry of a knot. Firstly, we categorify the types of symmetries on knots that will interest us.

\begin{defn}
A knot $(K, \tau)$ is said to be \textit{strongly invertible} if 
$\tau$ in an orientation-preserving involution on $S^3$, which preserves the knot $K$ setwise and switches the orientation on the knot $K$. On the other hand, a knot $(K, \tau)$ is called \textit{periodic} if there exists an involution $\tau$ of $S^3$, which preserves the knot setwise and the fixed set of $\tau$ does \textit{not} intersect $K$.
\end{defn}
\noindent
Following up on the study of $3$-manifolds with symmetry, the second author studied the action of symmetry of a strongly invertible knot or a periodic knot on the knot Floer homology \cite{mallick2022knot},
\[
\tau_K: CFK(K) \rightarrow CFK(K).
\]
Montesinos showed \cite{montesinos1975surgery} Dehn-surgery on a strongly invertible knot (or a periodic knot) $(K,\tau)$ induces a symmetry $\tau$ on the surgered manifolds. It was shown in \cite{mallick2022knot} that the action of the symmetry on the knot Floer homology is related to the action of the (induced) symmetry on the surgered manifold, via the \textit{equivariant surgery formula}.

In a related work in \cite{DMS}, the second author jointly with Dai and Stoffregen defined various invariants of the equivariant concordance group, using knot Floer homology. These invariants are related to the involutive $d$-invariants mentioned above. One such invariant will be useful for the purpose of this article. Let us denote by $A_0(K)$ the large surgery (sub)-complex of $CFK(K)$. Recall that $A_0(K)$ is generated by the elements of the form
\[
[\bold{x}, i, j] \; \; \mathrm{with} \; i \leq 0 \; \mathrm{and} \; j \leq 0.
\]
We now record the definition of the relevant invariant:
 \begin{defn}
Given a strongly invertible knot $(K,\tau)$. Let $A_0(K)$ denote the the large surgery complex of $CFK(K)$. The invariant $\underline{V}^{\tau}_{0}(K,\tau)$ is defined as:
\[
\underline{V}^{\tau}_{0}(K,\tau):=-\frac{1}{2} \underline{d}_{\tau}(A_0(K),\tau).
\]
\end{defn}
\noindent
It was shown in \cite{DMS} that $\underline{V}^{\tau}_{0}(K,\tau)$ is an invariant of the equivariant knot concordance group $\tilde{\mathcal{C}}$. There is also a similarly defined invariant $\overline{V}^{\tau}_{0}(K,\tau)$ which will not be relevant to us for the purpose of this article. These invariants can be thought of as the involutive analog of the $V_0$ invariant defined by Rasmussen \cite{Rasmussen}. This invariant $V_0$ also satisfies a similar relation as its involutive counterpart:
\begin{align}\label{V=d}
V_0(K)=-\frac{1}{2}d(A_0(K))=-\frac{1}{2}d(S^3_{+1}(K)).
\end{align}
Moreover, the following relation follows from \cite{HM}
\[
\underline{V}^{\tau}_{0}(K,\tau) \geq V_0(K).
\] 
We will now compute these invariants for some knots. 
\begin{prop}\label{trefoil_sum}
Let $T_{2,3} \# T_{2,3}$ be equipped with the strong involution $\tau \# \tau$, where $\tau$ is the unique strong involution on $T_{2,3}$ as in Figure~\ref{trefoil_connect_intro}. Then 
\[
\underline{V}^{\tau}_0(T_{2,3} \# T_{2,3}) > V_0(T_{2,3} \# T_{2,3}) > 0.
\] 
\end{prop}

\begin{proof}
In Figure~\ref{trefoil_connected_complex}, we have described the knot Floer chain complex for $T_{2,3}$ and $T_{2,3} \# T_{2,3}$. Now, recall that using the integer surgery formula \cite{ozsvath2008knot} we can compute the $V_0$ invariant from the $A_0$-complex. We immediately obtain that 
\begin{figure}[h!]
\center
\includegraphics[scale=0.75]{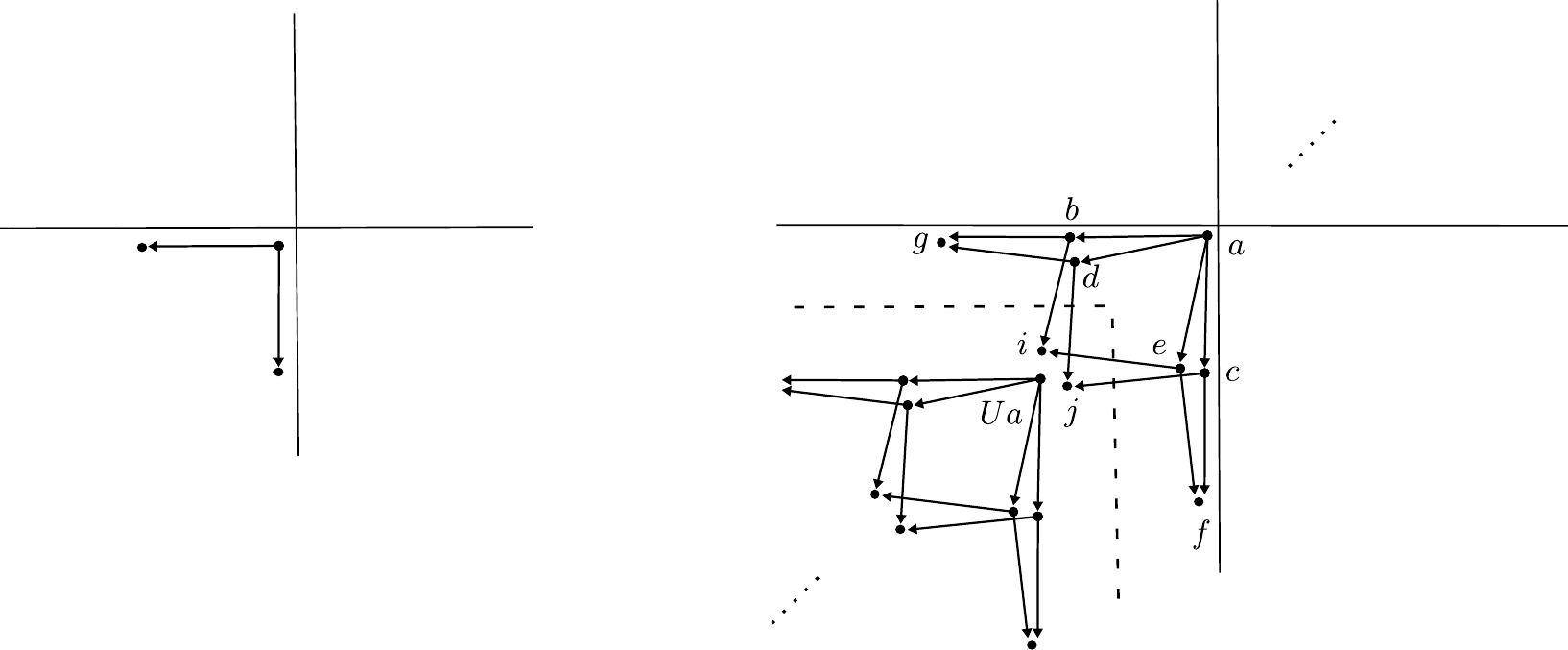}
\caption{Left: Knot Floer complex of $T_{2,3}$, Right:  Knot Floer complex of $T_{2,3} \# T_{2,3}$. The $A_0$-complex is represented by the complex lying south-west of the dotted line.}\label{trefoil_connected_complex}
\end{figure}
\[
V_0(T_{2,3} \# T_{2,3}) = 1.
\]
Now, in order to calculate the $\underline{V}^{\tau}_0(T_{2,3} \# T_{2,3})$ invariant, we first observe that by the equivariant connected sum formula proved in \cite{DMS}, the action of $\tau \# \tau$ on $CFK(T_{2,3} \# T_{2,3})=CFK(T_{2,3}) \otimes CFK(T_{2,3})$ can be identified with $\tau \otimes \tau$. It is then clear that since neither of $[i]$ or $[j]$ are fixed by $\tau \otimes \tau$, we have:
\[
\underline{V}^{\tau}_0(T_{2,3} \# T_{2,3}) > V_0(T_{2,3} \# T_{2,3}).
\]
Readers can in fact check that $\underline{V}^{\tau}_0(T_{2,3} \# T_{2,3})$ is $2$. This completes the proof.
\end{proof}

\begin{prop}\label{odd integer vno}
Let $\tau$ be the unique strong-involution on $T_{2,2n+1}$ torus knot and $n$ be any positive odd integer,
\[
\underline{V}^{\tau}_0(T_{2,2n+1} \# T_{2,2n+1}) > V_0(T_{2,2n+1} \# T_{2,2n+1}) > 0.
\] 
\end{prop}

\begin{proof}
The proof is similar to the Proposition~\ref{trefoil_sum}. After a change of basis, it is can be checked that the $H_*(A_0(T_{2,2n+1} \# T_{2,2n+1}))$ and the $\tau$ action on it is as described in Figure~\ref{tau_action_trefoil}.
\begin{figure}[h!]
\center
\includegraphics[scale=0.86]{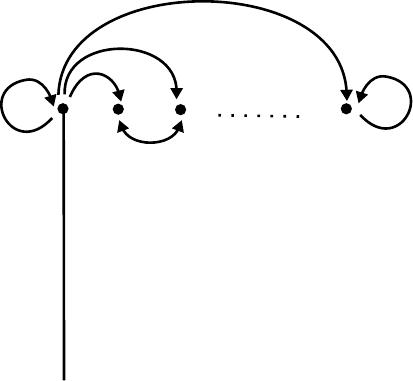}
\caption{The action of $\tau$ on $H_{*}(A_0(T_{2,2n+1} \# T_{2,2n+1})$ up to a change of basis. There are $n$-many $U$-torsion generators and a single $U$-non-torsion tower. The $\tau$ action pairs all but one $U$-torsion generator, while the tower generator maps to itself and all the $U$-torsion generators.}\label{tau_action_trefoil}
\end{figure}
It follows that
\[
V_0^{\tau}(T_{2,2n+1} \# T_{2,2n+1}) > V_0(T_{2,2n+1} \# T_{2,2n+1}).
\]
Moreover, a Maslov-grading computation yields that 
\[
V_0(T_{2,2n+1} \# T_{2,2n+1})= n.
\]
 \end{proof}
\noindent
Finally, we record the behavior of the $\underline{d}_{\tau}$ under the equivariant cobordism. Firstly, we recast the Definition~\ref{spinc_cobordism} with cobordisms in mind, as in \cite{DHM20}.
\begin{defn}
Let $(Y_i,\phi_i)$ be a pair of integer homology spheres $Y_i$ each equipped with an orientation-preserving diffeomorphism $\phi_i$. Let $(W,\s,f)$ be a tuple such that $W$ is a cobordism from $Y_1$ to $Y_2$ and $\s$ is a $\spinc$-structure on it such that $f$ restricts to $\phi_i$ on $Y_i$. We call $(W,\s,f)$ an equivariant $\spinc$-\textit{preserving} (resp. $\spinc$-\textit{reversing}) cobordism if $f (\s)= \s$ (resp. $f (\s) = \bar{\s}$).
\end{defn}

We now record the behavior of the invariants discussed earlier under specific equivariant cobordisms. Firstly we define the following for any cobordism $W$ equipped with a $\spinc$-structure $\s$:
\[
\Delta(W,\s)= \dfrac{c_1(\s)^2 - 2 \chi(W) - 3\sigma(W)}{4}.
\]
\begin{thm}[{\cite[Theorem 1.5]{DHM20}}]\label{equivariant_Cobordism}
Let $(W,\s,f)$ be a negative-definite, equivariant cobordism from $(Y_1,\phi_1)$ to $(Y_2,\phi_2)$ with $b_1(W)=0$ and $\Delta(W,\s)=0$. If $(W,\s,f)$ is $\spinc$-preserving then
\[
\underline{d}_{\tau}(Y_1,\phi_1) \leq \underline{d}_{\tau}(Y_2,\phi_2).
\]
\end{thm}
\noindent
Lastly, since we do not explicitly use the local equivalence class formalism to phrase most of our results, we choose to omit an explicit description for it. However, in some discussions,  familiarity with the formalism will be helpful. We request readers to refer to \cite{HMZ} and \cite{DHM20} for an introduction to the local invariants $h_{\iota}$, $h_{\tau}$ and $h_{\iota \circ \tau}$. These are enhanced versions of the involutive $d$-lower invariants discussed earlier. 

\subsection{A constraint from instanton theory} 

As a constraint from instanton Floer theory, we use the following result: 
\begin{thm}[{\cite[Theorem 1.8]{NST19}}] \label{NST}
For any knot $K$ in $S^3$, if $h(S_{+1}^3(K))<0$, then $\{S^3_{1/n}(K)\}_{n \in \Z_{>0}}$ is a linearly independent set in the homology cobordism group. 
\end{thm}

Note that $h(-)$ denotes the instanton Fr\o yshov invariant introduced in \cite{Fr02}, which can be seen as an instanton analog Heegaard Floer $d$-invariant or equivalently Seiberg--Witten Fr\o yshov invariant. However, compared with Heegaard Floer $d$-invariant, the techniques of computing instanton Fr\o yshov invariant have not been developed yet. Recently, the following ``tau-like invariants" of knots:
\begin{itemize}
    \item[(i)] instanton tau-invariant $\tau^\#$ defined using framed instanton homology \cite{BS21}, 
    \item[(ii)] instanton tau-invariant $\tau_I$ defined using sutured instanton homology for knots \cite{GLW19}, 
    \item[(iii)]
    s-tilde invariant $\wt{s}$ \cite{DISST22}
\end{itemize}
are introduced and these enable us to calculate the instanton Fr\o yshov invariant partially. These are knot concordance invariants. 

It is proven that $\tau^\#=\tau_I$ in \cite{GLW19}. Also, these invariants are known to be {\it slice torus invariants}, i.e. functions $f$ from the knot concordance group $\mathcal{C}$ to $\R$ satisfying the following conditions: 

\begin{itemize}
    \item $f$ is a homomorphism, 
    \item $f(K) \leq g_4(K)$, and 
    \item $f(T_{p,q}) = g_4(T_{p,q})$.
\end{itemize}

For any slice torus invariant $f$, it is proven in \cite{Liv04} that $f(K) = -\frac{1}{2} \sigma(K) $ for any alternating knot and $f(K) = g_4(K) $ for any quasi-positive knot. 

\begin{rem}
It is conjectured that 
\[
\tau(K)= \wt{s}(K) = \tau^\# (K)
\]
in \cite{DISST22}. It is true for any quasi-positive knot and alternating knot. 
\end{rem}

The following relations are proven recently: 

\begin{thm} [\cite{DISST22, BS22}] \label{relation s with h} 
Let $K $ be a knot in $S^3$. 
Suppose one of the following two conditions are satisfied: 
\begin{itemize}
    \item[(i)] $\tau^\#(K)>0$ or equivalently $\tau^I(K)>0$, 
    \item[(ii)] $\wt{s}(K)>0$.
\end{itemize}
Then we have $h(S^3_{+1}(K)) <0$.
\end{thm}

In particular, as examples, one can compute the followings: 

\begin{ex}\label{instanton ex}
Since $K_n = T_{2,2n+1} \# T_{2,2n+1}$ is a non-slice strongly quasi-positive knot for every $n$, we have 
\[
\tau^\# ( K_n) = \wt{s} (K_n) = g_4 (K_n) = g_3(K_n) >0,  
\]
where $g_4$ and $g_3$ denote the four and three genus of knots respectively. 
Therefore, we know $h(S_{+1}^3 (K)) <0 $ from  \cref{relation s with h} .
\end{ex}

\begin{rem}\label{distinguish more precisely}
The proof of the linear independence in \cref{NST} uses the homology cobordism invariant $r_0(Y)\in (0, \infty]$ of homology 3-spheres. 
If we assume 
\[
r_0(S_{1/n}^3(K)) \neq r_0(S_{1/m}^3(K')), 
\]
we can also distinguish $M_n(K)$ and $M_m(K')$ for different knots $K$, $K'$ and positive integers $n$ and $m$, where $M_n(K)$ and $M_m(K')$ are exotic 4-manifolds obtained in \cref{main:exotic}. 
For example, since $S^3_{1/2}(T_{2,2n+1}\# T_{2,2n+1} ) = \Sigma (\operatorname{Wh} (T_{2,2n+1}))  $, it is discussed \cite{NST19} that 
\[
r_0 (\Sigma (\operatorname{Wh} (T_{2,2n+1}))) > r_0 (\Sigma (\operatorname{Wh} (T_{2,2(n+1)+1}))),  
\]
where $\operatorname{Wh}$ denotes the satellite operation of the positively cusped Whitehead double. 
Thus, for $K_n = T_{2,2n+1}\# T_{2,2n+1}$,  one can actually distinguish $\partial M_2(K_n)$ and $\partial M_2(K_m)$ in \cref{main:exotic} if $m \neq n$.
\end{rem}

\section{Small exotic manifolds}\label{proof_main_exotic}

This section will be devoted to giving a proof of \cref{main:exotic} providing many examples of small exotic manifolds.

\subsection{Proof of \cref{main:exotic}}

Before going into the proof, since we use two different theories: Heegaad Floer theory and Seiberg--Witten Floer theory, we recall the relation between Fr\o yshov type invariants: 

\begin{rem}
As pointed out in \cite[Remark1.1]{LRS18}, 
the equality
\begin{align}\label{d=delta}
\frac{1}{2} d ( Y, \F  ) =  \delta (Y)
\end{align}
holds for every oriented homology 3-sphere $Y$.
This is a consequence of
the isomorphism between the Heegaard Floer homology and the monopole Floer homology \cite{KLTI}, \cite{KLTII}, \cite{KLTIII}, \cite{KLTIV}, \cite{KLTV} by Kutluhan, Lee and Taubes, or alternatively, \cite{CGHI} \cite{CGHII} \cite{CGHIII} by Colin, Ghiggini, and Honda and \cite{Ta10} by Taubes, also combined with \cite{LM18} by Lidman and Manolescu. 
(For the comparison of $\Q$-gradings, see \cite{RG18}, \cite{CD13} and \cite{HR17}.)
\end{rem}

In order to prove \cref{main:exotic}, we use Akbulut--Ruberman's technique \cite{AR16} stated below: 

\begin{thm}[\cite{AR16}, Theorem 
 1.1]\label{AR exotic}
Let $W$ be a compact smooth $4$-manifold with boundary. Suppose that there is a homeomorphism $F : W \to W$ such that $F|_{\partial W } $ does not extend to a diffeomorphism on $W$. 
Then, there are two smooth compact codimension-$0$ submanifolds $M$ and $M'$ in $W $ satisfying the following conditions: 

\begin{enumerate}[label=(\alph*)]
 \item $M$ and $M'$ are exotic, i.e. $M$ and $M'$ are homeomorphic but not diffeomorphic, 
    \item $ \partial M$ and $\partial M'$ are diffeomorphic to each other, and they are smoothly homology cobordant to $\partial W$, and
    \item $M$ and $M'$ are homotopy equivalent to $W$.
\end{enumerate}
\end{thm}
\noindent
In order to apply Theorem~\ref{AR exotic} to our situation, we shall construct a diffeomorphism $\tau' : Y'_ n = - S^3_{1/n}(K) \# S^3_{1/(n-1)}(K) \to Y'_{n}$ such that 
\begin{itemize}
    \item[(i)] $\tau'$ extends to $W_{2n}$ as a homeomorphism, and 
    \item[(ii)] $\tau'$ does not extend to $W_{2n}$ as a diffeomorphism. 
\end{itemize}
\noindent
The first statement follows from \cite[Theorem 1.5, Addendum]{Fr82}. 

\begin{proof}[Proof of \cref{main:exotic}]

We define
\[
\tau' := \tau \# \id : Y'_n =  - S^3_{1/n}(K) \# S^3_{1/(n-1)}(K) \to - S^3_{1/n}(K) \# S^3_{1/(n-1)}(K)= Y'_n, 
\]
where $\tau$ is a diffeomorphism on $S^3_{1/n}(K)$ obtained as an extension of the strong involution of $K$, as constructed in \cite[Lemma~5.2]{DHM20}. 

In order to apply \cref{AR exotic} to $(W_{2n}, \tau')$, we claim the following: 
\begin{claim}\label{claim*}
The diffeomorphism $\tau'$ does not extend to $W_{2n}$ smoothly (as a diffeomorphism). 
\end{claim}

From now on, for simplicity, we will denote $S^{3}_{1/n}(K)$ as $Y_n(K)$. We begin by observing that for any $n \geq 1$, there exists a $4$-manifold $X^{4}_{n}(K)$ such that $b^{+}(X_{n}(K))=1$ with $\partial X_{n}(K)= Y_{n}(K)$. This is obtained from the Kirby diagram for the $Y_n(K)$, as in Figure~\ref{kirby_cobordism}.
Note that if $n$ is even $X_n(K)$ is spin. Now rest of this paragraph uses the analysis in \cite[Section 5]{DHM20}, so we will be terse. Firstly, $\tau$ on $Y_n(K)$ corresponds to the diffeomorphism induced from the strong involution on the boundary of $X_n$, as depicted in Figure~\ref{kirby_cobordism}. 
Then by $\tau$ extending it to the two-handlebody, we obtain an extension of $\tau$ to $X_n$, say $f$. Note that $H^2(X_{n}(K)) \cong \mathbb{Z} \oplus \mathbb{Z}$, and under the isomorphism $H^2(X_n(K)) \cong H_2(X_{n}(K), \partial X_{n}(K))$, the generators of the second cohomology correspond to the cocore of the two $2$-handles attached. Again it follows from \cite[Section 5]{DHM20} that the $f$ switches the orientation of the cocores, and hence acts as $-1$ on $H^2(X_{n}(K))$ in both coordinates. This in turn implies that $f$ conjugates any $\spinc$-structure.
In particular, when $n$ is even, there exists a unique spin structure $\mathfrak{s}_0$ so that $f^{*}(\mathfrak{s}_0)= \mathfrak{s}_0$.
We now claim that $f$ switches the orientation of $H^{+}(X_{n}(K))$.
To see this, let $a$ and $b$ be the generators of $H^{2}(W)$ obtained as cocores of 2-handles given in \cref{kirby_cobordism} for given orientations of attaching spheres. Then it can be checked that $(a+b)$ generates $H^{+}(X_{n}(K))$, in particular $H^{+}(X_{n}(K))$ is 1-dimensional,  and hence $f$ switches its orientation. Let us now define a negative-definite cobordism $\widetilde{W}_{n}(K)$ from $Y_n(K)$ to $Y_{n-1}(K)$, by following the Kirby diagram described the Figure~\ref{kirby_cobordism}.

\begin{figure}[h!]
\center
\includegraphics[scale=0.95]{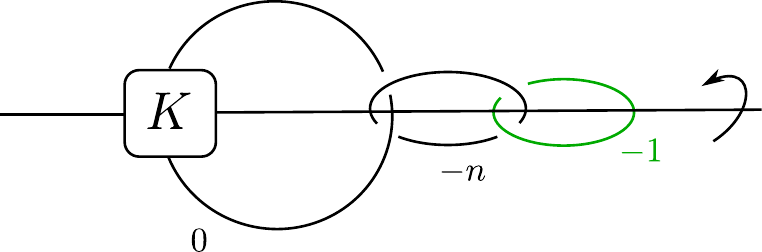}
\caption{In black: The $4$-manifold $X_n$ given by the two $2$-handles, with $\partial X_n = Y_n$. $X_n$ inherits a diffeomorphism induced from the symmetry depicted. In green: The equivariant cobordism $\widetilde{W}_n(K)$, given by the trace of the $(-1)$-framed green $2$-handle.}\label{kirby_cobordism}
\end{figure}
It follows that the cobordism $\widetilde{W}_{n}(K)$ is equivariant with respect to $\tau$ action in both boundary component, i.e. there exists a diffeomorphism $h$ of $\widetilde{W}_{n}(K)$ which restricts to the strong involution $\tau$ in each components. Note that the intersection form of $\widetilde{W}_{n}(K)$ is $(-1)$ and its generated by the meridian of $-1$-framed $2$-handle. Again from \cite[Section 5]{DHM20} it follows that $h$ acts multiplication by $-1$ on $H_2(\widetilde{W}_{n}(K),\partial \widetilde{W}_{n}(K))$, which in turn implies that $h$ is $\spinc$-conjugating. That is, let $\mathfrak{s}$ be the $\spinc$-structure such that $c_1(\mathfrak{s})$ generates $H^2(\widetilde{W}_n(K))$, then $h^{*}(\mathfrak{s})=\bar{\mathfrak{s}}$. 
In particular, $(\widetilde{W}_n(K), \s , h)$ is $\spinc$-reversing in the sense of Definition~\ref{spinc_cobordism}.

Towards contradiction, assume that Claim~\ref{claim*} is false. Then there is a diffeomorphism $\wt{\tau}' : W_n(K) \to W_n(K)$ extending $\tau'$. Now, since around the neck of the connected sum $- S^3_{1/n}(K) \# S^3_{1/(n-1)}(K)$, $\tau' = \tau \# \mathrm{id}$ was identity, one can suppose $\wt{\tau}'$ is identity on some small open neighborhood of a point in the neck in $W_n(K)$. 
By adding a $1$-handle and extending $\wt{\tau}'$ by identity, we obtain an equivariant cobordism from  $(S^3_{1/n}(K), \tau )$ to $(S^3_{1/(n-1)}(K), \id)$. Note that here the underlying cobordism is $\widetilde{W}_n(K)$. Hence, in order to prove Claim~\ref{claim*}, it is enough to prove the following claim:
\begin{claim}\label{claim**}
    There is no diffeomorphism on $\widetilde{W}_{2n}(K)$ which restricts to $\tau$ on $Y_{2n}(K)$ and the identity on $Y_{2n-1}(K)$. 
\end{claim}
\noindent
For the ease of notation, from this point onward, we will denote   $\widetilde{W}_{n}(K)$ as $W_n(K)$. Towards contradiction, we assume that there is such an extension, say $g$. We break this up in two different cases depending on how $g$ acts on $H^{2}(W_{2n}(K))$. 
\\
{\bf Proof of \cref{claim**} under $g_*=1$}
Firstly, suppose that $g$ acts as identity on $H^{2}(W_{2n}(K))$. Let us now consider 
\[
Z_{2n}(K):= X_{2n}(K) \cup_{Y_{2n}(K)} W_{2n}(K),
\] 
equipped with the diffeomorphism $\tilde{g}$ obtained by concatenating $f$ on $X_{2n}(K)$ and $g$. Let $\tilde{\s}$ be the $\spinc$-structure on $Z_{2n}(K)$ also obtained be concatenating $\s_0$ and $\s$. Note that $\partial Z_{2n}= Y_{2n-1}(K)$ and $b^{+}(Z_{2n})=1$, moreover $\tilde{g}$ preserves the $\spinc$-structure $\tilde{\s}$ on $Z_{2n}(K)$. Hence $(Z_{2n}, \tilde{g}, \tilde{s})$ satisfies the hypothesis for Theorem~\ref{theo: main theo}, which implies 
\begin{align}\label{delta_ineq}
\delta(Y_{2n-1}(K)) \geq  \frac{c_{1}(\tilde{\s})^{2} - \sigma(Z_{2n}(K))}{8} =0.
\end{align}
Now it follows from \eqref{d=delta}, that 
\[
d(Y_{2n-1}(K))=2 \delta(Y_{2n-1}(K)).
\]
Note that $d$-invariant of $+(1/n)$-surgery on a knot is same as the $d$-invariant of $(+1)$ surgery on it. Hence it follows from (\ref{V=d}) that
\[
d(Y_{2n-1}(K))= d(S^{3}_{+1}(K)) = -2V_{0}(K)
\]
Together with Equation~\ref{delta_ineq} contradicts the assumption $V_0(K) > 0$. 
\\
{\bf Proof of \cref{claim**} under $g_*=-1$}

Let us now assume that $g$ acts as multiplication by $(-1)$ on $H^2(W_{2n}(K))$. Consider $h \circ g$ on $W_{2n}(K)$, and note that $h \circ g$ restricted to $Y_{2n}(K)$ is identity and $Y_{2n-1}(K)$ is $\tau$.
Moreover, $h \circ g$ now acts trivially on $H_2(W_{2n}(K))$. Hence it follows that the following diagram commute up to chain homotopy:
\[\begin{tikzcd}[ampersand replacement=\&,sep=scriptsize]
	{CF(Y_{2n}(K))} \&\& {CF(Y_{2n-1}(K))} \\
	\\
	{CF(Y_{2n}(K))} \&\& {CF(Y_{2n-1}(K)).}
	\arrow["{F_{W,\mathfrak{s}}}", from=1-1, to=1-3]
	\arrow["{\mathrm{id}}"', from=1-1, to=3-1]
	\arrow["\tau", from=1-3, to=3-3]
	\arrow["{F_{W,\mathfrak{s}}}"', from=3-1, to=3-3]
\end{tikzcd}\]
Hence, we obtain from Theorem~\ref{equivariant_Cobordism} that
\begin{align}\label{ineq_1}
\underline{d}_{\tau}(Y_{2n}(K), \mathrm{id}) \leq \underline{d}_{\tau}(Y_{2n-1}(K), \tau).
\end{align}
Moreover, Definition~\ref{defn_d_lower} implies that $\underline{d}_{\tau}(Y_{2n}(K), \mathrm{id})= d(Y_{2n}(K))$. Now, we construct a cobordism $W'_{2n-1}(K)$ from $Y_{2n-1}(K)$ to $Y_{2n-3}(K)$ given in Figure~\ref{thm_proof_cobordism_2}:
\begin{figure}[h!]
\center
\includegraphics[scale=1.1]{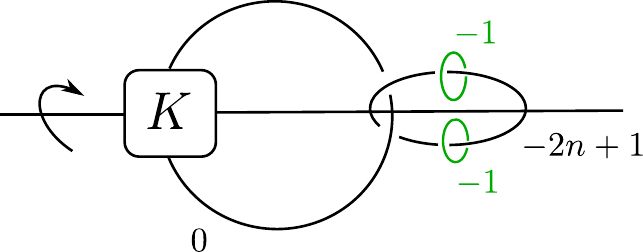}
\caption{The equivariant cobordism $W'_{2n-1}(K)$, induced by the green $(-1)$-framed $2$-handles.}\label{thm_proof_cobordism_2}
\end{figure}
\noindent
In the proof of \cite[Theorem 1.6]{DHM20}, it was shown that $W'_{2n-1}(K)$ is an equivariant, $\spinc$-preserving (for a certain $\spinc$-structure, say $\s^{\prime}$), negative-definite cobordism with $\Delta(W'_{2n-1}(K), \s^{\prime})=0$, from $(Y_{2n-1},\tau)$ to $(Y_{2n-3},\tau)$. In particular, again by Theorem~\ref{equivariant_Cobordism} we obtain:
\begin{align}\label{ineq_2}
\underline{d}_{\tau}(Y_{2n-3}(K), \tau) \leq \underline{d}_{\tau}(Y_{2n-1}(K), \tau).
\end{align}
We then apply a similar equivariant cobordism repeatedly to obtain
\begin{align}\label{ineq_3}
\underline{d}_{\tau}(Y_{2n-1}(K), \tau) \leq \underline{d}_{\tau}(Y_{1}(K), \tau).
\end{align}
Now, we claim that
\[
\underline{d}_{\tau}(Y_{1}(K), \tau) \leq \underline{d}_{\tau}(A_0(K), \tau).
\]
To see this, following \cite{dai20222} we construct an $U$-equivariant map, 
\[
F: CF(Y_{1}(K)) \rightarrow A_0(K)
\]
that induce isomorphism in homology after localizing with respect to $U$ and intertwines with the induced action of $\tau$ on $CF(Y_{1}(K)$ and $A_0(K)$ up to homotopy. This map is constructed as a composition of two maps. Firstly, we define an $\tau$-equivariant, $\spinc$-preserving cobordism map $f_W^{\prime}$
\[
f_{W^{\prime}, \s_0}: CF(Y_{1}(K)) \rightarrow CF(S^3_{+m}(K)).
\]
Here $W^{\prime}$ is given by the Figure~\ref{thm_proof_cobordism_3}, $m$ is a large positive integer and $\s_0$ on $W$ is the unique spin structure on $W^{\prime}$ (note that $W^{\prime}$ has no torsion hence it has a unique spin structure) \footnote{We also take $m$ to be odd, so that there is a unique spin structure on $S^3_{+m}(K)$}.
\begin{figure}[h!]
\center
\includegraphics[scale=1.1]{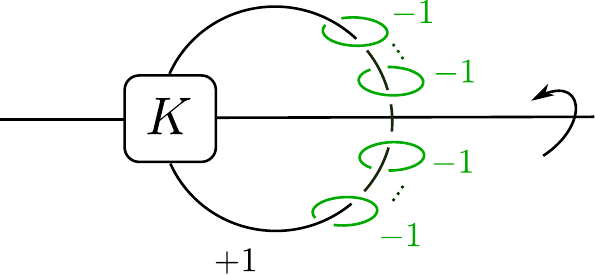}
\caption{The equivariant cobordism $W'$, induced by the green $(-1)$-framed $2$-handles.}\label{thm_proof_cobordism_3}
\end{figure}
It follows from the Figure~\ref{thm_proof_cobordism_3} that $W^{\prime}$ is $\tau$-equivariant, and since $\s_0$ is the unique spin structure, $W^{\prime}$ is $\spinc$-preserving. 
Now we consider the large surgery homotopy equivalence $\Gamma_{n,0}$ from \cite{ozsvath2008knot}:
\[
\Gamma_{n,0}: CF(S^3_{+m}(K), [0]) \rightarrow A_0(K).
\]
It follows from the equivariant surgery formula \cite{mallick2022knot} that $\Gamma_{n,0}$ intertwines with $\tau$ actions on $CF(S^3_{+n}(K), [0])$ and  $A_0(K)$ up to homotopy ($\tau$ induce an action on $A_0(K)$).
We now define the local map $F$ as:
\[
F:= \Gamma_{n,0}   \circ  f_{W^{\prime}, \s_0}.
\]
By construction, this map is grading preserving and it intertwines with the action of $\tau$ on both ends. The existence of $F$ implies (see the proof of \cite[Proposition 4.10]{DHM20})
\begin{align}\label{ineq_4}
\underline{d}_{\tau}(Y_{1}(K), \tau) \leq \underline{d}_{\tau}(A_0(K),\tau).
\end{align}
Combining the above inequalities (\ref{ineq_1}),(\ref{ineq_2}),(\ref{ineq_3}),(\ref{ineq_4}) yields:
\[
-2V_0(K)=d(Y_{2n}(K))=\underline{d}_{\tau}(Y_{2n}(K), \mathrm{id}) \leq -2 \underline{V}^{\tau}_{0}(K).
\]
This contradicts the hypothesis that 
\[
\underline{V}^{\tau}_{0}(K) > V_0(K) > 0.
\]
Hence, we have shown that there is no diffeomorphism $g$ on $W_{2n}(K)$, which restricts to $\tau$ on $Y_{2n}(K)$ and to the identity on $Y_{2n-1}(K)$. 

This completes the proof of \cref{claim**} thus we have \cref{claim*}, i.e. there is no smooth diffeomorphism extension of $\tau': Y_{2n}' = - S^3_{1/2n}(K) \# S^3_{1/(2n-1)}(K)  \to Y_{2n}'$ to $W_{2n}$. 
On the other hand, we have an extension of $\tau'$ as a homomorphism $F_{2n}$ on $W_{2n}$. We now apply \cref{AR exotic} to $(W_{2n}, F_{2n}, \tau' )$ and obtain smooth compact codimension-$0$ $4$-manifolds $M_n$ and $M'_n$ in $W_{2n}$ satisfying the following conditions: 
\begin{itemize}
    \item $M_n$ and $M'_n$ are exotic, 
    \item $ \partial M_n = \partial M'_n$ is a smoothly homology cobordant to $\partial W_{2n}$, and
    \item $M_n$ and $M'_n$ are homotopy equivalent to $W_{2n}$.
\end{itemize}
\noindent
Since we are assuming $h(S_1^3(K)) <0$, by \cref{NST} we know that $\{ S_{1/n}^{3}(K) \}$ is a linearly independent set in the homology cobordism group. 
Thus, in particular, $\partial W_{2n}$ and $\partial W_{2n'}$ are not smooth homology cobordant each other if $  n \neq n'$. 
Thus, in particular, $\partial M_n$ and $\partial M_{n'}$ are not diffeomorphic, equivalently not homeomorphic. 
\end{proof}

\subsubsection{Proof of Example~\ref{example exotic}}
We note our concrete examples given in \cref{example exotic} satisfy all assumptions of \cref{main:exotic}: 
\begin{proof}[Proof of \cref{example exotic}]
For odd integer $n$, it is confirmed in \cref{odd integer vno} that 
\[
(K_n  := T_{2,2n+1} \# T_{2,2n+1}, \tau \# \tau)
\]
satisfies 
\[
\underline{V}^{\tau}_0(T_{2,2n+1} \# T_{2,2n+1}) > V_0(T_{2,2n+1} \# T_{2,2n+1}) > 0.
\] 
On the other hand, it is confirmed in \cref{instanton ex} that one has 
\[
h(S^3_1(K_n)) <0. 
\]
This completes the proof. 
\end{proof}

Let us end this section with a few remarks. Firstly, note that in the proof of Theorem~\ref{main:exotic} the only input from the manifold $W_n$ we used was that it was an equivariant, negative-definite cobordism from $(S^{3}_{1/2n}(K), \tau)$ to $(S^{3}_{1/(2n-1)}(K), \tau)$ with intersection form $(-1)$, which acted as $-\mathrm{id}$ on $H_2$ . Hence we immediately obtain the following straightforward generalization of Claim~\ref{claim*}.
\begin{prop}\label{non extension}
Let $n$ be any positive integer and $(K,\tau)$ be a strongly invertible knot in $S^3$ satisfying the following conditions 
\[
\underline{V}^{\tau}_{0}(K) > V_0(K) > 0 \text{ and } h(S^3_1(K))<0.
\]
Let $W_{2n}$ be any $4$-manifold with $\partial W_{2n}:= S^{3}_{1/2n}(K) \# S^{3}_{1/(2n-1)}(K)$ with intersection form $(-1)$. If the diffeomorphism $\tau \# \tau$ extends over $W_{2n}$ acting as $-\mathrm{id}$ on $H_2(W_{2n})$ then the diffeomorphism $\tau \# \mathrm{id}$ does not extend over $W_{2n}$.

\end{prop}

Lastly, for the readers familiar with involutive Heegaard Floer homology, we remark that it is possible to use involutive Heegaard Floer homology to produce examples of exotic manifolds with $b_2=1$. Indeed, in light of the discussion in the proof of Theroem~\ref{main:exotic}, it suffices to demonstrate a tuple $(Y, W, \tau)$, such that $Y$ is an integer homology sphere, bounding a simply-connected $4$-manifold $W$ and $Y$ is equipped with a diffeomorphism $\tau$ so that $\tau$ does not smoothly extend over $W$. It is possible to use the invariants developed in \cite{DHM20} to produce such a tuple $(Y,\tau, W)$. For example, if  the $\tau$ and $\iota \circ \tau$-class of $(Y,\tau)$ is non-trivial for $(Y,\tau)$ that would suffice for this purpose. However, as far as the authors are aware, constructions of such a pair $(Y,\tau)$ will require $W$ to be a boundary-connected sum of a cork (a contractible manifold) and $(\pm 1)$-surgery on a knot. In contrast, the $4$-manifolds $W$ we consider, namely $W_{2n}$ do not admit any such obvious splitting as a boundary-connected sum of two $4$-manifolds (with non-$S^3$ boundary). For example, for the pairs $(Y,\tau)$ in Thereom~\ref{main:exotic}, the lack of equivariant mapping cone formula implies we cannot yet compute their $\iota \circ \tau$-class.

\section{Strong cork detection}\label{strong_cork_section}
In this section, we use the families Seiberg--Witten theory to detect strong corks. We begin by establishing the tools.
\subsection{Strong cork detection tools}

 \begin{proof}[Proof of \cref{strong_cork}]
Let $(Y,\phi)$ and $(X, \s, \Phi)$ be as in the hypothesis. Now towards a contradiction, let us assume that $(Y,\phi)$ is not a strong cork. In particular, there exists $\mathbb{Z}_2$-homology sphere $W_0$, such that $\phi$ extends over $W_0$. Let $\s_0$ be the unique spin-structure on $W_0$, hence it follows that the extension (say) $\tilde{\Phi}$ of $\phi$ to $W_0$ preserves $\s_0$.
We now glue $W_0$ to $X$ along $Y$. Let us denote the resulting $4$-manifold by $W^{\prime}:= X \cup_{Y} W_{0}$. We also concatenate the $\spinc$-structure $\s$ on $X$ with $\s_0$ to obtain a $\spinc$-structure $\s^{\prime}$. Finally, let us puncture $W_0 \subset W^{\prime}$, so that $\tilde{\Phi}$ fixes the puncture. We then apply Theorem~\ref{theo: main theo} to the tuple $(W^{\prime}, \s^{\prime}, \Phi \cup \tilde{\Phi})$, but $\delta (S^3)=0$, so we immediately obtain a contradiction. Hence $(Y,\tau)$ is in fact a strong cork.
\end{proof}

\begin{proof}[Proof of \cref{strong_cork_conjugation}]
The proof is verbatim to the proof of Theorem~\ref{strong_cork}. $(Y,\phi)$ and $(X, \s, \Phi)$ be as in the hypothesis. Now towards a contradiction, let us assume that $(Y,\phi)$ is not a strong cork. As before we construct  $(W^{\prime}, \s^{\prime}, \Phi \cup \tilde{\Phi})$, and now apply Theorem~\ref{thm: charge conj} to it to get a contradiction. Hence $(Y,\tau)$ is in fact a strong cork.

\end{proof} 

\subsection{(Non)-Extendability of diffeomorphims over $b^{+}=1$ bounds}\label{extend_diffeo}
In this subsection, we will show that our formalism sometimes obstructs extensions of certain diffeomorphisms on a $3$-manifold to its bounding $4$-manifold. Notably, these bounds will have $b^{+}=1$. The results established in this subsection will be useful for us to construct examples of strong corks later.

 \begin{thm}
\label{thm: nonext b+1}
Let $n$ be any odd positive integer, and let $\tau$ be the involution on
\begin{enumerate}
\item $-\Sigma(2,3,6n+1)$,

\item $-\Sigma(2,2n+1,4n+3)$.
\end{enumerate}
shown in Figure~\ref{equivariant_diagram}. Then $\tau$ does not extend over \textit{any} spin $4$-manifold $X$ bound with $b^{+}(X)=1$, $b_1(X)=0$ and $\sigma(X)<0$, preserving the orientation of $H^+(X)$.
\end{thm}
\begin{figure}[h!]
\center
\includegraphics[scale=0.4]{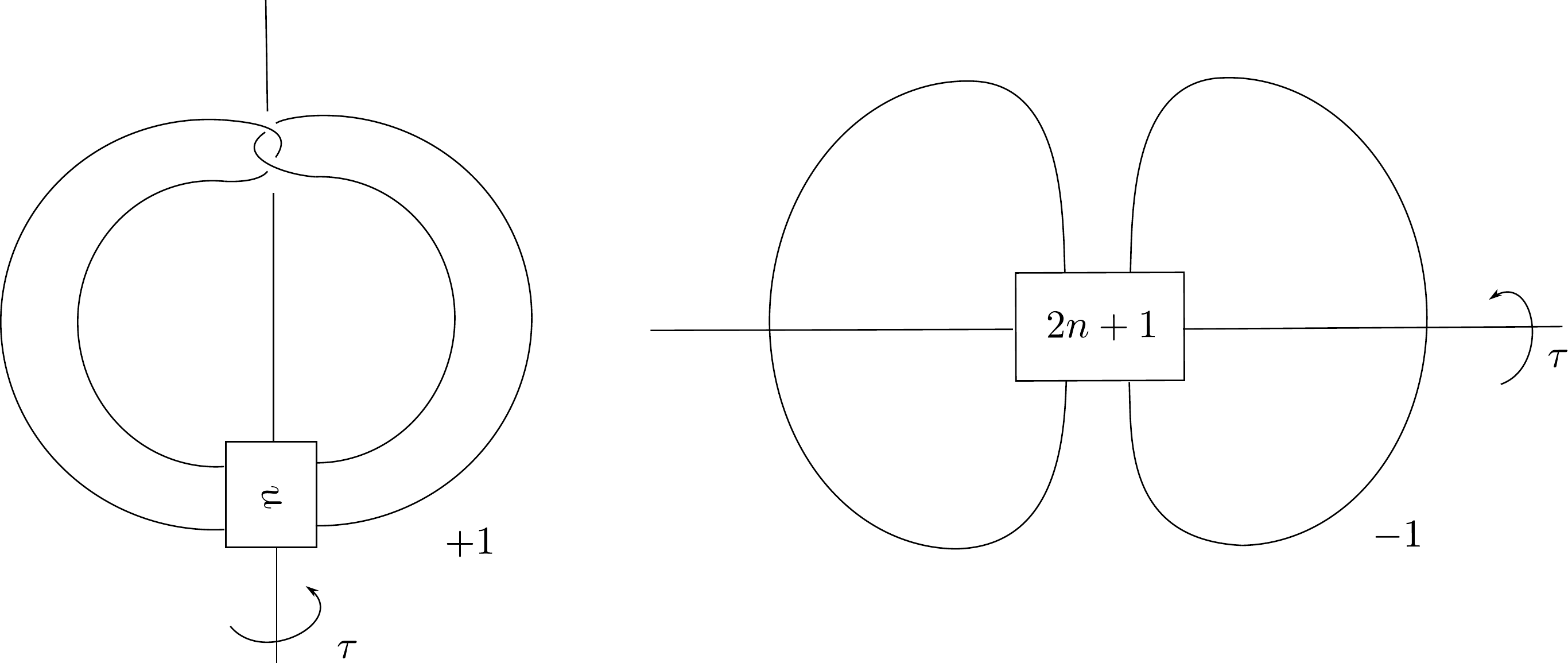}
\caption{Left: $\Sigma(2,3,6n+1)$ represented as surgery on twist-knots, together with $\tau$. Right:  $\Sigma(2,2n+1,4n+3)$ represented as surgery on torus knot, together with $\tau$.}\label{equivariant_diagram}
\end{figure}

\begin{proof}
Firstly, for each member of the two families of Brieskron spheres (with the specified orientation), in  \cite[proof of Lemma 7.7 and 7.8]{DHM20} the authors construct a negative-definite, $\spinc$-conjugating cobordism $W$ with intersection form $(-1,-1, \cdots, -1)$ (odd many) to $-S^3$ that is equivariant with respect to $\tau$ on the Brieskorn sphere and id on $-S^3$. To be specific, these cobordisms are gotten from inverting those constructed in \cite[proof of Lemma 7.7 and 7.8]{DHM20}. We denote the relevant diffeomorphism on $W$ by $f$ and the $\spinc$-structure which is conjugated by it as $\s$. Towards contradiction, suppose that some member among the above two families, say $\Sigma_0$, do admit a spin bound with $b^{+}(X)=1$ and $\sigma(X) < 0$, so that $\tau$ extends over it as $\tilde{\tau}$, preserving the orientation of $H^{+}(X)$. Let us define $Z:= X \cup_{\Sigma_0} W$. By concatenating the $\s$ with the spin structure from $X$, we obtain a $\spinc$-structure $\tilde{\s}$ on $Z$, which is conjugated by $\tilde{\tau} \cup f$. Hence by applying Theorem~\ref{thm: charge conj} to the tuple $(Z,\tilde{\s},\tilde{\tau} \cup f)$, we obtain a contradiction:
\[
\delta(S^3) \geq  \frac{c_{1}(\tilde{\s})^{2} - \sigma(Z)}{8} > 0.
\]

\end{proof}

We now state a similar result for another element of the mapping class group of the families Brieskorn spheres.
\begin{thm}\label{thm: nonext b+1_2}
Let $n$ be any odd positive integer, and let $\sigma$ be the involution on
\begin{enumerate}
\item $-\Sigma(2,3,6n+1)$,

\item $-\Sigma(2,2n+1,4n+3)$.
\end{enumerate}
shown in Figure~\ref{equivariant_diagram_2}. Then $\sigma$ does not extend over \textit{any} spin $4$-manifold $X$ bound with $b^{+}(X)=1$, $b_1(X)=0$ and $\sigma(X)<0$, reversing the orientation of $H^+(X)$.
\end{thm}

\begin{proof}
As before, for each member of the two families of Brieskorn spheres, in \cite[proof of Lemma 7.7 and 7.8]{DHM20} the authors construct a negative-definite, $\spinc$-preserving cobordism $W$ with intersection form $(-1,-1, \cdots, -1)$ (odd many) to $-S^3$ that is equivariant with respect to $\sigma$ on the Brieskorn sphere and id on $-S^3$.
\begin{figure}[h!]
\center
\includegraphics[scale=0.4]{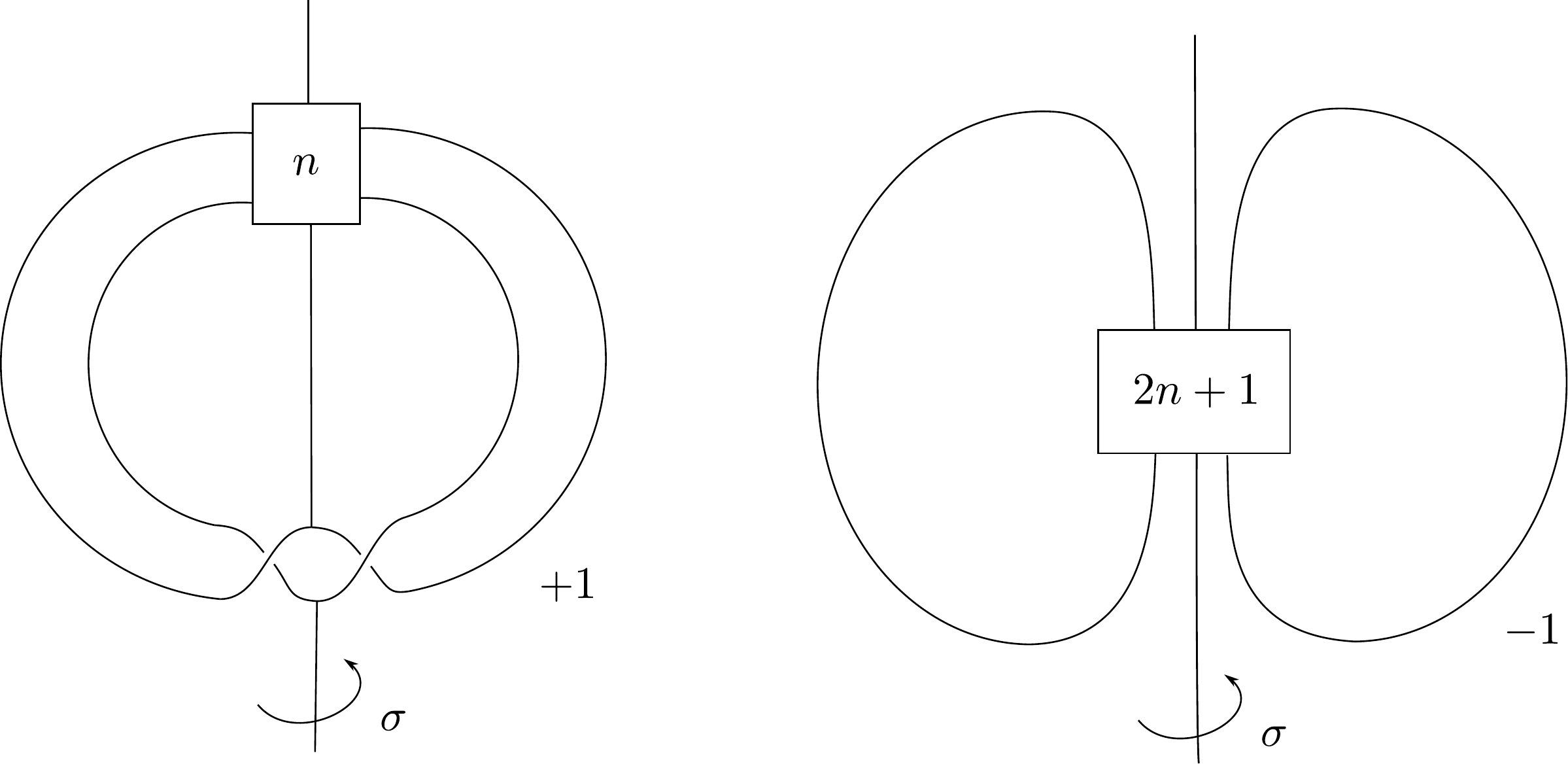}
\caption{Left: $\Sigma(2,3,6n+1)$ with $\sigma$. Right: $\Sigma(2,2n+1,4n+3)$ with $\sigma$.}\label{equivariant_diagram_2}
\end{figure}
Again towards a contradiction, suppose that some member among the above two families, say $\Sigma_0$, do admit a spin bound with $b^{+}(X)=1$ and $\sigma(X) < 0$, so that $\sigma$ extends over it as $\tilde{\sigma}$, preserving the orientation of $H^{+}(X)$. We define the tuple $(Z,\tilde{\s}, \tilde{\sigma} \cup f)$, exactly as in the proof of Theorem~\ref{thm: nonext b+1}. The only difference here is $f$ preserves $\s$, hence $\tilde{\sigma} \cup f$ preserves $\tilde{\s}$. We then apply Theorem~\ref{theo: main theo} to the tuple $(Z,\tilde{\s}, \tilde{\sigma} \cup f)$ to obtain a contradiction:
\[
\delta(S^3) \geq  \frac{c_{1}(\tilde{\s})^{2} - \sigma(Z)}{8} > 0.
\]

\end{proof}

\begin{ex}\label{example_bound}
We observe that both families of Brieskorn spheres from Theorem~\ref{thm: nonext b+1} and \ref{thm: nonext b+1_2} do bound $4$-manifolds with $b^{+}(X)=1$, $b_1(X)$ and $\sigma(X)< 0$. Indeed, Saveliev in \cite{Sa98} showed $-\Sigma(2,q,2qk+1)$ bounds a simply connected, spin $4$-manifold $X_{q,k}$ with intersection form
\[
\left(\dfrac{q+1}{4} \right) (-E_8) \oplus \begin{bmatrix}
0 & 1\\
1 & 0 
\end{bmatrix}. 
\]
Here, $E_8$ is the positive definite non-degenerate bilinear form of rank $8$. In particular, for putting $q=3$ with $k=n$ and $q=2n+1$ with $k=1$ in $-\Sigma(2,q,2qk+1)$ yields corresponding bounds for the Brieskorn spheres $-\Sigma(2,3,6n+1)$ and $-\Sigma(2,2n+1,4n+3)$ respectively.
\end{ex}
\noindent
Of course, we can modify the hypothesis of Theorem~\ref{thm: nonext b+1} and \ref{thm: nonext b+1_2} to produce a more general statement. We include it below for completeness.
\begin{thm}
Let $(Y,\phi)$ be a pair of an integer homology sphere and a diffeomorphism $\phi$ on it. Suppose that there is an equivariant, $\spinc$-preserving, negative-definite cobordism $W$ from $(Y,\phi)$ to $(Y_0,\mathrm{id})$, where $Y_0$ is some integer homology sphere. Then $\phi$ does not extend over \textit{any} spin $4$-manifold $X$ bound with $b^{+}(X)=1$, $b_1(X)=0$ and $\sigma(X)<0$, reversing the orientation of $H^+(X)$. If $W$ was $\spinc$-reversing instead, then $\phi$ does not extend over $X$, preserving the orientation of $H^+(X)$.
\end{thm}

\begin{proof}
Similar to the proof of Theorem~\ref{thm: nonext b+1} and \ref{thm: nonext b+1_2}.
\end{proof}

\subsection{Examples of strong corks}
We now use our obstruction to demonstrate examples of strong corks. Before diving into the proof, we recall that the mapping class group of Brieskorn homology spheres of the form $\Sigma(p,q,r)$ is $\mathbb{Z}/2\mathbb{Z}$ \cite{BO, MS}, unless the sphere is $S^3$ or $\Sigma(2,3,5)$. We will now establish several Lemma for the ease of the argument.
\begin{lem}\label{lemma_1}
Let $Y_n$ be any member of families of Brieskorn spheres from Theorem~\ref{thm: nonext b+1} and equipped with the diffeomorphism $\tau$. Then $\tau$ is the non-trivial element in the mapping class group of $Y_n$.
\end{lem}
\begin{proof}
If $\tau$ is isotopic to identity in $Y_n$, then $\tau$ would admit an extension to the negative-definite, simply connected, spin bound (say $Z_n$) described in Example~\ref{example_bound}, preserving the orientation of $H^+(Z_n)$. This contradicts Theorem~\ref{thm: nonext b+1}.
\end{proof}

\begin{lem}\label{lemma_2}
Let $X_n$ be defined by the plumbing diagram given in Figure~\ref{plumb}. Then $\partial X_n= -\Sigma(2,3,6n+1)$, for odd positive integers $n$. Moreover, the strong involution $\tau$ depicted in Figure~\ref{plumb} extends over $X_n$, reversing the orientation of $H^{+}(X_n)$.
\end{lem}
\begin{proof}
The first assertion follows from \cite{Sa98}, see Example~\ref{example_bound}. The fact that $\tau$ extends over the $2$-handlebody follows by a similar reasoning given in \cite[Section 5]{DHM20}. 
\begin{figure}[h!]
\center
\includegraphics[scale=0.7]{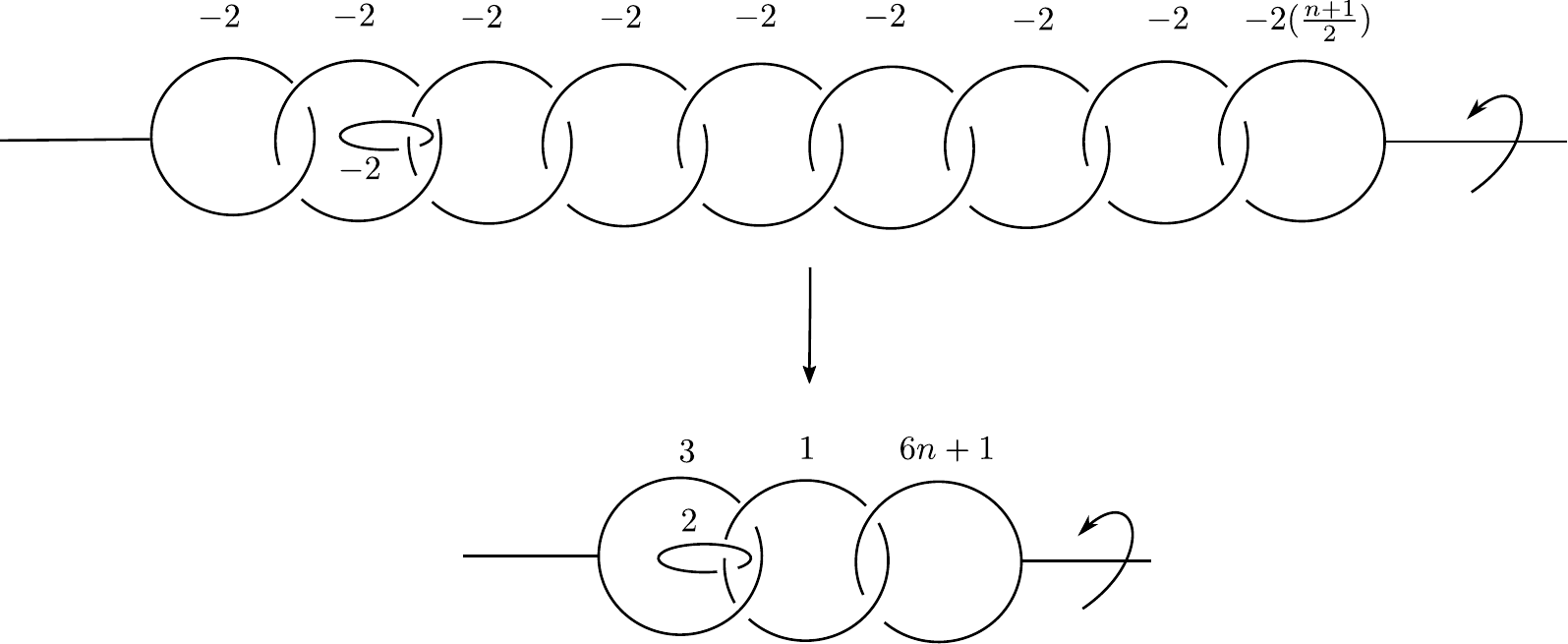}
\caption{Top: The plumbing diagram for $X_n$ with boundary $-\Sigma(2,3,6n+1)$. Bottom: Diagram for $X_n$ after equivariant blowups and blowdowns.}\label{plumb}
\end{figure}
We now claim that $\tau$ on $-\Sigma(2,3,6n+1)$ is the non-trivial mapping class element. This can be seen in many different ways. For example, after a combination of equivariant blow-ups and blow-downs, it can be seen $\tau$ is in fact induced from the complex conjugation action (after identifying the Brieskorn sphere with a link of singularity ) on $-\Sigma(2,3,6n+1)$, which is known to be the non-trivial element of the mapping class group of the Brieskorn spheres by \cite{BO,MS}, see Figure~\ref{plumb}. Now suppose that extension of $\tau$ to $X_n$ preserves the orientation of $H^+(X_n)$. Then we can isotope $\tau$ on $-\Sigma(2,3,6n+1)$, so that it agrees with $\tau$ from Lemma~\ref{lemma_1} and apply Theorem~\ref{thm: nonext b+1} to obtain a contradiction. 
\end{proof}

\begin{lem}\label{lemma_3}Let $\sigma$ be the diffeomorphism on $-\Sigma(2,3,6n+1)$ from Theorem~\ref{thm: nonext b+1_2}, for $n$ odd. Then $\sigma$ is the trivial element in the mapping class group of $-\Sigma(2,3,6n+1)$.
\end{lem}

\begin{proof}Suppose it is the non-trivial mapping class element. Then we can isotope it so that it equals $\tau$ from Lemma~\ref{lemma_2}. Then by Lemma~\ref{lemma_2} $\sigma$ extends over $X_n$, reversing the orientation of $H^+(X)$, contradicting Theorem~\ref{thm: nonext b+1_2}.
\end{proof}

\begin{rem}Both Lemma~\ref{lemma_2} and Lemma~\ref{lemma_3} also holds for the family $-\Sigma(2,2n+1,4n+3)$, with similar proofs. We omit them for brevity.
\end{rem}
\noindent
We are now in place to furnish examples of strong corks.

\begin{proof}[Proof of Theorem~\ref{cork_detection}]
We find it useful to include a sketch of the proof to convey the main idea. Suppose that we want to prove $(Y, \phi)$ is a strong cork.
\begin{figure}[h!]
\center
\includegraphics[scale=0.7]{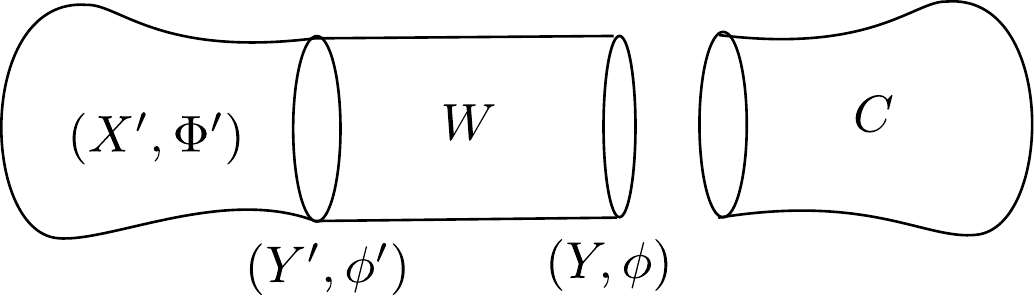}
\caption{Schematic description of the $4$-manifold $X$ used in the proof.}\label{sketch}
\end{figure}
Our strategy for constructing a tuple $(X, \s, \Phi)$ for $(Y, \phi)$ satisfying the hypothesis of Theorem~\ref{strong_cork} or Theorem~\ref{strong_cork_conjugation} will be as follows. We will start by finding a suitable pair $(Y^{\prime}, \phi^{\prime})$ such that $Y^{\prime}$ bounds a $4$-manifold $X^{\prime}$, which satisfies the following conditions:

\begin{enumerate}[label=(\alph*)]

\item There exists a diffeomorphism $\Phi^{\prime}$ on $X^{\prime}$ which restricts to $\phi^{\prime}$ on the boundary,

\item $X^{\prime}$ has a unique spin structure $\s^{\prime}$. 

\item ${\Phi^{\prime}}$ reverses (or preserves) orientation of $H^+(X^{\prime})$,

\item $\frac{c_1(\fraks^{\prime})^2-\sigma(X^{\prime})}{8}= \frac{-\sigma(X^{\prime})}{8} > 0$.

\end{enumerate}
After finding such a pair $(Y^{\prime},\phi^{\prime})$ and a tuple $(X^{\prime}, \s^{\prime}, \Phi^{\prime})$, we will then construct an equivariant, negative-definite, $\spinc$-preserving (or reversing) cobordism $W$ from  $(Y^{\prime},\phi)$ to $(Y,\phi)$ and glue it to $X^{\prime}$ along $Y^{\prime}$. Now let us define $X:= X^{\prime} \cup_{Y^{\prime}} W$. It follows that $X$ will satisfy the hypothesis of Theorem~\ref{strong_cork} (or \ref{strong_cork_conjugation}). A schematic diagram of this is depicted in Figure~\ref{sketch}. This completes our sketch.

Hence it is enough to give an explicit construction for the aforementioned $4$-manifolds $X^{\prime}$ and $W$ for each of the candidate strong corks $(Y,\phi)$. Depending on the case, $X^{\prime}$ will be one of the members of the families of Brieskorn spheres from Theorem~\ref{thm: nonext b+1} or \ref{thm: nonext b+1_2}, while $W$ will be constructed using various equivariant cobordisms used in \cite{DHM20}.

\textit{Proof of $(a)$}: Let $(Y,\phi)$ represent any of surgered $3$-manifold with its specified involution. Let us take $Y^{\prime}$ as $-\Sigma(2,3,7)$ and $X^{\prime}$ as the $-E_8 \oplus H$-bound for $-\Sigma(2,3,7)$ from Example~\ref{example_bound}. Here we think of $-\Sigma(2,3,7)$ as coming from the boundary of the plumbing diagram in Lemma~\ref{lemma_2}. We do not quite specify the involution on $Y^{\prime}$ yet. Now in \cite[Figure 32]{DHM20} the authors constructed an equivariant, negative-definite cobordism $W^{\prime}$ from $(Y,\phi)$ to $\Sigma(2,3,7)$, equipped with some diffeomorphism $\mathfrak{f}$. Now following the algorithm described in \cite{saveliev2002invariants}, we can identify $\mathfrak{f}$ as covering involution of a knot in $S^3$. It can then be checked that when $W^{\prime}$ is $\spinc$-preserving, $\mathfrak{f}$ represent the non-trivial element of the mapping class group of $\Sigma(2,3,7)$. If however, $W^{\prime}$ was $\spinc$-reversing then $\mathfrak{f}$ on $\Sigma(2,3,7)$ is the trivial element of the mapping class group. In both cases, we will revert $W^{\prime}$ to obtain $W$, a $\spinc$-preserving/reversing, equivariant, negative-definite cobordism from $(-\Sigma(2,3,7), \mathfrak{f})$ to $(-Y,\phi)$. For a specific $(Y,\phi)$, let us first assume that $W^{\prime}$ is $\spinc$-preserving. Then we can isotope $\mathfrak{f}$ so that it equals $\tau$ on $\partial X^{\prime}$ coming from Lemma~\ref{lemma_2}. Hence $\mathfrak{f}$ admit an extension to $X^{\prime}$. As explained in the sketch above, applying Theorem~\ref{strong_cork} to $X:= X^{\prime} \cup_{Y^{\prime}} W$, we get that $(-Y,\phi)$ and hence $(Y,\phi)$ is a strong cork.

If $W^{\prime}$ was $\spinc$-conjugating instead, we can isotope $\mathfrak{f}$ on $-\Sigma(2,3,7)$ to identity. In particular, $\mathfrak{f}$ extends to $X^{\prime}$ preserving the orientation of $H^{+}(X^{\prime})$. Again by applying Theorem~\ref{strong_cork_conjugation} to $X:= X^{\prime} \cup_{Y^{\prime}} W$, we get that $(Y,\phi)$ is a strong cork.

\textit{Proof of $(b)$}: In \cite[proof of Theorem 1.12]{DHM20} the authors constructed an equivariant, negative-definite, $\spinc$-preserving cobordism $W^{\prime}$ from $(M_n,\tau)$ to $(\Sigma(2,3,7),\tau)$ (here $\Sigma(2,3,7)$ and $\tau$ refers to the description from Theorem~\ref{thm: nonext b+1}). Hence, we invert $W^{\prime}$ to get $W$ and pick $X^{\prime}$ as the $-E_8 \oplus H$ bound for $-\Sigma(2,3,7)$. A combination of Lemma~\ref{lemma_1} and \ref{lemma_2} again implies that, this choice is sufficient for the conclusion.

\textit{Proof of $(c)$}: In \cite[proof of Theorem 1.13]{DHM20} the authors constructed an equivariant, negative-definite, $\spinc$-preserving cobordism $W$ from $(W_n,\tau)$ to $(\Sigma(2,2n+1,4n+3),\tau)$. Again from the argument similar to that in Lemma~\ref{lemma_1} and \ref{lemma_2} $\tau$ on $-\Sigma(2,2n+1,4n+3)$ extends to $X_n$ reversing the orientation of $H^{+}(X_n)$, for $n$ odd.

\textit{Proof of $(d)$}: In \cite[proof of Theorem 1.10]{DHM20} the authors constructed the following:
\begin{enumerate}

\item For $n$ odd, an equivariant, negative-definite cobordism from $(S^{3}_{+1} (\overline{K}_{-n,n+1}, \tau)$ to $(\Sigma(2,3,6n+1), \sigma)$ which is $\spinc$-conjugating.

\item For $n$ odd, an equivariant, negative-definite cobordism from $(S^{3}_{+1} (\overline{K}_{-n,n+1}, \sigma)$ to $(\Sigma(2,3,6n+1), \tau)$ which is $\spinc$-preserving.

\item For $n$ even, an equivariant, negative-definite cobordism from $(S^{3}_{+1} (\overline{K}_{-n,n+1}, \tau)$ to $(\Sigma(2,3,6(n+1)+1), \tau )$ which is $\spinc$-preserving.

\item For $n$ even, an equivariant, negative-definite, from $(S^{3}_{+1} (\overline{K}_{-n,n+1}, \sigma)$ to $(\Sigma(2,3,6(n+1)+1), \sigma)$ which is $\spinc$-conjugating.

\end{enumerate} 
\noindent
Here $\tau$ and $\sigma$ on $\Sigma(2,3,6n+1)$ are as in Theroem~\ref{thm: nonext b+1} and \ref{thm: nonext b+1_2}. By arguing similarly as above, we obtain that the following are all strong corks.
\[
\begin{cases}
(S^3_{1}(\overline{K}_{-n, n+1}), \tau) \; \text{and} \; (S^3_{1}(\overline{K}_{-n, n+1}), \sigma) &\text{if } n \text{ is odd}\\
(S^3_{1}(K_{-n, n+1}),\tau) \; \text{and} \; (S^3_{1}(K_{-n, n+1}),\sigma) &\text{if } n \text{ is even}.
\end{cases}
\]
The proof now follows from the observation that there is an equivariant, negative-definite cobordism from $(S^{3}_{1/(k+2)}(K),\tau / \sigma)$ to $(S^{3}_{1/k}(K),\tau / \sigma)$ which can be taken to be both $\spinc$-preserving and reversing, see \cite[Theorem 1.6]{DHM20}. By suitably attaching the above cobordism to our previous argument we obtain that for any odd $k$, $1/k$-surgery on the above knots with corresponding involutions are also strong corks.
\end{proof}

\begin{rem}\label{comparison}
Lastly, we end this section with a comparison of our method of detecting strong corks with that from \cite{DHM20}. Both methods rely on the intermediate step of constructing an equivariant, negative-definite cobordism from the cork to a different manifold (often a Seifert-homology sphere). In \cite{DHM20} this step was helpful because the non-triviality of the invariant for the outgoing end (after an explicit computation of the invariant) of the cobordism implied non-triviality of the incoming end. While, in our case, the intermediate step allows us to construct a suitable bound for the cork, over which the cork-twist extends, which in turn obstructs the extension of the cork-twist to any $\mathbb
{Z}_2$-homology spheres. This is again reminiscent of our main slogan. The effectiveness of both methods depends on building candidate outgoing ends for prospective corks. Since for us the outgoing ends only have to satisfy certain topological conditions (as opposed to an explicit computation of the Floer-theoretic invariant, which is significantly harder for non-Seifert homology spheres), sometimes we are able to produce candidate outgoing ends more easily. For example, in Figure~\ref{plumb}, we may replace any number of unknots with any collection of strongly invertible knots, and the resulting $4$-manifold will still be a candidate outgoing manifold $(X^{\prime}, \Phi^{\prime})$, in the language of the proof of Theorem~\ref{cork_detection}. In Figure~\ref{plumb_2} below, we depicted such a case, however we do not have any illuminating example of a cork for which admit an equivariant, negative-definite cobordism to it, when $K$ is not an unknot.
\begin{figure}[h!]
\center
\includegraphics[scale=1.0]{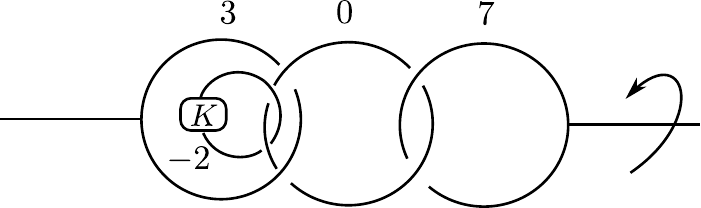}
\caption{A candidate for $(X^{\prime}, \Phi^{\prime})$, here $K$ is any strongly invertible knot. This is obtained by replacing the $(-2)$-framed unknot in the similar position as $K$ is, on the top part of Figure~\ref{plumb}, by $K$, and following the equivariant moves (blow-ups and blow-downs).}\label{plumb_2}
\end{figure}

\end{rem}

\section{Exotic embeddings into small $4$-manifolds}
In this section, we concentrate on results regarding exotic embeddings of 3-manifolds into small 4-manifolds.
\subsection{Proof of \cref{exotic embeddings}}

\cref{exotic embeddings} is derived from the following general result: 

\begin{thm}\label{general embedding}
Let $Y$ be an oriented homology 3-sphere. Suppose that there exist three compact 4-manifolds $X_1$ and $X_2$ with $\partial X_1=Y$ and $\partial X_2= -Y $ satisfying the following properties: 
\begin{itemize}
    \item $b^+(X_1)=0$ and $X_1$ admits a handle decomposition consists of only one 2-handle, 
     \item $b^+(X_2)=1$, $b_1(X_2)=0$, $X_2$ is spin, and $\sigma(X_2)<0$. 
\end{itemize}
Then, there is a pair of exotic embeddings of $Y$ into $ \C P^2 \# -\C P^2$. 
Moreover, these are still exotic after taking the connected sum of any connected smooth 4-manifold, attached outside the images of the embeddings.
\end{thm}
\begin{proof}
Let us first explain how to construct two embeddings.
    Since $X_1$ can be regarded as the trace of a knot surgery, 
    we see that the double of $X_1$ has the form: 
    \[
    D(X_1) \cong \C P^2 \# -\C P^2. 
    \]
     Now one embedding of $Y$ can be obtained as the inclusion $\iota: Y \to D(X_1)$. 
  In order to construct another embedding, we consider the connected sum of complex conjugations on $\mathbb{C}P^2 \# -\mathbb{C}P^2$. This gives us an orientation-preserving diffeomorphism $g : D(X_1) \to D(X_1)$ which acts as $(-1)$ on the second homology. Now the second embedding is obtained as $g \circ \iota: Y \to  D(X_1)$. We claim that there is a homeomorphism $h' : X_1 \to X_1$ such that 
\begin{itemize}
    \item[(i)] $(h')^* = -1 : H^2( X_1 ) \to H^2(X_1)$, 
    \item[(ii)] $h'|_{\partial (X_1) }=\id $.
\end{itemize}
First, by the classification theorem for topological 4-manifolds with boundary \cite{B86,B93}, we have a homeomorphism 
\[
X_2 \cong 
\begin{cases} 
-\mathbb{C}P^2 \# C \text{ if } \mu (Y) =0, \\ 
-\mathbb{C}P^2_{\operatorname{fake}} \# C \text{ if } \mu (Y) =1,
\end{cases}
\]
where $C$ is a topological contractible 4-manifold bounded by $Y$ and $\mathbb{C}P^2_{\operatorname{fake}}$ is the fake $-\mathbb{C}P^2$.
By \cite
[p.371, Theorem 1.5, Addendum]{Fr82}, any automorphism of the intersection form of a simply-connected, closed 4-manifold can be realized by a homeomorphism, so the $(-1)$-multiplication on $H_2(-\C P^2)$ and $H_2(-\C P_{\mathrm{fake}})$ can be realized by a homeomorphism $h_0 : -\C P^2 \to  -\C P^2$ and by a homeomorphism $h_0 : -\C P^2_{\mathrm{fake}} \to  -\C P^2_{\mathrm{fake}}$, respectively (the former one can be realized by just the complex conjugation).
Thus it suffices to prove the following: for a closed, oriented topological 4-manifold $M$ and an orientation-preserving homeomorphism $\varphi : M \to M$, there is a homeomorphism $\psi : M \to M$ that is topologically isotopic to $\varphi$ that pointwise fixes a locally flat 4-disk in $M$.
Indeed, once we prove this, we may isotope $h_0$ so that it has a pointwise fixed 4-disk in $-\C P^2$ or $-\C P^2_{\mathrm{fake}}$, and then extend it by the identity of $C$ to get the desired $h' : X_1 \to X_1$.
We shall prove the fact to get a fixed disk in \cref{lem: topological isotopy fixed disk}.
Now we have seen the claim.

Now, by pasting two copies of $h'$ along the boundary, we obtain a homeomorphism $h : D(X_1)\to D(X_1)$ such that $h^* = -1$. 
By a result of Perron~\cite{P86} and Quinn~\cite{Q86}, we see that $h$ and $g$ are topologically isotopic. In particular, $\iota$ and $g \circ \iota$ are topologically isotopic. 

Now, we see that $\iota$ and $g \circ \iota$ are not smoothly isotopic. If so, by composing isotopy, one can suppose that the diffeomorphism $g$ preserves the components $X_1$ and $-X_1$.  Then, as restriction of $g$, we obtain a diffeomorphism $g' $ on $X_1$
 such that 
 \begin{itemize}
     \item $(g')^* = -1$, 
     \item $g' |_ {\partial X_1=Y}= \id $. 
 \end{itemize}
  We now take a spin$^c$ structure $\mathfrak{s}$ on $X_1 \cup X_2$ corresponding to 
  \[
  ( \operatorname{(a\  generator)}, 0) \in H^2(X_1; \Z) \oplus  H^2(X_2; \Z).
  \]
Then, we have $g^* \frak{s} \cong \overline{\frak{s}}$.  Note that all assumptions of \cref{thm: charge conj} are satisfied. Thus, by the use of \cref{thm: charge conj}, we obtain 
\[
0<\frac{c_1(\mathfrak{s})^2-\sigma(X_1\cup X_2 ) }{8} \leq 0,
\]
which is a contradiction. This completes the proof of the exoticness of $\iota$ and $g \circ \iota$. 

Now, we also prove that these smooth embeddings are still exotic after taking the connected sum with a connected smooth 4-manifold $M$. Since $M$ is connected, there are two cases
\[
D(X_1) \# M = \begin{cases}
(M \# X_1) \cup_{Y} -X_1,\\
 X_1 \cup_{Y} (-X_1\#M). 
\end{cases}
\]
For the second case, our proof works without any change. For the first case, one can just use $-X_1$ with the opposite orientation for our previous discussion since $D(X_1)$ has a symmetry exchanging $X_1$ and $-X_1$. This completes the proof. 
\end{proof}

We now prove \cref{exotic embeddings} by using \cref{general embedding}:
\begin{proof}[Proof of \cref{exotic embeddings}]
We define 
    \[
    Y = \Sigma (2,3,7) = S^3_{+1} (4_1)  = S^3_{-1}(\mathrm{T_{2,3}} )
    \]
    and $X_2$ is a compact spin 4-manifold bounded by $-Y= - \Sigma (2,3,7) $ and whose intersection from is isomorphic to $-E_8 \oplus H$, where $H$ is the intersection form of $S^2\times S^2$.
    For the existence of such a bound, see \cite{Sa98} for example. 
   Then, there is a negative-definite $b^+=0$ and $b^-=1$, 4-manifold bounded by  $\Sigma(2,3,7)$, which is obtained as the trace of the surgery along the right-handed trefoil and we put $X_1$ as it.
   Then, all assumptions of \cref{general embedding} are satisfied. This completes the proof. 
\end{proof}

We give the proof of the following fact used in the proof of \cref{general embedding}:

\begin{lem}
\label{lem: topological isotopy fixed disk}
Let $M$ be a closed, connected oriented topological 4-manifold and $\varphi : M \to M$ be an orientation-preserving homeomorphism.
Then there is a homeomorphism $\psi : M \to M$ that is topologically isotopic to $\varphi$ and that pointwise fixes some locally flat 4-disk $D^4$ embedded in $M$.
\end{lem}

\begin{proof}
The proof is similar to the proof that connected sum is a well-defined operation on topological 4-manifolds (see, such as \cite[Subsection 5.2]{FNOP19}).

Let $D$ be a locally flat 4-disk embedded in $M$. 
Since $M$ is connected,
we may find a path between points in $D$ and in $\varphi(D)$.
Moving $D$ along this path and scaling down $D$ in a chart near $\varphi(D)$, we may find a homeomorphism $f : M \to M$ such that $f$ is topologically isotopic to the identity and $f(D) \subsetneqq \varphi(D)$.
By the 4-dimensional annulus theorem in the topological category by Quinn~\cite{Q82}, $\varphi(D) \setminus f(D)$ is homeomorphic to the annulus $D^4 \times [0,1]$.
Thus we may find a homeomorphism $f' : M \to M$ such that $f'$ is topologically isotopic to the identity and $f(D) = \varphi(D)$.

We claim that any orientation-preserving homeomorphism $g : D^4 \to D^4$ is topologically isotopic to a homeomorphism $g' : D^4 \to D^4$ that admits a pointwise fixed 4-disk inside the interior of $D^4$.
This claim completes the proof of the \lcnamecref{lem: topological isotopy fixed disk}.
Indeed, if we admit the claim, we may take $\psi : M \to M$ to be $f'^{-1} \circ \varphi$.

To show the claim, first note that there is a 4-disk $D'$ inside the interior of $D^4$ such that $g(D')=D'$ and $g|_{\del D'}$ is the identity of $\del D'$.
This follows from the fact that $\pi_0(\Homeo^+(S^3))=1$, where $\Homeo^+$ denotes the orientation-preserving homeomorphism group.
Thus the claim follows once we know that $\pi_0(\Homeo(D^4, \del))=1$, where $\Homeo(D^4, \del)$ denotes the group of homeomorphisms that fix the boundary pointwise.
However, this is the direct consequence of the Alexander trick.
This completes the proof of the claim and the \lcnamecref{lem: topological isotopy fixed disk}.
\end{proof}

We now give an alternative obstruction to construct exotic embeddings using involutive Heegaard Floer theory. This obstruction is posed in terms of local equivalence formulation from the work of Hendricks-Manolescu \cite{HM} and Hendricks-Manolescu and Zemke \cite{HMZ}. In particular, we refer readers to \cite[Section 6.7]{HM} for the definition of $\overline{V}$-invariant.

\begin{thm}\label{inv exotic emb}
Let $K$ be a knot such that:
\[
\overline{V}_0(\overline{K}) < 0.
\]
Then, $S^{3}_{-1}(K)$ admits a pair of exotic embeddings into $\C P^2 \# -\C P^2$.
Moreover, these are still exotic after taking the connected sum of any connected smooth 4-manifold, attached outside the images of the embeddings.
\end{thm}

\begin{proof}
The basic strategy to prove \cref{inv exotic emb} is similar to that of \cref{general embedding}.  We shall do the same discussion given in the proof of \cref{general embedding} replacing $X_1$ with $X$, where $X$ is the negative-definite $4$-manifold obtained by the trace of the surgery $S^{3}_{-1}(K)$. As before, we consider the double $D(X)$ and one embedding is given as the inclusion $i: Y:=S^{3}_{-1}(K) \to D(X)$ and the other embedding is given by $g \circ i: Y \to D(X)$, where $g$ is a diffeomorphism of $D(X) $ given by the connected sum of complex conjugations on each component of $\C P^2 \# - \C P^2 \cong D(X)$ and $i$ is the inclusion. Then, the proof of \cref{general embedding} implies that $i$ and $g \circ i$ are topologically isotopic. Suppose that $i$ and $g \circ i$ are smoothly isotopic. Then, again the proof of \cref{general embedding} implies that there is an equivariant, spin$^c$-reversing cobordism from $(S^3, \id)$ to $(Y, \id)$. Then, functoriality in involutive Heegaard Floer theory (for example, a modification of \cite[Theorem 1.5]{DHM20}) implies 
\[
0 \leq  \underline{d}(Y).
\]
But since $\overline{V}_0(\overline{K}) < 0$, we get $ \underline{d}(Y) < 0$, which is a contradiction.  The proof for stabilization result is the same as that of \cref{general embedding}.

\end{proof}
\noindent
Note that \cref{exotic embeddings} can be also derived from \cref{inv exotic emb} by taking $K=T_{2,3}$. In fact, the condition on $\overline{V}_{0}(\overline{K}) <  0$ holds for any positive torus knot $K$.

\section{Non-smoothable homeomorphisms preserving orientations of $H^+$}
In this section, we provide non-smoothable homeomorphisms using \cref{thm: charge conj}. Our construction {\it does} assume the orientation of $H^+$ is preserved by non-smoothable homeomorphisms. 
\subsection{Proof of \cref{thm: non-realizable homeo}}

\begin{proof}[Proof of \cref{thm: non-realizable homeo}]
The intersection form of $X$ can be decomposed into 
\begin{align}
\label{eq: decomp nonrealizable proof}
-E_8 \oplus H \oplus (-1)^{n-1}.
\end{align}
Let $A \in \operatorname{Aut}(H_2(X;\Z))$ be the automorphism given by the matrix
\[
\id_{-E_8 \oplus H} \oplus (-1)^{n-1}
\]
along the above decomposition.
By the subjectivity of the natural map 
\[
\pi_0(\Homeo(X)) \to \operatorname{Aut}(H_2(X;\Z))
\]
\cite[Theorem 1.5, Addendum]{Fr82}, the automorphism $A$ is realized by a homeomorphism $f$ of $X$.

The homeomorphism $f$ preserves orientation of $H^+(X)$,
since $H^+(X)$ is generated by an element of $H$.
Thus it suffices to prove that $f$ is not realized by a diffeomorphism.
Let $c \in H^2(X;\Z)$ be the characteritic given by 
\[
c = 0 \oplus 0 \oplus (-1)^{n-1} 
\]
along the decomposition \eqref{eq: decomp nonrealizable proof}.
Let $\fraks$ be the spin$^c$ structure on $X$ that corresponds to $c$.
Then we have $f^\ast \fraks \cong \bar{\fraks}$.
Thus, if $f$ was realized by a diffeomorphism, we have a contradiction from \cref{thm: charge conj}.
This completes the proof.
\end{proof}

\section{Relative genus bounds from diffeomorphisms}
\label{branched section}
In this section, we shall make use of the families Fr\o yshov inequality to give lower bounds of genera of surfaces in 4-manifolds. We start by stating some general results, followed by establishing certain bounds for the $\alpha, \beta, \gamma$ and $\delta$ invariants, and the genus bounds.

\subsection{General results} 
For a given knot $K$ in $S^3$(more generally, in an oriented homology 3-sphere),
the double branched covering space $\Sigma(K)$ has a unique spin structure. Associated with this spin structure, we have knot concordance (homology concordance) invariants 
\[
\alpha(K) := \alpha (\Sigma (K)), \beta(K) := \beta (\Sigma (K)), \gamma(K) := \gamma(\Sigma (K)), \delta (K) := \delta (\Sigma (K)). 
\]
Here $\alpha, \beta, \gamma, \delta$ are Fr{\o}yshov-type invariants of spin rational homology 3-spheres introduced by Manolescu~\cite{Ma16}, which are in fact invariant under spin rational homology cobordisms.
If $K$ and $K'$ are smoothly concordant, then $\Sigma(K)$ and $\Sigma(K')$ are spin rational homology cobordant. Thus we have 
\[
\alpha(K) =\alpha(K'), \ \beta(K)= \beta(K') , \  \gamma(K)= \gamma(K') , \ \delta (K) = \delta(K'). 
\]
Thus, these give functions 
\[
\alpha, \beta, \gamma, \delta: \mathcal{C} \to \Z. 
\]
The map $\delta : \mathcal{C} \to \Z$ has been studied by Manolescu and Owens in \cite{MO07}. 

By just combining relative versions of Theorem A, B, C \cite{Fr10, FLin17, KT22} and relative 10/8 inequality \cite{Ma14} to double branched covers, 
we have the following result, which will be used to obtain our genus bounds:
\begin{prop}
\label{non family genus bound}
Let $K$ be a knot in $S^3$. 
Let $X$ be an oriented smooth closed 4-manifold with $b_1(X)=0$ and $S$ be a smoothly and properly embedded surface in $X \setminus \operatorname{Int} D^4$ bounded by $K$ such that $[S] $ is divisible by $2$. 
\begin{itemize}
    \item 
Suppose $ 2 b^+(X) + g(S) -\frac{1}{4}[S]^2 +  \frac{1}{2} \sigma(K) = 0$. Then for any spin$^c$ structure $\frak{s}$ on the double branched covering space $\Sigma(S)$ along $S$, we have 
\begin{align}\label{pt Froyshov ineq}
    \frac{1}{8} (c_1(\frak{s})^2  -2\sigma(X) +\frac{1}{2}[S]^2-  \sigma(K))  \leq \delta  (K).  
\end{align}

\item Suppose $\operatorname{PD}[S] \equiv w_2(X)$ and $H_1(X ;\Z)=0$. Then, we have the following: 

\begin{itemize}
  \item[(I)] Suppose $2b^+(X)+ g(S) -\frac{1}{4}[S]^2 +  \frac{1}{2} \sigma(K) = 0$, then 
    \begin{align}\label{gamma ineq}
         -\frac{1}{8} ( 2\sigma(X) -\frac{1}{2}[S]^2+ \sigma(K))  \leq \gamma  (K). 
    \end{align} 

    \item[(II)] Suppose $  2b^+(X) + g(S) -\frac{1}{4}[S]^2 +  \frac{1}{2} \sigma(K) = 1$, then 
    \[
    -\frac{1}{8} ( 2\sigma(X) -\frac{1}{2}[S]^2+ \sigma(K))  \leq \beta (K).  
    \]

      \item[(III)] Suppose $ 2b^+(X) + g(S) -\frac{1}{4}[S]^2 +  \frac{1}{2} \sigma(K) = 2$, then 
    \[
    -\frac{1}{8} ( 2\sigma(X) -\frac{1}{2}[S]^2+ \sigma(K))  \leq \alpha (K).  
    \]  
  
\item[(IV) ] We have 
\[
 -\frac{1}{4}\sigma (X) + \frac{5}{16 }[S]^2 - \frac{5}{8}\sigma (K)  - \kappa_M (K)  \leq 2 b^+(X) + g(S), 
 \]
 where $\kappa_M (K)$ is Manolescu's kappa invariant \cite{Ma14} for the double branched covering space of $K$ with a unique spin structure. 

\end{itemize}
\end{itemize}

\end{prop}

The inequality (IV) will not be used to prove our general results but, for comparison, we have stated the constraint. 
\begin{proof}[Proof of \cref{non family genus bound}]
We apply Theorem A, B, C, and 10/8 inequality for 4-manifolds with boundary \cite{Ma14, FLin17, KT22} which are obtained as double branched covering spaces along surfaces in punctured 4-manifolds. In the proof, we use \cite[Lemma 4.2]{KMT21} to describe the signature and $b^+$ of these branched covers in terms of surfaces and base 4-manifolds.  The inequality (IV) is already stated in \cite[Theorem 6.1]{KMT21}. 
\end{proof}

In order to state our results, we will use the following subgroup of $\mathcal{C}$:
\begin{defn} We define 
    \[
 \mathcal{C}^{wt}:= \{ [K] \in \mathcal{C} | \text{ There is an oriented $\Z_2$-homology 3-sphere}
 \]
 \[
 \text{ $Y'$ such that }[(Y', \id_{Y'})] = {\bf \Sigma} ([K]) \in  \Theta^{3, \tau} _{\Z_2} \}. 
    \]
   We call the group $\mathcal{C}^{wt}$ the {\it subgroup of weakly trivial concordance classes}. 
\end{defn}
\noindent
Recall that Subsection~\ref{intro_genus} we defined the homomorphism $\bf \Sigma$. Note that we have the inclusions:
\[
\operatorname{Ker} {\bf \Sigma} \subset  \mathcal{C}^{wt} \subset \mathcal{C}. 
\]
\subsubsection{Examples of elements in $\mathcal{C}^{wt}$}
In this section, we provide several concordance classes of examples lying in $\mathcal{C}^{wt}$. 
\begin{ex}
Let $T_{p,q}$ be a torus knot of type $(p,q)$ for a coprime pair of integers. Then the double branched cover of $T_{p,q}$ is a Brieskorn homology 3-sphere $\Sigma(2,p,q) = \{ (x,y,z) \in \C^3 | x^2 + y ^p + z^q =0\} \cap S^5 $. The covering involution is realized as $(-1)$-multiplication on the $x$-coordinate. This can be realized as the restriction of the Seifert $S^1$-action on $\Sigma(2,p,q)$. This implies that this $(-1)$-multiplication is smoothly isotopic to the identity. Thus we have
\[
[T_{p,q}] \in \mathcal{C}^{wt}. 
\]
This actually shows $\Z^\infty \subset \mathcal{C}^{wt}$.
\end{ex}


\begin{ex}\label{example negatively amphi}
A strongly negative amphichiral knot $(K, \sigma)$ is a smooth knot $K\subset S^3$ along with a smooth {\it orientation-reversing involution} $\sigma: S^3 \to S^3$ such that $\sigma(K) = K$ and $\sigma$ has exactly two fixed points on $K$.

It is pointed out in \cite{AMMMS20} that for a given strongly negative amphichiral knot $K'$, 
\[
K := K' \# (-K')^r
\]
gives an element in $\operatorname{Ker} {\bf \Sigma}$ (see the discussion around Question 1.4 in \cite{AMMMS20}). Also, several other concrete examples $\{K_{m,n}\}$ lying in $\operatorname{Ker} {\bf \Sigma}$ are given in \cite{AMMMS20}. In particular, \cite[Theorem 1.2, comments after Question 1.4]{AMMMS20} implies 
\[
\Z_2 ^5 \subset \operatorname{Ker} {\bf \Sigma}. 
\]
\end{ex}
\noindent
At the moment, the authors do not know how large the subgroup $\operatorname{Ker} {\bf \Sigma}$ is. So we pose:
\begin{ques}\label{ques}
  How big is $\operatorname{Ker} {\bf \Sigma}$? 
\end{ques}
\noindent
The knots given in \cite{Li22} might be candidates of infinitely many elements. Since $\mathcal{C}^{wt}$ is a subgroup in the knot concordance group, any linear combination of the above three examples is also an element in $\mathcal{C}^{wt}$. 
\noindent

\subsection{Bounds for Manolescu's $\alpha$, $\beta$, $\gamma$ and Fr\o yshov's $\delta$}
We will now produce various bounds for $\alpha, \beta, \gamma, \delta$ invariants derived from families Seiberg--Witten theory for knots in $\mathcal{C}^{wt}$.
\begin{thm}\label{general genus bounds}
Let $K$ be a knot in $S^3$ with $[K] \in \mathcal{C}^{wt}$. 
Let $X$ be an oriented smooth closed 4-manifold with $b_1(X)=0$ and $S$ be a smoothly and properly embedded surface in $X \setminus \operatorname{Int} D^4$ bounded by $K$ such that $[S] $ is divisible by $2$. 
\begin{itemize}
    \item 
Suppose $ g(S) -\frac{1}{4}[S]^2 +  \frac{1}{2} \sigma(K) = 1$ and $b^+(X)=0$. Then for any spin$^c$ structure $\frak{s}$ on the double branched covering space $\Sigma(S)$ along $S$ satisfying $\tau^* \frak{s} \cong \frak{s}$, we have 

\begin{align}\label{family delta}
    \frac{1}{8} (c_1(\frak{s})^2  -2\sigma(X) +\frac{1}{2}[S]^2-  \sigma(K))  \leq \delta  (K).  
\end{align}
\noindent
Moreover, if $[K] \in \operatorname{Ker} {\bf \Sigma}$, under the assumptions $ g(S) -\frac{1}{4}[S]^2 +\frac{1}{2} \sigma(K) = 1$ and $b^+(X)=0$, then we have 
\[
   c_1(\frak{s})^2  -2\sigma(X) +\frac{1}{2}[S]^2 - \sigma(K) \leq 0.  
    \]

\item Suppose $\operatorname{PD}[S] \equiv w_2(X)$ and $H_1(X ;\Z)=0$. Then, we have the following: 

\begin{itemize}
  \item[(I)] Suppose $ g(S) -\frac{1}{4}[S]^2 +  \frac{1}{2} \sigma(K)  = 1$ and $b^+(X)=0$, then 
    \begin{align}\label{family gamma}
            -\frac{1}{8} ( 2\sigma(X) -\frac{1}{2}[S]^2+ \sigma(K))  \leq \gamma  (K).  
    \end{align}
    \noindent
    Moreover if $[K] \in \operatorname{Ker} {\bf \Sigma}$, under the assumtpions $ g(S) -\frac{1}{4}[S]^2  +\frac{1}{2} \sigma(K)= 1$ and $b^+(X)=0$ we have 
\[
    -2\sigma(X) +\frac{1}{2}[S]^2 + \sigma(K) \geq 0.  
    \]
    \item[(II)] Suppose $  g(S) -\frac{1}{4}[S]^2 +  \frac{1}{2} \sigma(K) = 0$ and $b^+(X)=1$, then 
 \begin{align}\label{family beta}
            -\frac{1}{8} ( 2\sigma(X) -\frac{1}{2}[S]^2+ \sigma(K))  \leq \beta  (K).  
    \end{align}
      Moreover if $[K] \in \operatorname{Ker} {\bf \Sigma}$, under the assumtpions $ g(S) -\frac{1}{4}[S]^2  = 0$, we have 
\[
      -2\sigma(X) +\frac{1}{2}[S]^2 + \sigma(K) \geq 0.  
    \]

      \item[(III)] Suppose $ 2b^+(X) + g(S) -\frac{1}{4}[S]^2 +  \frac{1}{2} \sigma(K) = 3$ and $b^+(X)=2$ or $0$, then 
 \begin{align}\label{family alpha}
            -\frac{1}{8} ( 2\sigma(X) -\frac{1}{2}[S]^2+ \sigma(K))  \leq \alpha  (K).  
    \end{align}
  Moreover if $[K] \in \operatorname{Ker} {\bf \Sigma}$, under the assumtpions $ 2b^+(X)+  g(S) -\frac{1}{4}[S]^2+\frac{1}{2} \sigma(K)  = 3$ and $b^+(X)=0 $ or $2$ we have 
\[
   -2\sigma(X) +\frac{1}{2}[S]^2  + \sigma(K)\geq 0.  
    \]

\end{itemize}
\end{itemize}

\end{thm}

\begin{proof}
Let us write the double branched cover of $S$ as $\Sigma(S)$ with covering involution $\wt{\tau}$. Now, since we are assuming that $\wt{\tau}|_{\partial (\Sigma(S)) = \Sigma(K) }$ is $\Z_2$-equivariantly $\Z_2$-homology cobordant to $(Y', \id)$, we take an equivariant cobordism $(W, f)$ from $(\Sigma(K), \tau)$ to $(Y', \id)$ and glue 
    \[
    (M, g) : = (\Sigma(S) \cup_{\Sigma(K)} W , \wt{\tau} \cup_{\tau} f). 
    \]
    Note that we have 
    \begin{align*}
        b^+(M)&= b^+(\Sigma(S)) =  2b^+(X) + g(S) -\frac{1}{4} [S]^2   + \frac{1}{2} \sigma (K) \\
        \sigma(M)&= \sigma(\Sigma(S))=2 \sigma (X) - \frac{1}{2} [S]^2)+ \sigma (K) 
        \end{align*}
        from \cite[Lemma 4.2]{KMT21}. 

We shall apply \cref{theo: main theo} and \cref{theo: main theo2} to $(M, g)$.
Let us confirm that the assumptions in \cref{theo: main theo} and \cref{theo: main theo2} are satisfied for $(M,g)$. 
First, we focus on the inequality for $\delta$. 
Since $g^* : H^+(M) \to H^+(M)$ is order $2$, we consider the eigenvalue decomposition
\[
H^+(M) = H^+(M)_{+1} \oplus H^+(M)_{-1} 
\]
with respect to $g^\ast$.
Here $H^+(M)_{\pm1}$ are the eigenspaces associated with the eigenvalues $\pm1$ respectively. 
Since we are assuming that $\dim H^+(M)=1$, it is sufficient to show $H^+(M)= H^+(M)_{-1}$. On the other hand, a standard Mayer--Vietoris argument shows $\dim H^+(M)_{+1}= b^+(X)$. Thus, if we assume $b^+(X)=0$, all assumptions of \cref{theo: main theo} are satisfied for a fixed spin$^c$ structure $\frak{s}$ so that $g^*  \frak{s} = \frak{s}$. 
Then from $ g(S) -\frac{1}{4}[S]^2 +  \frac{1}{2} \sigma(K) = 1$, we have 
\[
b^+(M)=   2b^+(X) + g(S) -\frac{1}{4} [S]^2   + \frac{1}{2} \sigma (K) =1. 
\]
Thus, we apply \cref{theo: main theo} to the branched cover $\Sigma (S)$ and obtain 
\[
\frac{c_{1}(\fraks)^{2} - \sigma(\Sigma (S))}{8} \leq \delta(\Sigma(K) ,\frakt)
\]
which is equivalent to \eqref{family delta} from \cite[Lemma 4.2]{KMT21}, where $\frakt$ is the unique spin structure on $\Sigma(K)$. This completes the proof of \eqref{family delta}.

For invariants $\alpha$, $\beta$ and $\gamma$, we need a spin structure on $\Sigma(S)$ which is preserved by the covering involution. 
 The condition $\operatorname{PD}[S] \equiv w_2(X)$ is equivalent to the existence of a spin structure on $\Sigma(S)$. Also, if we suppose $H_1(X ;\Z)=0$, one can check $H^1(\Sigma(S); \Z_2)=0$. Thus, there is a unique spin structure on $\Sigma(S)$. In particular, the isomorphism class of the spin structure on $M$ is preserved by the diffeomorphism $g$.
The proofs of the inequalities for $\alpha, \beta$, and $\gamma$ are similar. Therefore, we only explain the proof for $\beta$. 
Also, again, we consider the eigenvalue decomposition
\[
H^+(M) = H^+(M)_{+1} \oplus H^+(M)_{-1}. 
\]
From assumption , we have $1=b^+(X) = \dim H^+(M)$, thus 
\[
b^+(M)=   2b^+(X) + g(S) -\frac{1}{4} [S]^2   + \frac{1}{2} \sigma (K) =2
\]
In particular, we have $\dim H^+(M)_{-1} =1$.  
Then, we see $g^*$ reverses an orientation of $H^+(M)$. 
Now we apply \cref{theo: main theo2}(II) and obtain 
\[
\frac{ - \sigma(\Sigma (S))}{8} \leq \beta(\Sigma(K) ,\frakt). 
\]
 Combined with \cite[Lemma 4.2]{KMT21}, this completes the proof. 
\end{proof}

\subsection{Genus bounds}

The following theorem is the most general theorem for genus bounds in this paper:
\begin{thm}\label{general genus bound}
    Let $K$ be a knot with $[K] \in \mathcal{C}^{wt}$. Let $X$ be an oriented smooth closed 4-manifold with $b_1(X)=0$ and $S$ be a smoothly and properly embedded surface in $X \setminus \operatorname{Int} D^4$ bounded by $K$ such that $[S] $ is divisible by $2$ and $\operatorname{PD}[S]/2 \equiv w_2(X)$ and $H_1(X ;\Z)=0$. 
    We define 
    \[
    i(K, [S], X): = \# \{ i \in \{ \min \{ \gamma (K), \delta(K)\}, \beta(K), \alpha(K) \} | \text{ $i$ satisfies \eqref{key ineq}}  \}, 
    \]
    where \eqref{key ineq} is given by 
     \begin{align}\label{key ineq}
     [S]^2 - 2 \sigma(K) > 4\sigma(X) + 16 i. 
     \end{align}
     Suppose $  i(K, [S], X)>0$.
    Then we have the following: 
\begin{itemize}
    \item If $b^+(X)=0$, then 
    \[
    g(S) -\frac{1}{4}[S]^2+\frac{1}{2} \sigma(K) \geq  \begin{cases}
         2 \text{ if }  i(K, [S], X) =1,2 \\
         4 \text{ if }  i(K, [S], X) =3.  
    \end{cases} . 
    \]
    \item If $b^+(X)=1$, then 
    \[
    g(S) -\frac{1}{4}[S]^2+\frac{1}{2} \sigma(K) \geq  
         1 \text{ if }  i(K, [S], X) \geq 2  .
    \]
    \item If $b^+(X)=2$, then
      \[
    g(S) -\frac{1}{4}[S]^2+\frac{1}{2} \sigma(K) \geq  0 \text{ if }  i(K, [S], X)=3. 
    \]
\end{itemize}

\end{thm}

\begin{proof}[Proof of \cref{general genus bound}]
In \cite{CN20}, the following estimate is proven: 
 Let $S  \subset  X \setminus \operatorname{Int}D^4$ be a locally flat, properly embedded surface of genus $g$ bounded by a knot $K \subset S^3$. If the homology class $[S ] \in  H_ 2(X;\Z) $ is divisible by $2$, then 
 \[
 \left| \sigma (K) + \sigma (X) -\frac{1}{2} [S]^2 \right| \leq b_2(X) + 2g (S)
 \]
 holds. In particular, 
\begin{align}\label{topological estimate}
b^+(X) + g(S) -\frac{1}{4}[S]^2+\frac{1}{2} \sigma(K) \geq 0 
 \end{align}
holds.  
Let us first suppose that $b^+(X)=0$ and $i(K)\geq 1$. 
Then a straightforward computation shows that \eqref{topological estimate} is sharp if and only if $b^+(\Sigma(S)) =0$.
Then, from \eqref{pt Froyshov ineq} and \eqref{gamma ineq}, we have  
\[
 \frac{1}{8} (c_1(\frak{s})^2  -2\sigma(X) +\frac{1}{2}[S]^2-  \sigma(K))  \leq \delta  (K)  
\]
and
\[
 -\frac{1}{8} ( 2\sigma(X) -\frac{1}{2}[S]^2+ \sigma(K))  \leq \gamma  (K)
\]
which contradict with \eqref{key ineq} under $i(K)\geq 1$.
Thus \eqref{topological estimate} cannot be sharp.
Thus we have 
\begin{align}\label{topological estimate1}
g(S) -\frac{1}{4}[S]^2+\frac{1}{2} \sigma(K) \geq 1. 
 \end{align}
 We again claim \eqref{topological estimate1} cannot be sharp. Suppose that \eqref{topological estimate1} is the equality.  
 Then \eqref{family delta} and \eqref{family gamma} imply
\[
 \frac{1}{8} (c_1(\frak{s})^2  -2\sigma(X) +\frac{1}{2}[S]^2-  \sigma(K))  \leq \delta  (K)  
\]
and 
\[
 -\frac{1}{8} ( 2\sigma(X) -\frac{1}{2}[S]^2+ \sigma(K))  \leq \gamma  (K)
\]
which again contradict with \eqref{key ineq} under $i(K)\geq 1$.
Thus we have 
\begin{align}\label{topological estimate2}
g(S) -\frac{1}{4}[S]^2+\frac{1}{2} \sigma(K) \geq 2. 
 \end{align}
 We again claim \eqref{topological estimate2} cannot be sharp if $i(K)\geq 3$. (Here, note that $i(K)\geq 2$ is {\it not} enough. )  Suppose \eqref{topological estimate2} is an equality. 
 For a technical reason, we add a small 1-handle to $S$ and assume $g(S)  -\frac{1}{4}[S]^2 + \frac{1}{2} \sigma(K) =3$.
 Then one can use \eqref{family alpha} and obtain 
 \[
 -\frac{1}{8} ( 2\sigma(X) -\frac{1}{2}[S]^2+ \sigma(K))  \leq \alpha  (K), 
\]
which contradicts with \eqref{key ineq} under $i(K)\geq 3$.
Thus we have that \eqref{topological estimate2} cannot be sharp. A similar discussion also shows 
\[
g(S) -\frac{1}{4}[S]^2+\frac{1}{2} \sigma(K) \geq  4 
\]
by using \eqref{family gamma} when $ i(K)=3$.
 This completes the proof in the case of $b^+(X)=0$.

Next, we suppose $b^+(X)=1$. Again, we assume the inequality for locally flat surfaces
\begin{align}\label{eee1}
g(S) -\frac{1}{4}[S]^2+\frac{1}{2} \sigma(K) \geq -1 = -b^+(X)
\end{align}
is sharp. Then, one can see $b^+(\Sigma(S)) = 2b^+(X) + g(S) -\frac{1}{4}[S]^2+\frac{1}{2} \sigma(K) =1$. Suppose $i(K)=2$. Again add a small 1-handle to $S$ and assume $b^+(\Sigma(S))=2$
Then one can use \cref{non family genus bound} (II) which contradicts with \eqref{key ineq} under $i(K)\geq 2$. Thus, we have 
\[
g(S) -\frac{1}{4}[S]^2+\frac{1}{2} \sigma(K) \geq 0. 
\]
Then again one can use \cref{non family genus bound} (II) which contradicts with \eqref{key ineq} under $i(K)\geq 2$. 
Again suppose $g(S) -\frac{1}{4}[S]^2+\frac{1}{2} \sigma(K) = 0$. In this case, we have $b^+(\Sigma(S))=2$. Again one can use \cref{non family genus bound} (II) which contradicts with \eqref{key ineq} under $i(K)\geq 2$.

Finally, we focus on the case $b^+(X)=2$. In this case, the topological genus bound is 
\[
 g(S) -\frac{1}{4}[S]^2+\frac{1}{2} \sigma(K) \geq -3.  
\]
Suppose $g(S) -\frac{1}{4}[S]^2+\frac{1}{2} \sigma(K)= -3$. Then, $b^+(\Sigma(S))=1$. Thus, one can use \eqref{non family genus bound} (II) and obtain a contradiction with \eqref{key ineq} under $i(K)\geq 2$. Similarly, one can see 
\[
 g(S) -\frac{1}{4}[S]^2+\frac{1}{2} \sigma(K) \geq -2
\]
by using \eqref{non family genus bound} (III) under $i(K)\geq 3$.

Next, we suppose 
$
 g(S) -\frac{1}{4}[S]^2+\frac{1}{2} \sigma(K) = -1$ 
 In this case, one has $b^+ (\Sigma(S)) = 3$. Thus one can use \eqref{family gamma} to get a contradiction with \eqref{key ineq} under $i(K)\geq 3$.
This completes the proof. 
\end{proof}

We shall also give another genus bounds derived from \cref{thm: charge conj}.
\begin{thm}
    Let $K$ be a knot in $S^3$ with $[K ]\in \operatorname{
Ker} {\bf \Sigma}$. Let $X$ be an oriented smooth closed 4-manifold with $b_1(X)=0$ and $b^+(X)=1$ and $S$ be a smoothly and properly embedded surface in $X \setminus \operatorname{Int} D^4$ bounded by $K$ such that $[S] $ is divisible by $2$. Suppose there is a spin$^c$ structure $\frak{s}$ on the branched cover $\Sigma(S)$ 
satisfying $\tau^* \frak{s} \cong  \overline{\frak{s}}$ and 
    \begin{align}\label{real family Fr ineq}
c_{1}(\fraks)^{2} - 2 \sigma (X) + \frac{1}{2} [S]^2- \sigma (K)> 0.
\end{align}
Then, we have 
\[
 g(S)  \geq  \frac{1}{4} [S]^2  -\frac{1}{2} \sigma (K) . 
\]

\end{thm}
\begin{proof}
As in the proof of \cref{general genus bounds}, we define 
    \[
    (M, g)  = (\Sigma(S) \cup_{\Sigma(K)} W  , \wt{\tau} \cup_{\tau} f)
    \]
 and have 
    \begin{align}
        b^+(M)&= b^+(\Sigma(S)) =  2b^+(X) + g(S) -\frac{1}{4} [S]^2   + \frac{1}{2} \sigma (K) \label{eq: bplus M conj}\\
        \sigma(M)&= \sigma(\Sigma(S))=2 \sigma (X) - \frac{1}{2} [S]^2+ \sigma (K) .\nonumber
        \end{align}
        From the assumption $[K] \in \operatorname{Ker} {\bf \Sigma }$, we can take $W$ so that $\partial W = \Sigma(K)$, so $M$ is a closed 4-manifold. 
Now, we assume $b^+(M)=0$. Then, the usual Fr\o yshov inequality implies 
\[
c_{1}(\fraks)^{2} - 2 \sigma (X) + \frac{1}{2} [S]^2- \sigma (K)\leq 0,
\]
which contradicts with \eqref{pt Froyshov ineq}.

Then, next, we suppose $b^+(\Sigma(S))=1$. Again, since $g^* : H^+(M) \to H^+(M)$ is order $2$, we consider the eigenvalue decomposition
\[
H^+(M) = H^+(M)_{+1} \oplus H^+(M)_{-1}. 
\]
From assumptions, we have
\[
\dim H^+(M)_{+1}  =1 \text{ and } H^+(M)_{-1} = \{0\}
\]
which imply all assumptions of \cref{thm: charge conj}. Thus we have
\[
c_{1}(\fraks)^{2} - 2 \sigma (X) + \frac{1}{2} [S]^2- \sigma (K)\leq 0,
\] 
which is a contradiction. Thus we obtain $b^+(\Sigma (S)) \geq 2$. This combined with \eqref{eq: bplus M conj} completes the proof. 
\end{proof}

\subsubsection{Genus bounds in spin 4-manifolds}

This section focuses on the lower bounds of H-slice genera in $S^4$, $S^2 \times S^2$, and $\#_2 S^2 \times S^2$. Recall that we denote by 
$g_{X,x} (K)$ the relative smooth 4-genus of $K$ in a closed smooth 4-manifold $X$ with the homology class $x \in H_2(X;\Z)$. For $x=0$, we write $g_{X,x} (K)$ by $g_{X} (K)$ as in the introduction.  As a corollary of \cref{general genus bound}, by putting $X$ as $S^4$, $S^2\times S^2$ and $\#_2 S^2\times S^2$, we obtain the following genus bounds. 
\begin{cor}\label{genus spin bound}
    Let $K$ be a knot with $[K] \in \mathcal{C}^{wt}$.  We use
    \[
    i(K) = \# \{ i \in \{ \min \{ \gamma (K), \delta(K)\}, \beta(K), \alpha(K) \} |  \sigma(K) <  -8 i \}.  
    \]

    Then we have the following: 

    \begin{itemize}
    \item If $b^+(X)=0$, then 
    \[
    g(S)  \geq  -\frac{1}{2} \sigma(K) + \begin{cases}
         2 \text{ if }  i(K) =1,2 \\
         4 \text{ if }  i(K) =3.  
    \end{cases}  
    \]
    \item If $b^+(X)=1$, then 
    \[
    g(S) \geq  -\frac{1}{2} \sigma(K)  + 
         1 \text{ if }  i(K) \geq 2  .
    \]
    \item If $b^+(X)=2$, then
      \[
    g(S) \geq  -\frac{1}{2} \sigma(K)  \text{ if }  i(K)=3. 
    \]
\end{itemize}

\end{cor}

\begin{proof}
   We put $S^4 , S^2\times S^2$ and $\#_2S^2\times S^2$ as $X$ and apply \cref{general genus bound}.  
\end{proof}

\begin{ex}\label{computation for torus knots}
We consider torus knots of the type 
$T_{3,q}$. We put lists of computations of $\alpha, \beta, \gamma$, and $\sigma$.
\begin{center}
\begin{tabular}{| c | |  c | c | c | c | c | c | c | c| }
\hline
torus knot & {$\alpha$} & {$\beta$} & {$\gamma$} & {$\sigma$} &$i(K)$  \\
\hline
\hline
$T_{3,12k-5}$  & $1$ & $-1$ & $-1$ & $-8(2k-1)$ & $2(k=1)$ \\
\hline
$T_{3,12k-1}$  & $2$ & $0$ & $0$ & $-16k$ & $2(k=1)$ \\
\hline
$T_{3,12k+1}$  & $0$ & $0$ & $0$ & $-16k$  & $3$ \\
\hline
$T_{3,12k+5}$  & $1$ & $1$ & $1$ & $-8(2k+1)$ & $0 (k=0)$ \\
\hline
\end{tabular}
\end{center}
These computations of $\alpha, \beta$ and $\gamma$ were done in \cite{Ma16} and the computations of the signatures are well-known.  
Then, using \cref{genus spin bound}, we have the following genus bounds: 
\begin{itemize}
    \item $6 \leq g_{S^4} (T_{3,7}) $, $10 \leq g_{S^4} (T_{3,11}) $, $11\leq g_{S^4}(T_{3,13})$. 
    \item $5 \leq g_{S^2\times S^2} (T_{3,7}) $, $9 \leq g_{S^2\times S^2} (T_{3,11}) $, $9\leq g_{S^2\times S^2}(T_{3,13})$. 
    \item  $8\leq g_{\#_2S^2\times S^2}(T_{3,13})$.
\end{itemize}
Note that the Milnor conjecture for knots $T_{3,7}$ and $T_{3,11}$ can be confirmed also from our inequality.
Above bounds are true even for homology $S^4$, $S^2\times S^2$ and $\#_2 S^2 \times S^2$ respectively. Note that we can also change knot concordance classes by taking the connected sum with $K' \# (-K')^r$, where $K'$ is a strongly negative amphichiral knot. The resulting genus bounds are the same as the genus bounds for the above torus knots. The inequality given in \cref{main:torus} is a part of the above inequalities. 
\end{ex}




\subsubsection{Genus bounds in non-spin 4-manifolds}
We also produce genus bounds in non-spin 4-manifolds. We consider $\#_2\C P^2$ first and give a proof of \cref{genus application}. 

\begin{proof}[Proof of \cref{genus application}]
We set $X= \#_2\C P^2$ and $S = (2,6) \in H_2(X;\Z)$.
For a knot $K$ with 
    $
   [K] \in \operatorname{Ker} {\bf \Sigma}  $, one can see 
    \[
    \alpha(K ) = \beta (K) = \gamma (K) = \delta(K) =0
    \]
    since $\Sigma(K)$ is $\Z_2$-homology cobordant to $S^3$.
    We need to check $i(K, [S], X)=3$. It follows from 
    \[
    [S]^2 = 40, \sigma(K)=0, \alpha(K ) = \beta (K) = \gamma (K) = \delta(K) =0 \text{ and } \sigma(X)=2
    \]
    Then the desired inequality follows from \cref{general genus bound}. 
\end{proof}

Let us now concentrate on a non-spin indefinite 4-manifold $\C P^2 \# - \C P^2$.
\begin{cor}\label{non-spin bound}
Let $K$ be a knot with $[K] \in \mathcal{C}^{wt}$.  Let $a$ and $b$ be integers with $a/2 , b/2 \in 2\Z +1$.
We use 
\[
    i(K,a,b): = \# \{ i \in \{ \min \{ \gamma (K), \delta(K)\}, \beta(K), \alpha(K) \} |  a^2 - b^2 - 2 \sigma(K) > 16 i \}.
    \]  
Let $x$ be the homology class $x\in H_2(\C P^2 \# - \C P^2)$ corresponding to $(a [\C P^1], b[\C P^1])$. Then, we have 

 \[
    g_{\C P^2 \# - \C P^2, x}  (K)\geq  1+ \frac{1}{4}(a^2-b^2) - \frac{1}{2} \sigma(K) \text{ if } i(K,a,b) \geq 2.
    \]
    
     \end{cor}

     \begin{ex}
     For simplicity, we assume $a^2=b^2$.
         Then, using \cref{non-spin bound}, we have the following genus bounds: 
\[
5 \leq g_{\C P^2 \# - \C P^2, x} (T_{3,7}) ,\  9 \leq g_{\C P^2 \# - \C P^2, x} (T_{3,11}) 
\]
\[
9\leq g_{\C P^2 \# - \C P^2, x}(T_{3,13}).
\]
\end{ex}

\bibliographystyle{amsalpha}
\bibliography{tex}

\end{document}